\documentclass{article}
 
\usepackage{hyperref}

\usepackage{amssymb,amsthm,graphicx,float,fancyhdr,tikz,dsfont,enumerate,mathtools,booktabs,blkarray,multirow}
\usepackage{caption}
\usepackage{subfig}
\usepackage{amsmath}
\usepackage{mleftright}
\usepackage{pdfsync}
\usepackage[all]{xy}
\usepackage{listings}
\usepackage{color}
\definecolor{mygreen}{RGB}{28,172,0}
\definecolor{mylilas}{RGB}{170,55,241}
\usetikzlibrary{matrix}

\theoremstyle{plain}

\newtheorem{thm}{Theorem}[section]
\newtheorem{lem}[thm]{Lemma}
\newtheorem{prop}[thm]{Proposition}
\newtheorem{cor}[thm]{Corollary}

\theoremstyle{definition}
\newtheorem{defn}[thm]{Definition}
\newtheorem{rmk}[thm]{Remark}

\newtheorem{exam}[thm]{Example}

\numberwithin{thm}{section}

\def\multiset#1#2{\ensuremath{\left(\kern-.3em\left(\genfrac{}{}{0pt}{}{#1}{#2}\right)\kern-.3em\right)}}
\def\supermultiset#1#2{\ensuremath{\:\left(\kern-.5em\left(\genfrac{}{}{0pt}{}{#1}{#2}\right)\kern-.5em\right)\:}}

\def\triset#1#2{\ensuremath{\left \langle \genfrac{}{}{0pt}{}{#1}{#2}\right \rangle}}

\newcommand*{\lcdot}{\raisebox{-0.25ex}{\scalebox{4}{$\cdot$}}}

\DeclareMathOperator\Hom{Hom}
\DeclareMathOperator\Ext{Ext}
\DeclareMathOperator\Tor{Tor}
\DeclareMathOperator\Map{Map}

\usepackage[colorinlistoftodos, textwidth=2.6cm, obeyDraft]{todonotes}
\usepackage[T1]{fontenc}

\begin{document}

\begin{center}
	{\huge \bf The cohomology of free loop spaces of homogeneous spaces} \\
	\vspace{1.5cm}
	{\bf Matthew Burfitt}
	\vspace{2cm}
\end{center}

\begin{abstract}
	The free loops space $\Lambda X$ of a space $X$ has become an important object of study particularly in the case when $X$ is a manifold.
	The study of free loop spaces is motivated in particular by two main examples.
	The first is their relation to geometrically distinct periodic geodesics on a manifold, originally studied by Gromoll and Meyer in $1969$.
	More recently the study of string topology and in particular the Chas-Sullivan loop product has been an active area of research. 
	
	A complete flag manifold is the quotient of a Lie group by its maximal torus and is one of the nicer examples of a homogeneous space.
	Both the cohomology and Chas-Sullivan product structure are understood for spaces $S^n$, $\mathbb{C}P^n$ and most simple Lie groups.
	Hence studying the topology of the free loops space on homogeneous space is a natural next step.
	
	In the thesis we compute the differentials in the integral Leray-Serre spectral sequence associated
	to the free loops space fibrations in the cases of $SU(n+1)/T^n$ and $Sp(n)/T^n$.
	Study in detail the structure of the third page of the spectral sequence in the case of $SU(n)$ and give the module structure of
	$H^*(\Lambda(SU(3)/T^2);\mathbb{Z})$ and $H^*(\Lambda(Sp(2)/T^2);\mathbb{Z})$.
\end{abstract}

\newpage
\tableofcontents

\newpage
\section{Introduction}
		
			The free loop space of a topological space $X$ is defined to be the mapping space $Map(S^1,X)$, the space of all unpointed maps from the circle to $X$.
			This differs from the based loops space $\Omega X=Map_*(S^1,X)$, the space of all pointed maps from the circle to $X$.
			The based loop space functor is an important classical object in algebraic topology and has been well studied.
			The topology of free loop spaces is much less well behaved and is still only well understood in a handful of examples.
			In this thesis we will explore the cohomology of the free loop space of homogeneous spaces.
			In doing so we will uncover some surprising combinatorial connections and we will compute the cohomology algebras for some flag manifolds of low rank Lie groups.
		
			There are two main motivations behind the study of the topology of the free loop space, which we now discuss.
			It is a classical question to ask about the closed geodesics on a closed manifold $M$.
			In particular how many distinct closed geodesics are there on $M$. 
			In general the answer to this question is not fully understood,
			however some problems we can answer by understanding the topology of $M$.
			If $M$ is not simply connected then information on its geodesics can be obtained by studying the conjugacy classes of $\pi_1(M)$.
			If $M$ is simply connected one can consider the free loop space of $M$.
			In particular Gromoll and Meyer prove in \cite{PeriodicGeodesics}, that
			for simply connected closed manifold $M$, if the Betti numbers are unbounded then $M$ has infinitely many distinct closed geodesics. 
			For more information on this subject see for example \cite{ClosedGeodesicServay}.
			
			String topology in its most general sense is the study of algebraic structure on the homology of the space of free loops.
			The area of study began with the unpublished paper \cite{StringTopology} of Chas and Sullivan, released in $1999$.
			In the paper new algebraic structures were presented concerning the homology of free loop spaces of a manifold.
			In particular for a manifold $M$ of dimension $d$, there is an intersection product
			\begin{center}
			$\circ \; \colon H_p(\Lambda M) \otimes H_q(\Lambda M) \to H_{p+q-d}(\Lambda M)$
			\end{center}
			for each $p,q\geq 0$, which has become know as the Chas-Sullivan loop product.
			More recently much work has been done on the subject and connections with many other interesting areas in algebraic topology including
			topological quantum field theory, operads and topological cyclic homology have been established.
			For more information see for example \cite{MR2251006}.

			In $2002$ Cohen, Jones and Yan \cite{loop_homology_spectral_sequence} gave a homotopy theoretic interpretation
			of the Chas-Sullivan product based on earlier work of Cohen and Jones \cite{MR1942249}.
			They then used this description of the Chas-Sullivan product to produce a second quadrant homology spectral sequence, converging
			to the Chas-Sullivan product.
			As a consequence of these results,
			the cohomology of the free loop space would give us the module structure of the homology,
			hence could help us understanding the Chas-Sullivan structure.
		
			A manifold is called homogeneous if it comes equipped with a transitive Lie group action.
			This captures the intuitive idea that a homogeneous space looks the same wherever on it you are.
			Under loose conditions all homogeneous space are the quotient of a Lie groups by a closed subgroup.
			One of the nicest examples of a homogeneous space are the complete flag manifolds, the quotient of a Lie group by its maximal torus.
			When studying Lie groups as a consequence of the classification, it is most important to study the simple Lie groups.
			Hence when studying homogeneous spaces it is most important to study the homogeneous space obtained as the quotient of a simple Lie group.
			
			The Chas-Sullivan products for some low dimensional spheres were computed by Menichi in \cite{StringTop2006}.
			Then in $2002$ the loop product of spheres and projective spaces was given in full by Cohen, Jones and Yan, in \cite{loop_homology_spectral_sequence}.
			More recently Hepworth \cite{HepwothString} gave the rational and $\mathbb{Z}_2$ descriptions of the Chas-Sullivan product on $SO(n)$.
			Following this Kupers \cite{Kupers1010} gave the Chas-Sullivan product for $SU(n)$ and $Sp(n)$ integrally,
			$G_2$ rationally and over $\mathbb{Z}_2$ and $F_4$ rationally.
			Therefore it would be a natural next step to investigate the free loop space of homogeneous spaces.
			
			In this thesis our primary goal is the investigate the cohomology algebra of the free loop space complete flag manifolds.
			Our main tool to achieve this is the cohomology Leray-Serre spectral sequence associated with the free loop fibration of the flag manifolds.
			We give constructions in the cases $SU(n+1)/T^n$ and $Sp(n)/T^n$, though our methods should be applicable more generally.
			In both cases our first main result is the explicit calculation of all non-trivial differentials in the spectral sequences.
			For the Leray-Serre spectral sequence associated with the free loop fibration of $SU(n+1)/T^n$, we investigate in detail the structure of the third page
			of the spectral sequence.
			These more general results allow us to deduce the cohomology of $SU(3)/T^2$ and $Sp(2)/T^2$ as $\mathbb{Z}$-modules,
			which we give in Theorems \ref{thm:L(SU(3)/T^2)} and \ref{thm:FreeLoopSp(2)/T}.
			
			{
			\renewcommand{\thethm}{\ref{thm:L(SU(3)/T^2)}}
			\begin{thm}
				The free loop cohomology of $SU(3)/T^2$ is given by
				\begin{equation*}
					H^*(\Lambda(SU(3)/T^2);\mathbb{Z})=A/I,
				\end{equation*}
				where
				\begin{equation*}
					A=\Lambda_{\mathbb{Z}}(\gamma_i,\; (x_4)_m,\; y_i,\; (x_2)_m(y_1(\gamma_1+\gamma_2)-y_2\gamma_2),\; (x_2)_my_2(\gamma_1^2-\gamma_1\gamma_2),\;
					(x_2)_m\gamma_1^2\gamma_2)
				\end{equation*}
				and
				\begin{align*}
					I=[(x_2)_1^m-m!(x_2)_m,\; (x_4)_1^m-m!(x_4)_m,\; \gamma_1^2+\gamma_2^2+\gamma_1\gamma_2,\;
					\gamma_1^3,\; y_1(2\gamma_1+\gamma_2)-y_2(\gamma_1+2\gamma_2),\; \\
					3(x_2)_m(y_1\gamma_1^2+y_2\gamma_2^2),\; 3(x_2)_my_1y_2(\gamma_1-\gamma_2),\; 3(x_2)_my_1y_2\gamma_1,\; (x_2)_my_1y_2\gamma_1^2\gamma_2]
				\end{align*}
				where $1\leq i,j \leq n$, $m\geq 1$,
				$|\gamma_i|=2$, $|y_i|=1$, $|(x_2)_k|=2k$ and $|(x_4)_k|=4k$.
			\end{thm}
			\addtocounter{thm}{-1}
			}
			
			{
			\renewcommand{\thethm}{\ref{thm:FreeLoopSp(2)/T}}
			\begin{thm}
				The integral cohomology of the free loop space of the complete flag manifold of $Sp(2)$ is given by
				\begin{align*}
					H^*(\Lambda(Sp(2)/T^2);\mathbb{Z})= A/I,
				\end{align*}
				where
				\begin{align*}
					A=\Lambda_{\mathbb{Z}}&((x_6)_b\gamma_i,\;y_1y_2(x_2)_a(x_6)_b,\;(x_6)_by_i,\;(x_2)_m(x_6)_b(y_1\gamma_2+y_2\gamma_1), \\
						&\;(x_2)_m(x_6)_b(y_1\gamma_1-y_2\gamma_2),\;(x_2)_a(x_6)_b\gamma_1^3\gamma_2)
				\end{align*}
				and
				\begin{equation*}
					I=[(x_2)_1^m-m!(x_2)_m,\;(x_6)_1^m-m!(x_6)_m,\;\gamma_1^2+\gamma_2^2,\;\gamma_1^2\gamma_2^2,\;2(y_1\gamma_1+y_2\gamma_2),\;jy_1(x_2)_a\gamma_1^3)]
				\end{equation*}
				for $i,j=1,2$, $m\geq 1$, $a,b\geq 1$ either $j=2$ or $j=4$ and where $|(x_2)_m|=2m$, $|(x_6)_m|=6m$, $|y_i|=1$ and $|\gamma_i|=2$.
			\end{thm}
			\addtocounter{thm}{-1}
			}
			
			In Chapters \ref{sec:topology}, we discuss the relevant algebraic topology that we will use in latter chapters.
			In particular basic techniques for finding homotopy splitting of spaces,
			applying the universal coefficient theorems to deduce the relationship of the module structure between cohomology with integral coefficients
			and cohomology with coefficients over field of zero or prime characteristic
			and set out the essential properties of the cohomology Leray-Serre spectral sequence.
			This is our main tool for investigating the cohomology of free loop space of homogeneous spaces.
			
			In Chapter \ref{sec:SymPoly}, we review the basic theory of symmetric polynomials.
			This is particularly relevant for our work in Chapter \ref{sec:CombiSQ}. 
			In particular we discuss elementary symmetric, complete homogeneous symmetric polynomials and the fundamental theorem of symmetric polynomials.
			
			Chapter \ref{sec:TopLie} is the final background chapter, in which we give an overview of the homology and cohomology of
			of Lie groups, their complete manifolds and based loop spaces.
			We will present the results we intend to use in later chapters but also try to give the picture for all simple Lie groups.
			
			Chapter \ref{sec:CombiSQ} contains our main combinatorial work and is primarily concerned with understanding the structure of the quotient of
			the polynomial algebra by the ideal generated by symmetric polynomials.
			The chapter contains some known and some new results.
			We present a method for finding a simple additive basis of the polynomial symmetric quotient.
			Investigate the degree-wise size of this bases,
			before making a few remarks about the multiplication of basis element. 
			Which we put to use at the end of Chapter \ref{sec:FreeLoopSU(n+1)/Tn}. 
			
			In Chapter \ref{sec:FreeLoopSU(n+1)/Tn}, we investigate the cohomology Leray-Serre spectral sequence associated with the free loop space fibration
			of the complete flag manifold of $SU(n)$.
			First we derive a formula for the differentials in the spectral sequence, then investigating in detail the structure of the the third page.
			Lastly we calculate the module structure of $H^*(\Lambda(SU(3)/T^2);\mathbb{Z})$ by describing the algebra structure of
		  of the $E_\infty$-page of Leray-Serre spectral sequence of the free loop fibration of $\Lambda(SU(3)/T^2)$ in terms of generators and relations.
			
			We start Chapter \ref{sec:FreeLoopSp(n)/Tn} by describing the differentials of
			the Leray-Serre spectral sequence associated with the free loop fibration of $Sp(n)/T^n$.
			As a final result we calculate the module structure of $H^*(\Lambda(Sp(2)/T^2);\mathbb{Z})$ by describing the algebra structure of
		  of the $E_\infty$-page of Leray-Serre spectral sequence of the free loop fibration of $\Lambda(Sp(2)/T^2)$ in terms of generators and relations.

\newpage
\section{Methods in algebraic topology}\label{sec:topology}

	In this chapter we introduce the notions from algebraic topology necessary for obtaining our main results.

	\subsection{Homotopy theory}\label{sec:homotopy}
		
		In this section we give the notions from homotopy theory that are used in the work of this thesis.
		
		\begin{defn}\label{def:fibration}
			A map of spaces $p \colon E \to B$ is called a fibrations if for any other space $W$, homotopy $G \colon I\times W \to B$ and map $h \colon W \to E$
			there exists a homotopy $H \colon I\times W \to E$ such that $H_0=h$. 
			In this case we call the pre-image $F=p^{-1}(*)$, the fiber and usually write the fibration as
			\begin{equation*}
				F\to E \xrightarrow{p} B.
			\end{equation*}
			The map $f \colon X \to Y$ is homotopy fibration if there is a homotopy equivalents to a fibration.
			That is there is a fibration  $p \colon E \to B$ and homotopy equivalences $a$ and $b$ such that the diagram
			\begin{equation*}
						\xymatrix{
						{X}\ar[r]^(0.45){f}\ar[d]^(0.45){b} & {Y}\ar[d]^(0.45){a} \\
						{E}\ar[r]^(0.45){p} & {B} 
						}
			\end{equation*}
			commutes.
			The homotopy fiber of $f$ is defined to be $F=p^{-1}{*}$, where $*$ is the base-point.
		\end{defn}
		
		For the remainder of this section assume all spaces are path connected and have the homotopy type of a CW-complex.
		
		\begin{defn}
			Given a pointed space $X$, define the path space $PX$ to be $Map_*(I,X)$ the space of all paths  in $X$ ending at the base point.
		\end{defn}
		
		The space $PX$ is contractible and is the total space in the path space fibration
		\begin{equation}\label{eq:PathFibe}
			\Omega X \to PX \xrightarrow{p} X,
		\end{equation}
		where $p$ maps each path to its starting point. 
		
		The next two Lemmas give some of the properties of the homotopy fiber, for proofs see \cite[\S 3]{Arkowitz} Propositions $3.3.12$ and $3.5.10$.
		
		\begin{lem}
			Given a homotopy fibration $f\colon X\to Y$, the pullback $I_f$ of $f$ in the pullback diagram below
			has the homotopy type of the homotopy fiber $F$ of $f$.
			That is there is a homotopy equivalence $\alpha$ such that diagram 
			\begin{equation}\label{eq:HomotopyFiber}
						\xymatrix{
						{} 																			& {I_f} \ar[r]^{} \ar[d]^(.5){q}  & {PY} \ar[d]^(.45){p'} \\
							{F} \ar[r] \ar[ru]^(.5){\alpha} & {X} 	 \ar[r]^(.5){f}  					& {Y} }
			\end{equation}
			commutes and where $p\colon PY \to Y$ is the path space fibration. 
		\end{lem}
		
		\begin{lem}\label{lem:If}
			Let $f\colon X \to Y$ be a homotopy fibration and let $F\to E \xrightarrow{p'} B$ be a fibration.
			The sequence of maps
			\begin{equation*}
				\Omega Y \to I_f \xrightarrow{q} X,
			\end{equation*}
			where $q$ is induced by the pullback in (\ref{eq:HomotopyFiber}), is a fibration sequence. 
		\end{lem}
		
		As a consequence of Lemma \ref{lem:If}, we many extend any homotopy fibration sequence $F\to X \to Y$ to a sequence of homotopy fibrations
		\begin{equation*}
			\cdots \to \Omega^2Y \to \Omega F \to \Omega X \to \Omega Y\to F\to X \to Y.
		\end{equation*}
		
		The next two propositions are a common tool used to obtain a splitting of topological spaces.
		
		\begin{prop}\label{prop:FibeTrival}
			Let $F\xrightarrow{i}E\xrightarrow{p}B$ be a fibration sequence such that $p\colon E\to B$ is null-homotopic.
			Then there exists a homotopy section $s\colon E \to F$.
		\end{prop}
		
		\begin{proof}
			Consider diagram (\ref{eq:HomotopyFiber}) in Lemma \ref{lem:If} with $f=p$.
			Since $p\simeq *$, there exists a section $\bar{s}\colon E \to I_p$ and $s=\alpha^{-1}\circ \bar{s}\colon E \to F$ is the required section. 
		\end{proof}

		\begin{prop}\label{prop:FibeSection}
			If $\Omega{B}\to F \xrightarrow{p} E$ is a principle fibration arising from fibration $F\to E\xrightarrow{q} B$ with $B$ simply connected
			and homotopy section $s \colon E \to F$, then 
			\begin{equation*}
				F\simeq \Omega B \times E.
			\end{equation*}
		\end{prop}
		
		\begin{proof}
			Consider the maps of homotopy fibrations
			\begin{equation*}
						\xymatrix{
						{\Omega B}\ar[r]\ar[d]^(0.45){i} & {\Omega B \times E}\ar[r]\ar[d]^(0.45){1\times s} & {E}\ar[d]^(0.45){id} \\
						{\Omega B\times \Omega B}\ar[r]\ar[d]^(0.45){m} & {\Omega B\times F}\ar[r]\ar[d]^(0.45){\phi} & {E}\ar[d]^(0.45){id} \\
						{\Omega B}\ar[r] & {F}\ar[r] & {E,} 
						}
			\end{equation*}
			where $i$ is the inclusion into the first component, $m$ is the loop multiplication map and
			$\phi$ the action of $\Omega B$ on $F$.
			More precisely since $p\colon F \to E$ is a fibration, for any $x\in F$ and $\gamma\colon S^1\to B$ there is a lift $\bar{\gamma} \colon I \to E$
			such that $\bar{\gamma}(0)=x$.
			In which case we may define $\phi \colon \Omega B\times F \to F$ by $\phi(\gamma,x))=\bar{\gamma}(1)$.
			The induced maps in the associated long exact sequences of homotopy groups give us a commutative diagram
			\begin{changemargin}{-15mm}{40mm}
				\begin{equation*}
							\xymatrix{
								{\cdots}\ar[r] & {\pi_{n+1}E}\ar[d]\ar[r]  & {\pi_n\Omega B}\ar[d]\ar[r] & {\pi_n\Omega B\times\pi_n E}\ar[d]\ar[r] & {\pi_nE}\ar[d]\ar[r] & {\pi_{n-1}\Omega B}\ar[d]\ar[r] & {\cdots} \\
								{\cdots}\ar[r] & {\pi_{n+1}E}\ar[d]\ar[r] & {\pi_n\Omega B\times\pi_n\Omega B}\ar[d]\ar[r] & {\pi_n\Omega B\times\pi_nF}\ar[d]\ar[r] & {\pi_nE}\ar[d]\ar[r] & {\pi_{n-1}\Omega B\times\pi_{n-1}\Omega B}\ar[d]\ar[r] & {\cdots} \\
								{\cdots}\ar[r] & {\pi_{n+1}E}\ar[r] & {\pi_n\Omega B}\ar[r] & {\pi_nF}\ar[r] & {\pi_nE}\ar[r] & {\pi_{n-1}\Omega B}\ar[r] & {\cdots}. }
				\end{equation*}
			\end{changemargin}
			Using the five lemma and Whiteheads theorem we obtain the desired result.
		\end{proof}
		
		\begin{defn}
			For a space $X$, define the free loop space $\Lambda X$ to be be the space $\Map(S^1,X)$ of non-pointed maps from the unit circle to $X$.
		\end{defn}
		
		It can be show directly using Definition \ref{def:fibration} that 
		\begin{equation}\label{eq:FreeLoopFibration}
			\Omega X \to \Lambda X \xrightarrow{eval} X 
		\end{equation}
		where $eval$ is the maps sending a loop to the image of its base-point, is a fibration sequence called the free loop fibration of $X$.
		
		There is a canonical section $s \colon X \to \Lambda$ of fibration (\ref{eq:FreeLoopFibration}), given by sending a point to the constant loop at that point.
		However we cannot apply Proposition \ref{prop:FibeSection}
		to obtain a splitting since fibration (\ref{eq:FreeLoopFibration}) need not be a principle fibration.
		
	\subsection{The universal coefficients theorems}\label{UniversalCoefficients}
	
		In this section we discusses the universal coefficient theorems,
		which give the exact relationship between the module structure of the homology and cohomology of a space with respect to different coefficient rings.
		In particular we look at the relationship between cohomology with integral coefficients and cohomology over a finite field of prime characteristic.
		The relationships in the Universal coefficients theorems are given in terms of functors $\Ext$ and $\Tor$,
		for more informational and definition see for example \cite[Chapter 7]{rotman}.
		The next two theorems are known as the universal coefficients theorems,
		for proofs see for example \cite[\S $3.1$ and $3.A$]{hatcher} Theorems $3.2$ and $3A.3$.
		
		\begin{thm}[{\bf Universal coefficients theorem}]\label{thm:uct}
			Given any topological space $X$, an abelian group $G$ and an integer $n \geq 1$,
			there is a split exact sequence of abelian groups
			\begin{center}
					$0 \to \Ext(H_{n-1}(X;\mathbb{Z}),G) \to H^{n}(X;G) \to \Hom(H_{n}(X;\mathbb{Z}),G) \to 0$,
			\end{center}
			which is natural with respect to continuous maps between spaces.
		\end{thm}
		
		\begin{thm}[{\bf Universal coefficients theorem for homology}]\label{thm:ucth}
			Given a topological space $X$, an abelian group $G$ and an integer $n \geq 1$,
			there is an exact sequence of abelian groups
			\begin{center}
					$0 \to H_{n}(X;\mathbb{Z}) \otimes G \xrightarrow{\alpha_{n}} H_{n}(X;G) \to \Tor(H_{n-1}(X;\mathbb{Z}),G) \to 0$,
			\end{center}
			which is natural with respect to continuous maps between spaces.
		\end{thm}
	
		In particular the module structure of the  homology and cohomology with respect to any coefficient ring
		is completely determined by the homology or cohomology over the integers.
		In the case of coefficients over a finite field of prime order or the rationals, we have the following explicit relationship.
		
		\begin{cor}\label{cor:UniversalModp}
			For any topological space $X$ and for any $i \geq 0$, if 
			\begin{equation*}
				H^i(X;\mathbb{Z})\cong \mathbb{Z}^a \oplus \mathbb{Z}_{p_1}^{a_1} \oplus \cdots \oplus \mathbb{Z}_{p_j}^{a_j}
			\end{equation*}
			where $j\geq 0$ , $p_1,\dots,p_j$ are distinct primes and $a,a_1,\dots,a_j$ non-negative integers,
			then for each $1\leq k \leq j$ the cohomology of $X$ with coefficients in $\mathbb{Z}_{p_k}$ is given by
			\begin{align*}
												&H^i(X;\mathbb{Z}_{p_k})\cong \mathbb{Z}_{p_k}^{a+a_k} \\
				\text{and} \;\; &H^{i-1}(X;\mathbb{Z}_{p_k})\cong (H^{i-1}(X;\mathbb{Z})\otimes \mathbb{Z}_{p_k}) \oplus \mathbb{Z}_{p_k}^{a_k}.
			\end{align*}
			For prime $p\neq p_k$ for any $1\leq k \leq j$
			\begin{equation*}
				H^i(X;\mathbb{Z}_p)\cong \mathbb{Z}_p^a
			\end{equation*}
			and
			\begin{equation*}
				H^i(X;\mathbb{Q})\cong \mathbb{Q}^a.
			\end{equation*}
		\end{cor}
	
	\subsection{The Leray-Serre spectral sequence}\label{sec:SpecSeq}
	
		In this section we give the structure of the Leray-Serre spectral sequence for cohomology,
		a powerful tool for studying the cohomology algebra of spaces that sit in a fibrations sequence $F\to E \xrightarrow{p} B$.
	
		Given a commutative ring $R$, a bigraded module $M$ is an $R$-module with an index structure of the form
		\begin{center}
			$M= {\bigoplus}_{i,j \in \mathbb{Z}} M^{i,j}$
		\end{center}
		where each $M^{i,j}$ is an $R$-module. 
		A {bigraded algebra} is a bigraded module with an additional multiplicative structure
		such that if $a \in M^{i,j}$ and $b\in M^{k,l}$ then $ab \in M^{i+k,j+l}$.
		A differential $d$ of {bidegree} $(a,b)$ on a bigraded module $E$ is a collection of maps
		$d=d_{i,j} \colon E^{i,j} \to E^{i+a,j+b}$ such that $dd=0$.
		A {differential bigraded module} is a bigraded module with a differential, often denoted by $(E^{*,*},d)$.

		\begin{defn}\label{def:SS}
			A {spectral sequence} is a sequence of differential bigraded modules $(E_{r}^{*,*},d^{r})_{r\geq 1}$, where
			for each $r\geq 2$, $E_{r+1}^{*,*}$ is obtained from $(E_{r}^{*,*},d^{r})$ by $E_{r+1}^{*,*}=H(E_{r}^{*,*},d^{r})$,
			that is, the homology of the previous differential bigraded modules.
			We shall often refer to $(E_{r}^{*,*},d^{r})$ as the { $r^{th}$ page} of the spectral sequence.
		\end{defn}
		
		There is a standard construction which for each fibration $F \to E \xrightarrow{p} B$ produces a spectral sequence.
		Which are proven in \cite[\S 5]{UGSS} Theorem $5.2$ and Proposition $5.6$ or \cite[\S 1.2]{HSS}.
		A spectral sequence exits for any arbitrary fibration however only under certain conditions are they useful;
		these conditions are specified by the next two theorems on the convergence of a spectral sequence.

		\begin{thm}\label{thm:SS}
			Given a fibration $F \to E \xrightarrow{p} B$ such that $B$ is simply connected,
			there is a  spectral sequence $(E_r^{*,*},d^{r})$ satisfying the following:
			\begin{enumerate}
					\item $E_{r}^{i,j}=0$ for all $r \geq 2$ and $i < 0$ or $ j < 0$, that is, the spectral sequence is only non-zero in the first quadrant.
					\item Each differential $d^{r}$ has bidegree $(r,1-r)$. 
					\item There is an integer $1 \leq e<\infty$ for each $i,j\in \mathbb{Z}$, such that for each $r \geq e$, $d^{r}=0$ and so $E_{r+1}^{i,j}=E_{r}^{i,j}$.
								If $H^*(B)$ or $H^*(F)$ is bounded then such an $e$ exists for all $i,j$ simultaneously, in which case we denote $E_{e}^{*,*}$ by $E_{\infty}^{*,*}$.
					\item There is a filtration by subgroups of $H_{n}(E;R)$,
					$0 \subseteq F^{0}_{n} \subseteq \cdots \subseteq F^{n}_{n} =H_{n}(E;R)$
					such that $E_{\infty}^{p,n-p} \cong F^{p}_{n}/F^{p-1}_{n}$.
			\end{enumerate}
		\end{thm}

		A spectral sequence is said to {\it converge} if it satisfies $3.$ and $4.$ above.
		From now on, we will assume that $R=\mathbb{Z}$ unless otherwise stated.
		The next theorem gives us the Leray-Serre spectral sequence for cohomology.

		\begin{thm}\label{thm:SSS}
		The cohomology spectral sequence $(E_{r}^{*,*},d_{r})$ associated to the fibration $F \to E \xrightarrow{p} B$ where $B$ is simply connected,
		converges to $H^{*}(E)$ as an algebra.
		In addition it satisfies the following properties:
		\begin{enumerate}
			\item $E_{2}^{p,q} \cong H^{p}(B;H^{q}(F))$ for each $p,q \in \mathbb{Z}$.
			\item The product in $E_{2}^{*,*}$ is the maps $H^{p}(B;H^{q}(F)) \times H^{s}(B;H^{t}(F)) \to H^{p+s}(B;H^{q+t}(F))$ for each $p,q,s,t \in \mathbb{Z}$,
			given by $([\sum_{i}{a_i u_i}],[\sum_{j}{b_i v_i}]) \mapsto [\sum_{i,j}{(-1)^{qs}(a_{i}\smallsmile b_{j})(u_{i}\smallsmile v_{j})}]$,
			for cocycles $u_i$,$v_i$ and coefficients $a_i \in H^{q}(F)$, $b_j \in H^{t}(F)$, where $\smallsmile$ is the cup product in cohomology.
			\item All differentials satisfy the Leibniz rule.
		\end{enumerate}
		\end{thm}

		In a cohomology Leray-Serre spectral sequence on page $E^{*,*}_{2}$ the vertical axis is $E_{2}^{0,*}\cong H^{0}(B;H^{*}(F)) \cong H^{*}(F)$,
		so we will identify it with $H^{*}(F)$.
		Similarly the horizontal axis is $E_{2}^{*,0}\cong H^{*}(B;H^{0}(F)) \cong H^{*}(B)$, so is identified with $H^{*}(B)$.
		In particular by the formula given in the second part of Theorem \ref{thm:SSS}, the cup product structure in these axis agrees with multiplication on $E_{2}^{*,*}$.

\newpage
\section{Symmetric polynomials}\label{sec:SymPoly}

	A polynomial in $\mathbb{Z}[\gamma_1,\dots,\gamma_n]$ is called symmetric if it is invariant under permutations of the indices of variables $\gamma_1,\dots,\gamma_n$.
	The study of symmetric polynomials goes back more than three hundred years, originally used in the study of roots of single variable polynomials.
	Today symmetric polynomials have applications in a diverse range of areas of mathematics.
	In the thesis the relevance of the symmetric polynomials is brought by their presence in the cohomology rings of complete flag manifolds,
	in Section \ref{sec:CohomCompFlag}.
	In this chapter we summarise some basic concepts from the theory of symmetric polynomially that will be essential for our later work.
	A compete introduction to the topic can be see in \cite[\S $7$]{ECstanly} or \cite[\S $I$]{Macdonald}.
	
	\subsection{Elementary symmetric polynomials}\label{sec:elementary}
	
		Much of the language used to described symmetric polynomials is the language of partitions.
		So before describing the symmetric polynomials it is first necessary to introduce partitions.
			
		\begin{defn}
			An $n$ partition $\lambda$ is a sequence of non-negative integers $(\lambda_1,\dots,\lambda_k)$, for some integer $k\geq 1$, such that
			\begin{equation*}
				\lambda_1\geq\cdots\geq\lambda_k \;\; \text{and} \;\; \lambda_1+\cdots+\lambda_k=n.
			\end{equation*}
			By convention we consider partition $(\lambda_1,\dots,\lambda_k)$ and $(\lambda_1,\dots,\lambda_k,0,\dots,0)$ to be equal
			and abbreviate an $n$ partition $\lambda$ by $\lambda \vdash n$.
		\end{defn}
		
		The elementary symmetric polynomials are for any given $n$, a given collection of $n$ symmetric polynomials in $n$ variables.
		In the next theorem, we see that the elementary symmetric polynomials form a basis of the symmetric polynomials.
		That is any symmetric polynomials can be expressed as a unique polynomial in elementary symmetric polynomials.	
				
		\begin{defn}\label{defn:ElementarySymmetric}
			For each $n\geq 1$ and $1\leq l \leq n$, define the elementary symmetric polynomials $\sigma_l\in\mathbb{Z}[\gamma_1,\dots,\gamma_n]$ in $n$ variables by
			\begin{equation*}
				\sigma_l=\sum_{1\leq i_1<\cdots<i_l\leq n}{\gamma_{i_1}\cdots\gamma_{i_l}}.
			\end{equation*}
			For an partition $\lambda=(\lambda_1,\dots,\lambda_k)$ denote by $\sigma_\lambda$ the symmetric polynomial $\sigma_{\lambda_1}\cdots\sigma_{\lambda_k}$.
		\end{defn}
		
		\begin{exam}
			When $n=3$
			\begin{align*}
											\sigma_1&=\gamma_1+\gamma_2+\gamma_3,\\
											\sigma_2&=\gamma_1\gamma_2+\gamma_1\gamma_3+\gamma_2\gamma_3\\
				\text{and}\;\;\sigma_3&=\gamma_1\gamma_2\gamma_3.
			\end{align*}
		\end{exam}
		
		The following theorem is sometimes known as the fundamental theorem of symmetric polynomials.
		For a proof see for example \cite[\S $7.4$]{ECstanly}.
		
		\begin{thm}\label{thm:FunThmSym}
			For each $n\geq 1$, the set of $\sigma_\lambda$ where $\lambda$ ranges over all $n$ partitions forms an additive basis of all symmetric functions.
			That is for $1\leq i \leq n$,  the set of $\sigma_i$ form a multiplicative basis of  all symmetric functions.
		\end{thm}

	\subsection{Complete homogeneous symmetric polynomials}\label{sec:homogeneous}
	
		The complete homologous symmetric functions are another collection of $n$ symmetric polynomials in $n$ variables for each $n\geq 1$.
		In a sense which is made explicit in \cite[\S $7.6$]{ECstanly},
		the complete homogeneous symmetric polynomials can be thought of as dual to the elementary symmetric polynomials.
		
		\begin{defn}\label{defn:CompleteHomogeneous}
			For each $n\geq 1$ and $1\leq l \leq n$, define the complete homogeneous symmetric polynomials $h_l\in\mathbb{Z}[\gamma_1,\dots,\gamma_n]$ in $n$ variables by
			\begin{equation*}
				h_l=\sum_{1\leq i_1\leq\cdots\leq i_l\leq n}{\gamma_{i_1}\cdots\gamma_{i_l}}.
			\end{equation*}
			For a partition $\lambda=(\lambda_1,\dots,\lambda_k)$, denote by $h_\lambda$ the symmetric polynomial $h_{\lambda_1}\cdots h_{\lambda_k}$.
		\end{defn}
		
		\begin{exam}
			When $n=3$
			\begin{align*}
											h_1&=\gamma_1+\gamma_2+\gamma_3,\\
											h_2&=\gamma_1^2+\gamma_2^2+\gamma_3^2+\gamma_1\gamma_2+\gamma_1\gamma_3+\gamma_2\gamma_3\\
				\text{and}\;\;h_3&=\gamma_1^3+\gamma_2^3+\gamma_3^3+
																\gamma_1^2\gamma_2+\gamma_1^2\gamma_3+\gamma_2^2\gamma_1+\gamma_2^2\gamma_3+\gamma_3^2\gamma_1+\gamma_3^2\gamma_2+\gamma_1\gamma_2\gamma_3.
			\end{align*}
		\end{exam}
		
		Given an $n\times n$ matrix $M$ with entries in the non-negative integers, denote the row and column sums by
		\begin{align*}
											&row(M)=(r_1,\dots,r_n) \\
			\text{and} \;\; &col(M)=(c_1,\dots,c_n).
		\end{align*}
		For $n$ partitions $\lambda$ and $\mu$ denote by $M_{\lambda\mu}$, the number of $n\times n$ matrices $M$ with 
		\begin{align*}
											&row(M)=\lambda \\
			\text{and} \;\; &col(M)=\mu.
		\end{align*}
		
		The next theorem gives the relationship between the elementary symmetric and complete homogeneous symmetric polynomials.
		For a proof see for example\cite[\S $7.5$]{ECstanly}.
		
		\begin{thm}\label{thm:ElmComp}
			Let $\lambda$ be an $m$ partition. Then for each $n\geq 1$, the elementary symmetric
			and complete homogeneous polynomials in $n$ variables satisfy the following relationship
			\begin{equation*}
				h_\lambda=\sum_{\mu\vdash m}{M_{\lambda\mu}\sigma_{\mu}}.
			\end{equation*}
		\end{thm}
	
		As as consequence of Theorem \ref{thm:ElmComp},
		any polynomial in elementary symmetric polynomials can be replaced with a unique polynomial in complete homogeneous symmetric polynomials.
		Hence Theorem \ref{thm:FunThmSym} could equally well be stated in terms of $h_\lambda$
		rather than $\sigma_\lambda$.
		That is the complete homogeneous symmetric polynomials
		also form a basis of the symmetric polynomials.
		
\newpage	
\section{Topology of Lie groups and homogeneous space}\label{sec:TopLie}

	In this chapter we discuss the cohomology of simple Lie groups and some homogeneous space relevant to our later work.
	In addition we present the homology and cohomology of the based loop spaces of some such spaces. 
	
	\subsection{Lie groups}\label{sec:LieGroups}
	
		A Lie groups is a manifold with a group structure such that the operations of multiplication and inversion are smooth maps of the manifold.
		A compact connected Lie group is called simple if it is non-abelian, simply connected and has no non-trivial connected normal subgroups.
		The classification of simple Lie groups is equivalent to the classification of simple Lie algebras and was first attempted by Killing \cite{Killing1888},
		later improved by Cartan \cite{Cartan1914}, with the modern classification by Dynkin diagrams being completed by Dynkin in $1947$.
		\begin{defn}
			Given a field $K$, a Lie algebra over $K$ is a $K$-vectors space $V$ with a Lie bracket $[,]\colon V\times V \to V$ such that
			\begin{enumerate}
				\item
					$[aX+bY,Z]=a[X,Z]+b[Y,Z] \; \text{and} \; [Z,aX+bY]=a[Z,X]+b[Z,Y]$,
				\item
					$[X,X]=0$,
				\item
					$[X,[Y,Z]]+[Z,[X,Y]]+[Y,[Z,X]]=0$
			\end{enumerate}
			for all $a,b\in K$ and $X,Y,Z\in V$.
		\end{defn}
		For each $n\geq 1$, the classical Lie groups $SO(n), SU(n)$ and $Sp(n)$ are defined by the following sets of matrices, group operation matrix multiplication and
		subspace topology in $\mathbb{R}^{n^2},\mathbb{R}^{2n^2}$ and $\mathbb{R}^{4n^2}$ respectively.
		\begin{center}
		$SO(n)=\{ A \in M_{n}(\mathbb{R}) \;|\; {A}^{\intercal} A = I_{n},\; det(A)=1 \}$,
		\end{center}
		\begin{center}
		$SU(n)=\{ A \in M_{n}(\mathbb{C}) \;|\; {\bar{A}}^{\intercal} A = I_{n},\; det(A)=1 \}$,
		\end{center}
		\begin{center}
		$Sp(n)=\{ A \in M_{n}(\mathbb{H}) \;|\; {\bar{A}}^{\intercal} A = I_{n}, \}$,
		\end{center}
		where $M_{n}(R)$ denotes the set of $n \times n$ matrices over real division algebra $R$.
		The Lie group $\mathrm{Spin}(n)$ is defined to be the universal cover of $SO(n)$

		Let $\mathbb{O}$ denote the octonion real division algebra, the $8$-dimensional vector space with basis $1,e_1,e_2,e_3,e_4,e_5,e_6,e_7$
		and multiplication given in Table \ref{table:oct}.
		Conjugation on $\mathbb{O}$ is defined in the same way as the complex numbers and the quaternions.
		  
		\begin{table}[ht]
			\centering
			\caption{Multiplication in the octonion division algebra}
			\label{table:oct}
			\begin{tabular}{c|llllllll}
				 &$1$  &$e_1$&$e_2$&$e_3$&$e_4$&$e_5$&$e_6$&$e_7$ \\
			\cline{1-9}
			$1$&$-1$ &$e_1$&$e_2$&$e_3$&$e_4$&$e_5$&$e_6$&$e_7$  \\
			$e_1$&$e_1$&$-1$ &$e_3$&$-e_2$&$e_5$&$-e_4$&$-e_7$&$e_6$  \\
			$e_2$&$e_2$&$e_3$&$-1$ &$e_1$&$e_6$&$e_7$&$-e_4$&$-e_5$  \\
			$e_3$&$e_3$&$e_2$&$-e_1$&$-1$ &$e_7$&$-e_6$&$e_5$&$-e_4$  \\
			$e_4$&$e_4$&$-e_5$&$-e_6$&$-e_7$&$-1$ &$e_1$&$e_2$&$e_3$  \\
			$e_5$&$e_5$&$e_4$&$-e_7$&$e_6$&$-e_1$&$-1$ &$-e_3$&$e_2$  \\
			$e_6$&$e_6$&$e_7$&$e_4$&$-e_5$&$-e_2$&$e_3$&$-1$ &$-e_1$  \\
			$e_7$&$e_7$&$-e_6$&$e_5$&$e_4$&$-e_3$&$-e_2$&$e_1$&$-1$ 
			\end{tabular}
		\end{table}
		
		We can define the exceptional Lie group $G_2$ to be the set of automorphism of the octonion $\mathbb{R}$-algebra $\mathbb{O}$.
		That is treating elements of $\mathbb{O}$, as $8$-dimensional column vectors over $\mathbb{R}$,
		\begin{center}
			$G_2=\{ g \in GL(n,\mathbb{R}) \;|\; g(oo')=g(o)g(o') $ for all $ o,o'\in\mathbb{O} \}$.
		\end{center}
		Given an $\mathbb{R}$-algebra $A$ its complexification $A^C$ is defined to be $\{ a+ib \; | \; a,b\in A \}$, such that $i^2=-1$.
		Conjugation is given by $\tau(a+ib)=a-ib$ for each $a+ib\in A^C$. 
		Let $J=\{ X\in M(3,\mathbb{O}) \;|\; \bar{X}^{\intercal}=X \}$ with multiplication $X\circ Y = \frac{1}{2}(XY+YX)$, be the Jordan $\mathbb{R}$-algebra.
		We also define $X\times Y=\frac{1}{2}(2X\circ Y)-tr(X)Y-tr(Y)X+(tr(X)tr(Y)-(X,Y))I_3$, inner product $(X,Y)=tr(X\circ Y)$ and
		Hermitian inner product $\langle X,Y \rangle=(\tau X,Y)$ for all $X,Y\in J$.
		The operations $\circ$, $\times$, $(,)$ and $\langle,\rangle$ are defined in the same way in the complementation $J^C$.
		We define the exceptions Lie groups $F_4$ and $E_6$ by
		\begin{equation*}
			F_4=\{ \alpha \in Iso_{\mathbb{R}}(J) \; | \; \alpha(X\circ Y)=\alpha X \circ \alpha Y \; \text{for all} \; X,Y\in J \},
		\end{equation*}
		\begin{equation*}
			E_6=\{ \alpha \in Iso_{\mathbb{R}}(J^C) \; | \; \alpha(X\times Y)\alpha^{-1}=\alpha X \times \alpha Y,
					\; \langle\alpha X,\alpha Y\rangle =\langle X,Y \rangle \; \text{for all} \; X,Y\in {J^C} \}.
		\end{equation*}
		
		For $A,B\in J^C$, let $\tilde{A}$ in the dual space ${J^C}^*$ be given by $\tilde{A}X=A\circ X$ 
		and let $[,]\colon {J^C}^*\times {J^C}^*\to {J^C}^*$ be $[\tilde{A},\tilde{B}]X=\tilde{A}(\tilde{B}X)-\tilde{B}(\tilde{A}X)$ for all $X\in J^C$.
		Define $\vee \colon J^C\times J^C \to {J^C}^*$ by
		\begin{equation*}
			X\vee Y=[\tilde{X}, \tilde{Y}]+(X\circ Y-\frac{1}{3}(X,Y)I_3)^{\sim}.
		\end{equation*}
		We define a $\mathbb{C}$-algebra $B=J^C\oplus J^C \oplus \mathbb{C} \oplus \mathbb{C}$.
		Given $\phi\in {J^C}^*$, $L,F\in J^C$ and $v\in \mathbb{C}$, let $\Phi(\phi,L,F,v)\colon B \to B$ be give by
		\begin{align*}
			\Phi(\phi,L,F,v)&(X,Y,\xi,\eta)=\\
			&(\phi X-\frac{1}{3}vX+2F\times Y+\eta A,\;2L\times X-\phi Y+\frac{1}{3}vY+\xi F,\;(L,Y)-v\xi,\;(F,X)-v\eta).
		\end{align*}
		Multiplication in $B$ will be given by
		\begin{equation*}
			P\times Q=\Phi(\phi,A,B,v)
		\end{equation*}	
		for
		\begin{align*}
			\phi&=-\frac{1}{2}(X\vee W+Z \vee Y),\\
			A&=-\frac{1}{4}(2Y\times W-\xi Z- \zeta X),\\
			B&=\frac{1}{4}(2X\times Z-\eta W-\omega Y),\\
			C&=\frac{1}{8}((X,Y)-(Z,Y)+\xi\omega-\zeta\eta)\\
		\end{align*}
		for all $P=(X,Y,\xi,\eta),Q=(W,Z,\zeta,\omega)\in B$.
		We define the exceptional Lie group $E_7$
		\begin{equation*}
			E_7=\{ \alpha \in Iso_{\mathbb{C}}(B) \; | \; \det(\alpha P)=\det P, \; \langle\alpha P,\alpha Q\rangle =\langle P,Q \rangle \; \text{for all} \; P,Q\in B \},
		\end{equation*}
		where the Hermitian inner product is defined $\langle P,Q \rangle=\langle X,Z \rangle-\langle Y,W \rangle+\bar{\xi}\zeta-\bar{\eta}\omega$
		for all for all $P=(X,Y,\xi,\eta),Q=(W,Z,\zeta,\omega)\in B$.
		
		Define $\{ , \}\colon B \times B \to B$ by
		\begin{equation*}
			\{ P,Q \} = (X,W)-(Z,Y)+\xi \omega-\zeta \eta 
		\end{equation*}
		for all $P=(X,Y,\xi,\eta),Q=(W,Z,\zeta,\omega)\in B$.
		We will define $\mathbb{C}$-Lie algebra $D = B^*\oplus B \oplus B \oplus \mathbb{C} \oplus \mathbb{C} \oplus \mathbb{C}$.
		With Lie bracket
		\begin{equation*}
			[(\phi_1,P_1,Q_1,r_1,s_1,t_1),(\phi_2,P_2,Q_2,r_2,s_2,t_2)]=(\phi,P,Q,r,s,t)
		\end{equation*}
		where
		\begin{align*}
			\phi&=[\phi_1,\phi_2]+P_1\times Q_2-P_2\times Q_1\\
			P&=\phi_1P_2-\phi_2P_1+r_2P_2-r_2P_1+S_1Q_2-S_2Q_1\\
			Q&=\phi_1Q_2-\phi_2Q_1-r_1Q_2-r_2Q_1+t_1P_2-t_2P_1\\
			r&=\frac{1}{8}(-\{ P_1,Q_2 \}+ \{ P_2 Q_1 \} )+s_1t_2-s_2t_1\\
			s&=\frac{1}{4}\{ P_1,P_2 \}+2r_1s_2-2r_2s_1\\
			t&=-\frac{1}{4}\{ Q_1,Q_2 \}-2r_1t_2+2r_2t_1.
		\end{align*}
		Define involutions $\lambda$, $\lambda'$ and $\tau$ on $D$ by
		\begin{align*}
			\lambda(\phi,P,Q,r,s,t)=(\lambda\phi\lambda,\lambda P,\lambda Q,r,s,t)&,\\
			\lambda'(\phi,P,Q,r,s,t)=(\phi,Q,-P,-r,-t,-s)&,\\
			\tau(\phi,P,Q,r,s,t)=(\tau\phi\tau,\tau P,\tau Q,\tau r,\tau s,\tau t)&
		\end{align*}
		for each $(\phi,P,Q,r,s,t)\in D$ and involution $\lambda$ in $B$ is defined $\lambda(X,Y,\xi,\eta)=(Y,-X,\eta,-\xi)$ for each $(X,Y,\xi,\eta)\in B$.
		Let $\langle , \rangle\colon D\times D \to \mathbb{C}$ be given by
		\begin{equation*}
			(R_1,R_2)=(\phi_1,\phi_2)-\{ Q_1,P_2 \}+ \{ P_1,Q_2 \}-8r_1r_2-4t_1s_2-s_1t_2,
		\end{equation*}
		then
		\begin{equation*}
			\langle R_1,R_2 \rangle=(\tau\lambda'\lambda R_1,R_2)
		\end{equation*}
		for each $D_1=(\phi_1,P_1,Q_1,r_1,s_1,t_1),D_2(\phi_2,P_2,Q_2,r_2,s_2,t_2)=\in D$.
		Define the exceptional complex Lie group $E_8^C$ by
		\begin{equation*}
			E_8^C=\{ \alpha \in Iso_{\mathbb{C}}(D) \; |[\alpha D_1, \alpha D_2 \rangle] = \alpha [D_1,D_2] \; \text{for all} \; D_1,D_2\in D \}.
		\end{equation*}
		We define the exceptional Lie group $E_8$ as a subgroups of $E_8^C$ by
		\begin{equation*}
			E_8=\{ \alpha \in E_8^C \; | \; \langle \alpha X, \alpha Y \rangle = \langle X,Y \rangle \; \text{for all} \; X,Y\in E_8^C \}
		\end{equation*}
		The classifications of Lie groups states that the Lie groups defined above are the only simple Lie groups, see for example \cite[\S 5]{TLI&II} Theorem 6.27.
		
		\begin{thm}
				The only compact connected simple Lie groups are
				\begin{equation*}
					\mathrm{Spin}(m),\; SU(n),\; Sp(n),\; G_2,\; F_4,\; E_6,\; E_7,\; E_8
				\end{equation*}
				for $n\geq 1$ and $m \geq 2$.
		\end{thm}
		
		The next theorem gives a consequence of the classification of Lie groups which is a phrasing that better describes the importance in our situation,
		see \cite{Onishchik}.
		
		\begin{thm}\label{LieClassifiaction}
			Any compact connect Lie group is covered by a product of simple Lie groups and circles.
		\end{thm}

		As a consequence of Theorem \ref{LieClassifiaction}, when studying the topology of Lie groups and other related structures it is important to understand
		the topology of those associated with the classical and exceptional simple Lie groups.

	\subsection{Cohomology of simple Lie groups}\label{sec:SimpleLie}
		
		While the simple Lie groups are some of the most important spaces in topology, their cohomology rings in many cases are far from easily described.
		With coefficients in a field of characteristic $0$, the problem can be approached using methods utilizing de Rham cohomology, see for example \cite{RealCoHomogeneous}
		and these algebras were the first to be found.
		Integrally or over an arbitrary field the problem is more subtle and much work has been done by many mathematician including 
		Borel, Araki, Toda, Kono, Mimura and Shimada so today much is known.

		\begin{defn}
			Given a ring $R$, define the tensor algebra $TV$ over $R$-module $V$ to have module structure
			\begin{equation*}
				TV={\oplus}^{\infty}_{i=0}T^{i}V
			\end{equation*}
			where 
			\begin{equation*}
				T^{i}=\underbrace{V \otimes \dots \otimes V}_\text{$i$}.
			\end{equation*}
			Graded structure on $TV$ is given by $\deg{(v_{1}\otimes \cdots \otimes v_{k})=\sum_{i}{\deg v_{i}}}$ for $v_{i} \in V$ and
			multiplication is given by $v\cdot w=v \otimes w$ for each $v,w \in TV$.
			Define $\Lambda V=TV/I$ where $I$ is the ideal generated by elements of the form 
			\begin{center}
				$v\otimes w-(-1)^{\deg{v}\deg{w}} w \otimes v$ 
			\end{center}
			with $v,w \in TV$.
			Given a set of elements $\{ a_{1},\dots,a_{m} \}$ with given degrees, let $V$ be the free graded $R$-module generated by this set.
			In this case we may denote $TV$ by $T(a_{1},\dots,a_{m})$ and $\Lambda V$ by $\Lambda (a_{1},\dots,a_{m})$.
			In particular if all generators are of odd degree this algebra coincides with that of the exterior algebra.
			If all generators have even degree then $\Lambda V$ is a polynomial algebra.
		\end{defn}

		The integral cohomology of $SU(n)$ and $Sp(n)$ can be determined inductively using the Leray-Serre spectral sequence associated to the fibrations
		\begin{equation}\label{eq:FibeSU}
			SU(n)\to SU(n+1) \to S^{2n+1}
		\end{equation}
		and
		\begin{equation}\label{eq:FibeSp}
			Sp(n)\to Sp(n+1) \to S^{4n+3}.
		\end{equation}
		For the construction of these fibrations see for example \cite[\S 3.4]{Arkowitz}.
		
		\begin{thm}\label{thm:H*(SU(n))}
			For each $n\geq 1$, the cohomology of $SU(n)$ is given by
			\begin{equation*}
				H^*(SU(n);\mathbb{Z})=\Lambda(x_3,x_5,\dots,x_{2n-1}),
			\end{equation*}
			where $|x_i|=i$ for $i=3,5,\dots,2n-1$.
		\end{thm}
		
		\begin{proof}
			We know $SU(2)$ is diffeomorphic to $S^{3}$, hence
			\begin{equation*}
			H^{*}(SU(3)) = \Lambda(x_3)
			\end{equation*}
			where $|x_3|=3$.
			For each $m \geq 2$, $S^{m}$ is simply connected,
			hence $n\geq 1$ the Leray-Serre spectral sequence associated to fibration (\ref{eq:FibeSU}) converges.
			We proceed by induction on $n$.
			
			As shown in Figure \ref{pic:SS1}, on the $E_{2}^{*,*}$ page of the spectral sequence, due to the module structure of $H^{*}(S^{2n-1})$,
			the only non-zero columns are at $0$ and $2n-1$.
			Since differentials have bidegree $(r,1-r)$ all differential on pages other than $E_{2n-1}^{*,*}$ are zero, so
			\begin{equation*}
				E_{2n-1}^{*,*}=E_2^{*,*} \;\; \text{and} \;\; E_{2n}^{*,*}=E_{\infty}^{*,*}.
			\end{equation*}
			Assuming inductively that $H^{*}(SU(n-1)) \cong \Lambda (x_3,x_5, \dots ,x_{2n-3})$ with $deg(x_{i})=i$.
			The only non-zero entries of $E_{2n-1}^{*,*}$ are in
			$E_{2n-1}^{0,*} = {\Lambda}(x_3,x_5, \dots ,x_{2n-3})$ or $E_{2n-1}^{2n-1,*} = H^{*}(SU(n-1))$ as a module.
			The highest degree non-zero $E_{2n-1}^{0,q}$, is when $q=3 \cdot 5 \cdots 2n-1$.
			However as the bidegree of $d_{2n-1}$ is $(2n-1,2-2n)$ and the highest degree generator of $H^{*}(SU(n-1)$ is in dimension $2(n-1)-1$.
			differential $d_{2n-1}$ sends all generator  in column $E_{2n-1}^{0,*}$ to $0$.
			Therefore $d_{2n-1}$ is zero and $E_2^{*,*}=E_{\infty}^{*,*}$.
			
			Each negatively sloped diagonal of $E_{\infty}^{*,*}$ contains only one non-zero element. 
			Such non-zero elements occur only in odd entries of $E_{\infty}^{2(n-1),*}$ except for $E_{\infty}^{2n-1,0}$ lying in the negatively sloped diagonal containing 
			$E_{\infty}^{0,2(n-1)}$, which is zero since the first generator of $H^{*}(SU(n-1))$ occurs in degree $3$.
			Therefore there are no extension problems and the module structure of $H^{*}(SU(n))$ is clear.
			The multiplication in $H^{*}(SU(n))$ is freely generated
			with one additional algebra generator then $H^{*}(SU(n-1))$, which comes from $E_{\infty}^{2n-1,0}$ and hence this has degree $2n-1$ as required.
			
			\begin{center}
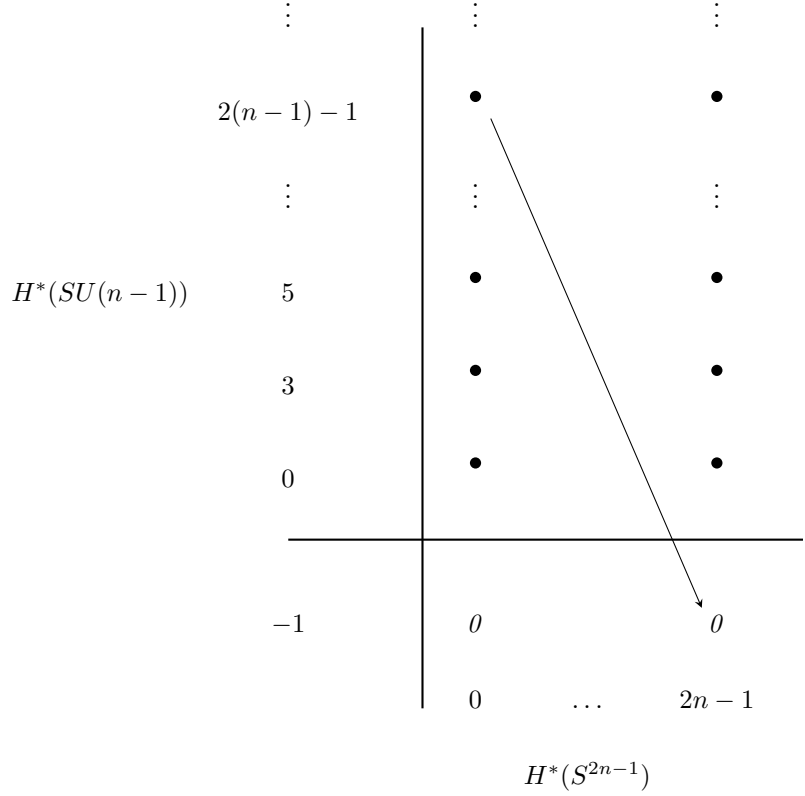

			\begin{tikzpicture}
				\matrix (m) [matrix of math nodes,
					nodes in empty cells,nodes={minimum width=5ex,
					minimum height=5ex,outer sep=-5pt},
					column sep=1ex,row sep=1ex]{
												 &	 \vdots		&&	\vdots	&									&	 \vdots	 & \\
												 &	2(n-1)-1  &&  \lcdot  &    					    &  \lcdot  & \\
												 &	 \vdots   &&  \vdots	&   				  		&  \vdots  & \\
					H^{*}(SU(n-1)) &			5     &&  \lcdot  &    					    &  \lcdot  & \\
												 &			3     &&  \lcdot  &   					    &  \lcdot  & \\
												 &			0     &&  \lcdot  & 								&  \lcdot  & \\
												 &						&&					&									&					 & \\	
												 &		 -1     &&  {\it0}  &       					&  {\it0}  & \\						
												 &\quad\strut &&     0    & 		 \dots 			&    2n-1  & \\
												 &						&&					& H^{*}(S^{2n-1})	&					 & \strut \\};
				\draw[-stealth] (m-2-4.south east) -- (m-8-6.north west);
			\draw[thick] (m-1-3.east) -- (m-9-3.east) ;
			\draw[thick] (m-7-2.north) -- (m-7-7.north) ;
			\end{tikzpicture}
			\captionof{figure}{Leray-Serre spectral sequence associated to fibration (\ref{eq:FibeSU})}
			\label{pic:SS1}
			\end{center}
		\end{proof}
		
		\begin{thm}\label{thm:H*(Sp(n))}
			For each $n\geq 1$, the cohomology of $Sp(n)$ is given by
			\begin{equation*}
				H^*(Sp(n);\mathbb{Z})=\Lambda(x_3,x_7,\dots,x_{4n-1}),
			\end{equation*}
			where $|x_i|=i$ for $i=3,7,\dots,4n-1$.
		\end{thm}
		
		\begin{proof}
			Using Fibration (\ref{eq:FibeSp}) and the fact that $Sp(1)$ is diffeomorphic to $S^3$,
			the proof is the same as that of Theorem \ref{thm:H*(SU(n))} with the exception that degree of the spheres increased. 
		\end{proof}
		
		With coefficients over a field of characteristic $0$ or $\mathbb{Z}_2$
		the cohomology has an easily presented form.
		The integral cohomology of $SO(n)$ contains $2$-torsion
		and cannot be straightforwardly deduced from the characteristic $0$ and $\mathbb{Z}_2$ cases in anything but low dimensions.
		However the integral cohomology for any $n\geq 1$ has been described in \cite{PITTIE}. 
		The rational and $\mathbb{Z}_2$ cohomology are as follows and can be found in \cite[\S 3.6,7.5]{TLI&II} Corollary $1.14$ and Theorem $1.18$
		or in \cite{CohomologyLieGroup}.
		
		\begin{thm}\label{thm:H*(SO(2n+1);QZ2)}
			For each $n\geq 1$, the cohomology of $SO(2n+1)$ with rational coefficients is given by
			\begin{equation*}
				H^*(SO(2n+1);\mathbb{Q})=\Lambda[x_3,x_7,\dots,x_{4n-1}],
			\end{equation*}
			where $|x_i|=i$ for $i=3,7,\dots,4n-1$.
			With $\mathbb{Z}_2$ coefficients the cohomology of $SO(2n+1)$ is given by
			\begin{equation*}
				H^*(SO(2n+1);\mathbb{Z}_2)=\frac{\mathbb{Z}_2[x_1,x_3,\dots,x_{2n-1}]}{[x_i^{a_i}]},
			\end{equation*}
			where $|x_i|=i$ and $a_i$ is the smallest power of two such that $ia_i\geq 2n-1$ for $i=3,5,\dots,4n-1$.
		\end{thm}
		
		\begin{thm}\label{thm:H*(SO(2n);QZ2)}
			For each $n\geq 1$, the cohomology of $SO(2n)$ with rational coefficients is given by
			\begin{equation*}
				H^*(SO(2n);\mathbb{Q})=\Lambda[x_3,x_7,\dots,x_{4n-3},x_{2n-1}],
			\end{equation*}
			where $|x_i|=i$ for $i=3,7,\dots,4n-1,2n-1$.
			With $\mathbb{Z}_2$ coefficients the cohomology of $SO(2n)$ is given by
			\begin{equation*}
				H^*(SO(2n);\mathbb{Z}_2)=\frac{\mathbb{Z}_2[x_1,x_3,\dots,x_{2n-1}]}{[x_i^{a_i}]},
			\end{equation*}
			where $|x_i|=i$ and $a_i$ is the smallest power of two such that $ia_i\geq 2n-1$ for $i=3,5,\dots,2n-3$.
		\end{thm}
		
		The integral cohomology of $G_2$ also  contains $2$-torsion, a proof of the following theorem can be found in \cite{CohomologyLieGroup} Theorem $2.14$.
		
		\begin{thm}\label{thm:H*(G2)}
			The cohomology of $G_2$ is given by
			\begin{equation*}
				H^*(G_2;\mathbb{Z})=\frac{\mathbb{Z}[x_3,x_{11}]}{[x_3^4,x_{11}^2,x_3^2x_{11},2x_3^2]},
			\end{equation*}
			where $|x_3|=3$ and $|x_{11}|=11$ .
		\end{thm}

		Much is also known about the cohomology of the other exceptional Lie groups.
		In \cite[\S 7]{TLI&II} it is shown the groups $F_4$, $E_6$ and $E_7$ have $2$ and $3$-torsion, while $E_8$ has $2$, $3$ and $5$-torsion. 
		The cohomology algebras over field of these theses characteristics can also be found in \cite[\S 7]{TLI&II}.

	\subsection{Cohomology of complete flag manifolds}\label{sec:CohomCompFlag}

		A manifold $M$ is called a homogeneous space if it can be equipped with a transitive $G$ action for some Lie groups $G$.
		In this case we have $M \cong G/H$ for some Lie subgroup $H$ of $G$ isomorphic to the orbit of a point in $M$.
		A Lie subgroup $T$ of Lie group $G$ isomorphic to a torus is called maximal
		if any Lie subgroup also isomorphic to a torus containing $T$ coincidences with $T$.
		The next proposition is straightforward to show, see for example \cite[\S $5.3$]{TLI&II} Theorem $3.15$.
		
		\begin{prop}\label{prop:tori}
				All maximal tori in $G$ are conjugate and the conjugate of a torus is a torus.
				In addition given a maximal torus $T$, for all $x\in G$ there exists an element $g\in G$ such that $g^{-1}xg\in T$.
				Hence the union of all maximal tori is $G$.
		\end{prop}
		
		It is therefore unambiguous to refer to maximal torus $T$ of $G$ and consider the quotient $G/T$, which is isomorphic regardless of the choice of $T$.
		The homogeneous space $G/T$ is called the complete flag manifold of $G$.
		The rank of Lie group $G$ is the dimension of a maximal torus $T$.
		The ranks of classical simple Lie groups can be deduced by considering the standard maximal tori of $SU(n),SO(n)$ and $Sp(n)$
		see for example \cite[Chapter 7]{MatrixGroups}.
		For the ranks of the exception simple Lie groups see \cite{Exceptional}.
		
		\begin{prop}\label{prop:rank}
				For $n\geq 1$, the ranks of $SU(n+1),Sp(n),SO(2n)$ and $SO(2n+1)$ are $n$.
				The ranks of $G_2,F_4,E_6,E_7$ and $E_8$ are $2,4,6,7$ and $8$ respectively.
		\end{prop}
		
		Define the Weyl group of Lie group $G$ with maximal torus $T$ to be
		$W_G=N_G(T)/Z(T)$ the normalizer of $T$ in $G$ quotient the centraliser of $T$ in $G$.
		The cohomology of homogeneous spaces was studied in detail by Borel in \cite{Borel}.
		In particular, from Borel's work it was possible to deduce the rational cohomology of $G/T$.
		\begin{thm}\label{thm:Borel}
			For compact connected Lie group $G$ with maximal torus $T$
			\begin{equation*}
				H^*(G/T;\mathbb{Q})\cong \frac{H^*(BT;\mathbb{Q})}{\tilde{H}^*(BT;\mathbb{Q})^{W_G}}
			\end{equation*}
			where $BT$ is the classifying space of $T$.
		\end{thm}
		In \cite{torG/T} Bott and Samelson, using Morse theory, extended Borel's work by showing that there is no torsion in $H^*(G/T;\mathbb{Z})$.
		This made it easier to deduced the integral structure of the cohomology of complete flag manifolds in the cases of $SU(n),Sp(n)$ and $G_2$.
		Toda later in \cite{toda1975} studied again the cohomology of homogeneous spaces, looking at the mod $p$ cohomology for prime $p$.
		In particular Toda was able to deduce in a nice form the integral cohomology algebras of complete flag manifolds in the case of $SO(n)$.
		Then in \cite{toda1974}, Toda and Watanabe computed the cohomology in the cases of $F_4$ and $E_6$.
		Finally the cohomology of complete flag manifolds of simple Lie groups was completed by Nakagawa in \cite{NakagawaE7} and \cite{nakagawaE8},
		finishing the cases $E_7$ and $E_8$. 
	
		\begin{thm}[\cite{Borel}, \cite{AplicationsOfMorse}]\label{thm:H*SU/T}
			For each $n\geq 0$, the cohomology of the complete flag manifold of the simple Lie group $SU(n+1)$ is given by
			\begin{equation*}
				H^*(SU(n+1)/T^n;\mathbb{Z})=\frac{\mathbb{Z}[\gamma_1,\dots,\gamma_{n+1}]}{[\sigma_1,\dots,\sigma_{n+1}]},
			\end{equation*}
			where $|\gamma_i|=2$.
		\end{thm}
		
		\begin{proof}[Sketch proof]
			By Theorem \ref{thm:Borel}
			\begin{equation*}
				H^*(SU(n+1)/T^n;\mathbb{Q})\cong \frac{H^*(BT^n;\mathbb{Q})}{\tilde{H}^*(BT^n;\mathbb{Q})^{W_{SU*n+1)}}}.
			\end{equation*}
			The cohomology of the classifying space of the $n$-torus $BT^n$ is $\mathbb{Q}[x_1,\dots,x_n]$ where $|x_i|=2$.
			The Weyl group $W_{SU(n+1)}$ is the symmetric group $S_{n+1}$.
			$W_{SU(n+1)}$ acts on the indices $x_1,\dots,x_n,x_{n+1}$, where $x_{n+1}=-x_1-\dots-x_n$.
			Hence the rational version of the theorem is proved.
			In \cite{AplicationsOfMorse} Bott and Samelson showed that the integral cohomology of complete flag manifolds is concentrated in even degrees and is torsion free.
			Therefore the problem of finding the integral cohomology to considering the map $H^*(SU(n+1)/T^n;\mathbb{Z})\to H^*(SU(n+1)/T^n;\mathbb{Q})$,
			induced by the universal coefficients theorem.
			This problem is easily resolved in the cases of $SU(n)$ and $Sp(n)$ but not for other simple Lie groups.
		\end{proof}

		\begin{thm}[\cite{toda1974}, Theorem $2.1$]\label{thm:H*SUodd}
			For each $n\geq 1$, the cohomology of the complete flag manifold of the simple Lie group $SO(2n+1)$ is given by
			\begin{equation*}
				H^*(SO(2n+1)/T^n;\mathbb{Z})=\frac{\mathbb{Z}[\gamma_1,\dots,\gamma_n,t_1,\dots,t_n]}{[\sigma_i-2t_i, \;\; t_{2i}+\sum_{j=1}^{2i-1}(-1)^jt_jt_{2i-j}]},
			\end{equation*}
			where $1\leq i \leq n-1$, $|\gamma_i|=2$ and $|t_i|=2i$.
		\end{thm}
	
		\begin{thm}[\cite{toda1974}, Corollary $2.2$]\label{thm:H*SOeven}
			For each $n\geq 1$, the cohomology of the complete flag manifold of the simple Lie group $SO(2n)$ is given by
			\begin{equation*}
				H^*(SO(2n)/T^n;\mathbb{Z})=
				\frac{\mathbb{Z}[\gamma_1,\dots,\gamma_n,t_1,\dots,t_{n-1}]}{[\sigma_i-2t_i, \;\; \sigma_n, \;\; t_{2i}+\sum_{j=1}^{2i-1}(-1^j)t_jt_{2i-j}]},
			\end{equation*}
			where $1\leq i \leq n$, $|\gamma_i|=2, |\gamma_n|=2$ and $|t_i|=2i$.
		\end{thm}
		
		\begin{thm}[\cite{Borel}, \cite{AplicationsOfMorse}]\label{thm:H*Sp/T}
			For each $n\geq 1$, the cohomology of the complete flag manifold of the simple Lie group $Sp(n)$ is given by
			\begin{equation*}
				H^*(Sp(n)/T^n;\mathbb{Z})=\frac{\mathbb{Z}[\gamma_1,\dots,\gamma_n]}{[\sigma_1^2,\dots,\sigma_{n}^2]},
			\end{equation*}
			where $|\gamma_i|=2$ and $\sigma_i^2$ denotes elementary symmetric polynomial $\sigma_i$ in variables $\gamma_1^2,\dots,\gamma_n^2$.
		\end{thm}

		\begin{thm}[\cite{AplicationsOfMorse}, Theorem $III'$]\label{thm:H*G2/T}
			The cohomology of the complete flag manifold of the exceptional simple Lie group $G_2$ is given by
			\begin{equation*}
				H^*(G_2/T^2;\mathbb{Z})=\frac{\mathbb{Z}[\gamma_1,\gamma_2,\gamma_3,t_3]}{[\sigma_1,\sigma_2,\sigma_3-2t_3,t_3^2]},
			\end{equation*}
			where $|\gamma_i|=2$ for $i=1,2,3$, $|t_3|=6$
			and $\sigma_3^2$ denotes elementary symmetric polynomial $\sigma_3$ in variables $\gamma_1^2,\gamma_2^2,\gamma_3^2$.
		\end{thm}

		\begin{thm}[\cite{toda1974}]\label{thm:H*F4/T}
			The cohomology of the complete flag manifold of the exceptional simple Lie group $F_4$ is given by
			\begin{equation*}
				H^*(F_4/T^4;\mathbb{Z})=\frac{\mathbb{Z}[\gamma_1,\gamma_2,\gamma_3,\gamma_4,\gamma,t_3,t_4]}
				{I},
			\end{equation*}
			where $|\gamma_i|=2$ for $i=1,2,3,4$, $|\gamma|=2$, $|t_3|=6$, $|t_4|=8$ and
			\begin{align*}
				I=[\sigma_1-2\gamma, \sigma_2-2\gamma^2,\sigma_3-2\gamma_3, \; \sigma_4-4\gamma t_3+8\gamma^4-3t_4, \; t_3^2-3\gamma^2t_4-4\gamma^3t_3+8\gamma^6, \\
				3t_4^2-6\gamma t_3t_4-3\gamma^4t_4-13\gamma^8, \; t_4^3-6\gamma^4t_4^2+12\gamma^8t_4-8\gamma^{12}].
			\end{align*}
		\end{thm}
		
		\begin{thm}[\cite{toda1974}]\label{thm:H*E6/T}
			The cohomology of the complete flag manifold of the exceptional simple Lie group $E_6$ is given by
			\begin{equation*}
				H^*(E_6/T^6;\mathbb{Z})=\frac{\mathbb{Z}[\gamma_1,\gamma_2,\gamma_3,\gamma_4,\gamma_5,\gamma_6,t_1,t_3,t_4]}{I},
			\end{equation*}
			where $|\gamma_i|=2$, $1\leq i \leq 6$, $|t_1|=2$, $|t_3|=6$, $|t_4|=8$ and
			\begin{align*}
				I=[
				\sigma_1-3t_1, \; \sigma_2-4t_1^2, \; \sigma_3-2t_3, \sigma_4+2t_1^4-3t_4, \; \sigma_5-\sigma_4t_1+\sigma_3t_1^2-2t_1^5, \\
				2\sigma_6-\sigma_4t_1^2-t_1^6+t_3^2, \; 9\sigma_6t_1^2+3\sigma_5t_1^3-t_1^8+3t_4(t_4-\sigma_3t_1+2t_1^4), \\
				t^9-3w^2t, \; w^3+15w^2t^4-9wt^8
				],
			\end{align*}
			where $t=t_1-\gamma_1$ and $w=t_1-\sigma_3t_1+2t_1^4+t(t_3-2t_1^3+t_1^2t-t_1t^2+t^3)$.
		\end{thm}

		\begin{thm}[\cite{NakagawaE7}]\label{thm:H*E7/T}
			The cohomology of the complete flag manifold of the exceptional simple Lie group $E_7$ is given by
			\begin{equation*}
				H^*(E_7/T^7;\mathbb{Z})=\frac{\mathbb{Z}[\gamma_1,\gamma_2,\gamma_3,\gamma_4,\gamma_5,\gamma_6,\gamma_7,\gamma,t_3,t_4,t_5,t_9]}{I},
			\end{equation*}
			where $|\gamma_i|=2=|\gamma|$, $1\leq i \leq 7$, $|t_3|=6$, $|t_4|=8$, $|t_5|=10$, $|t_9|=18$ and
			\begin{align*}
				I=[
				&\sigma_1-3\gamma, \\
				&\sigma_2-4\gamma^2, \\
				&\sigma_3-2t_3, \\
				&\sigma_4+2\gamma^4-3t_4, \\
				&\sigma_5-3\gamma t_4+2\gamma^2t_3-2t_5, \\
				&t_3^2+2\sigma_6-2\gamma t_5-3\gamma^3t_4+\gamma^6, \\
				&3t_4^2-2t_3t_5+2\gamma \sigma_7-6\gamma t_3t_4-9\gamma^2\sigma_6+12\gamma^3t_5+15\gamma^4t_4-6\gamma^5t_3-\gamma^8, \\
				&2\sigma_6t_3+\gamma^2\sigma_7-3\gamma^3\sigma_6-2t_9, \\
				&t_5^2-2\sigma_7t_3+3\gamma^3\sigma_7 \\
				&-6\gamma_0^8u+9\gamma_0^4u^2+2\gamma^6_0u^2-12\gamma_0^2uv+u^3+3v^2, \\
				&\gamma_0^14-6\gamma_0^10u-3\gamma_0^6u^2+4\gamma^8_0uv-3u^2v+3\gamma_0^2v^2, \\
				&-2\gamma_0^14u+6t^6_0u^3+9w^2-2\gamma_0^8uv-12\gamma_0^4u^2v-3u^3v-\gamma_0^6v^2+6\gamma^2_)uv^2-2v^3
				],
			\end{align*}
			where $\gamma_0=\gamma-\gamma_1$, $u=t_4-(2\gamma_1+\gamma_0)t_3+2\gamma_1^4+6\gamma_1^3\gamma_0^2+7\gamma_1^2\gamma_0^2+3\gamma_1\gamma_0^3$, \\
			$v=\sigma_6-(2\gamma_1+\gamma_0)t_5-3\gamma_1\gamma_0t_4+(4\gamma_1^2\gamma_0+2\gamma_1\gamma_0^2)t_3-3\gamma_1^5\gamma_0-8\gamma_1^4\gamma_0^2-8\gamma_1^3\gamma_0^3$
			and $w=\frac{1}{2}\gamma_0u^2$.
		\end{thm}
		
		\begin{thm}[\cite{nakagawaE8}]\label{thm:H*E8/T}
			The cohomology of the complete flag manifold of the exceptional simple Lie group $E_8$ is given by
			\begin{equation*}
				H^*(E_8/T^8;\mathbb{Z})=\frac{\mathbb{Z}[\gamma_1,\gamma_2,\gamma_3,\gamma_4,\gamma_5,\gamma_6,\gamma_7,\gamma_8,\gamma,t_3,t_4,t_5,t_6,t_9,t_{10},t_{15}]}{I},
			\end{equation*}
			where $|\gamma_i|=2=|\gamma|$ for $1\leq i \leq 8$, $|t_j|=2j$ for $j=3,4,5,6,9,10,15$ and
			\begin{align*}
				I=[
				&\sigma_1-3\gamma, \;
				\sigma_2-4\gamma, \;
				\sigma_3-2t_3, \;
				\sigma_4+2\gamma^4-3t_3, \;
				\sigma_5-3\gamma t_4+2\gamma^2t_3-2t_5, \\
				&\sigma_6-2t_3^2-\gamma t_5+\gamma^2t_4-\gamma^6-5t_6, \; 
				-3\sigma_8+3t_4^2-2t_3t_5+\gamma(2\sigma_7-6t_3t_4), \\
				&2\sigma_6t_3+\gamma\sigma_8+\gamma^2\sigma_7-3\gamma^3\sigma_6-2t_9, \;
				t_5^2-2\sigma_7t_3-\gamma^2\sigma_8+3\gamma^3\sigma_7-3t_{10}, \\
				&15t_6^2+2t_3t_4t_5-2\sigma_7t_5+2t_3^4+10t_3^2t_6-3\sigma_8t_4-2t_4^3
					+\gamma(\sigma_8t_3-2t_3^2t_5+4\sigma_7t_4+6t_3t_4^2) \\
					&+\gamma^2(3t_10-25t_4t_6-\sigma_7t_4+6t_3t_4^2) 
					+\gamma^3(25t_3t_6-3t_4t_5+10t_3^3)
					+\gamma^4(3\sigma_8+3t_3t_5+5t_4^2) \\
					&+\gamma^5(-3\sigma_7-5t_3t_4)+4\gamma^6t_3^2-7\gamma^8t_4+4\gamma^9t_3, \\
				&\sigma_7^2-3\sigma_8t_6+6t_4t_{10}-4\sigma_8t_3^2+6\sigma_7t_3t_4-6t_3^2t_4^2-12t_4^2t_6-2t_3t_5t_6 \\
					&+\gamma(24t_3t_4t_6-8\sigma_7t_3^2-8\sigma_7t_6+4\sigma_8t_5-6t_3t_{10}+12t_3^3t_4) \\
					&+\gamma^2(-2t_3t_4t_5+6t_4^3+2t_3^2t_6+20t_6^2-4t_3^4-\sigma_7t_5)
					+\gamma^3(-12t_3t_4^2+8\sigma_8t_3-5\sigma_7t_4+3t_5t_6) \\
					&+\gamma^4(3t_{10}-26t_4t_6+6\sigma_7t_3-4t_3^2t_4)
					+\gamma^5(24t_3t_6+3t_4t_5+12t_3^3)
					+\gamma^6(-6\sigma_8+2t_4^2)\\
					&-2(t_3^2+\sigma_6)(t_9-\sigma_6t_3)-2t_{15}, \\
				&t_9^2-9\sigma_8t_{10}-6t_4^2t_{10}-4t_3^3t_9-10t_3t_6t_9+2t_3t_5t_{10}-2t_3t_4t_5t_6-6\sigma_7t_4^2+3\sigma_8t_4t_6 \\
					&+\sigma_8t_3^2t_4+6t_3^2t_4^3+12t_4^3t_6+2\sigma_7^2t_4+2\sigma_7t_3^2t_5-2t_3^3t_4t_5+2\sigma-7t_5t_6+4t_3^6-10t_6^3 \\
					&+18t_3^4t_6+15t_3^2t_6^2-9\sigma_7\sigma_8t_3 
					 +\gamma(-2t_3t_5t_9-24\sigma_7t_4t_6+8\sigma_8t_4t_5+4\sigma_7t_3^2t_4+4\sigma_7t_{10} \\
					&-\sigma_8t_9+2\sigma_7^2t_3+4\sigma_8t_3t_6 
					 +12t_3t_4t_{10}-36t_3t_4^2t_6+12t_3^2t_5t_6+\sigma_8t_3^3+6t_3^4t_5-18t_3^3t_4^2) \\
					&+\gamma^2(24t_3^4t_4-2\sigma_8^2-\sigma_7t_9-11t_3^2t_10+2t_3t_4t_9-2\sigma_8t_3t_5+16\sigma_7t_3t_6-3\sigma_7t_4t_5 \\
					&+75t_4t_6^2-6t_4^4-9\sigma_8t_4^2+81t_3^2t_4t_6-13t_6t_{10}+4t_3t_4^2+t_5-\sigma_7t_3^3) \\
					&+\gamma^3(-3t_5t_{10}-150t_3t_6^2-135t_3^3t_6+6t_3^2t_9-2\sigma_7t_3t_5+21\sigma_7t_4^2+15\sigma_7\sigma_8+3t_4t_5t_6 \\
					&-3t_3^2t_4t_5+18t_3t_4^3+15t_6t_9+14\sigma_8t_3t_4-30t_3^5) \\
					&+\gamma^4(-13\sigma_8t_6+2t_4t_{10}-5\sigma_7^2-33t_3^2t_4^2+3t_5t_9-28t_3t_5t_6-45t_4^2t_6-41\sigma_7t_3t_4 \\
					&-13t_3^3t_5-9\sigma_8\sigma_8^2)
					 +\gamma^5(3\sigma_7t_6-6t_4^2t_5+23\sigma_7t_3^2+105t_3t_4t_6-6\sigma_8t_5-3t_4t_9+45t_3^3t_4) \\ 
					&+\gamma^6(11t_4^3-4t_3t_9+4\sigma_7t_5+9t_3t_4t_5+12t_3^4+66t_3^2+75t_6^2+2\sigma_8t_4) \\ 
					&+\gamma^7(-33t_3t_4^2+12t_3^2+15t_5t_6)
					+\gamma^8(-4t_{10}+21t_3^2t_4-5\sigma_7t_3-3t_4t_6) \\
					&+\gamma^9(6t_9-42t_3^3-99t_3t_6) 
					+\gamma^{10}(-4\sigma_8-6t_4^2-13t_3t_5)
					+\gamma^{11}(3\sigma_7+27t_3t_4) \\
					&+\gamma^{12}(60t_6+18t_3^2) 
					+6\gamma^{13}t_5-9\gamma^{14}t_4-12\gamma^{15}t_3+10\gamma_3^2, \\
				&9\gamma_8^20+45\gamma_8^{14}v+12\gamma_8^{10}w+60\gamma_8^8v^2+30\gamma_8^4vw+10\gamma_8^2v^3+3w^2, \\
				&11\gamma_8^{24}+60\gamma_8^{18}v+21\gamma_8^{14}w+105\gamma_8^12v^2+60\gamma_8^8vw+60\gamma_8^6v^3+9\gamma_8^4w^2+30\gamma_8^2v^2w+5v^4, \\
				&-9x^2-12\gamma_8^9vx-6\gamma_8^5wx+9\gamma_8^{14}vw-10\gamma_8^{12}v^3-3\gamma_8^{10}w^2+30\gamma_8^8v^2w-35\gamma_8^6v^4 \\
					&+6\gamma_8^4vw^2-10\gamma_8^2v^3w-4v^5-2w^3
				].
			\end{align*}
			where
			\begin{align*}
				v=&2t_6+t_3^2-\gamma_8t_5+t_4(-\gamma^4+\gamma_8^2)-\gamma_8^3t_3+\gamma^6-\gamma^4\gamma_8^2+\gamma^3\gamma_8^3+\gamma^2\gamma_8^4-\gamma \gamma_8^5 , \\
				w=&t_{10}+\gamma_8t_9-\gamma_8^3\sigma_7\gamma_8t_4t_5+2\gamma_8^2t_4^2-2\gamma_8^2t_3t_5 \\
					&+t_3t_4(-6\gamma\gamma_8^2+2\gamma_8^3)+t_3^2(2\gamma^2\gamma_8^2+2\gamma\gamma_8^3-2\gamma_8^4)+t_6(-5\gamma^2\gamma_8^2+5\gamma\gamma_8^3) \\
					&+t_5(\gamma^4\gamma_8+3\gamma^3\gamma_8^2+\gamma^2\gamma_8^3)+t_4(6\gamma^4\gamma_8^2-3\gamma^3\gamma_8^3-2\gamma^2\gamma_8^4-\gamma\gamma_8^5
					 -\gamma\gamma_8^5+\gamma_8^6) \\
					&+t_3(-6\gamma^5\gamma_8^2-2\gamma^4\gamma_8^3+4\gamma^3\gamma_8^4+6\gamma^2\gamma_8^5-4\gamma\gamma_8^6+\gamma_8^7) \\
					&+4\gamma^7\gamma_8^3-6\gamma^6\gamma_8^5+2\gamma^4\gamma_8^6+\gamma^3\gamma_8^7-\gamma^2\gamma_8^8, \\
				x=&t_{15}-20t_3t_6^2+3t_3^2t_9-23t_3^3t_6-6t_3^5+4t_6t_9+3\gamma_8t_4t_{10}-\gamma_8t_5t_9-3\gamma_8t_3^2t_4^2+3\gamma_8\sigma_7t_3t_4 \\
					&-6\gamma_8t_4^2t_6+t_3^3t_5(-3\gamma+2\gamma_8)+t_3t_5t_6(-4\gamma+4\gamma_8)+t_4t_9(-\gamma^2-\gamma_8^2)+\sigma_7t_3^2(\gamma^2+\gamma\gamma_8-\gamma_8^2) \\
					&+t_3t_4t_6(9\gamma^2+12\gamma\gamma_8+5\gamma_8^2)+t_3^3t_4(5\gamma^2+6\gamma\gamma_8+2\gamma_8^2)+\sigma_7t_6(3\gamma^2+4\gamma\gamma_8+\gamma_8^2)
					-\gamma_8^3t_3t_9 \\
					&+t_3^4(-6\gamma^3-2\gamma^\gamma_8-6\gamma\gamma_8^2+5\gamma_8^3)+t_4^3(3\gamma^2\gamma_8+\gamma_8^3)+\sigma_7t_5(2\gamma^2\gamma_8+3\gamma\gamma_8^2) \\
					&+t_6^2(-45\gamma^3+10\gamma^2\gamma_8-40\gamma\gamma_8^2)+t_3t_4t_5(\gamma^3-2\gamma^2\gamma_8+\gamma\gamma_8^2-\gamma_8^3) \\
					&+t_3^2t_6(-33\gamma^3+\gamma^2\gamma_8-31\gamma\gamma_8^2+13\gamma_8^3)+\sigma_7t_4(-2\gamma^4-4\gamma^3\gamma_8-3\gamma\gamma_8^3+3\gamma_8^4) \\
					&+t_5t_6(-9\gamma^4-6\gamma^3\gamma_8-18\gamma^2\gamma_8^2+5\gamma\gamma_8^3-3\gamma_8^4)+t_3^2t_5(-3\gamma^4-3\gamma^3\gamma_8-7\gamma^2\gamma_8^2
					 +5\gamma\gamma_8^3-4\gamma_8^4) \\
					&+t_3t_4^2(-\gamma^4-6\gamma^3\gamma_8-\gamma^2\gamma_8^2-3\gamma\gamma_8^3)
					 +t_{10}(-3\gamma^4\gamma_8-6\gamma^3\gamma_8^2+3\gamma^2\gamma_8^3+15\gamma\gamma_8^4) \\
					&+\sigma_7t_3(-3\gamma^4\gamma_8+\gamma^3\gamma_8^2+5\gamma^2\gamma_8^3+10\gamma\gamma_8^4-\gamma_8^5) \\
					&+t_3^2t_4(15\gamma^5-2\gamma^4\gamma_8+3\gamma^3\gamma_8^2+14\gamma^2\gamma_8^3-16\gamma\gamma_8^4+3\gamma_8^5) \\
					&+t_4t_6(39\gamma^5-13\gamma^4\gamma_8+8\gamma^3\gamma_8^2+35\gamma^2\gamma_8^3-31\gamma\gamma_8^4-3\gamma_8^5) \\
					&+t_9(\gamma^6-\gamma^4\gamma_8^2-\gamma^3\gamma_8^3-\gamma^2\gamma_8^4-\gamma\gamma_8^5-\gamma_8^6) \\
					&+t_3t_6(-13\gamma^6+12\gamma^5\gamma_8+5\gamma^4\gamma_8^2-56\gamma^3\gamma_8^3+8\gamma^2\gamma_8^4+21\gamma\gamma_8^5+2\gamma_8^6) \\
					&+t_4t_5(6\gamma^6+3\gamma^5\gamma_8+2\gamma^4\gamma_8^2+7\gamma^3\gamma_8^3+\gamma^2\gamma_8^4-8\gamma\gamma_8^5+3\gamma_8^6) \\
					&+t_3^3(-8\gamma^6+6\gamma^5\gamma_8+2\gamma^4\gamma_8^2-22\gamma^3\gamma_8^3+6\gamma^2\gamma_8^4+8\gamma\gamma_8^5-2\gamma_8^6) \\
					&+t_4^2(-6\gamma^7+\gamma^6\gamma_8-7\gamma^4\gamma_8^3+5\gamma^3\gamma_8^4+3\gamma^2\gamma_8^5+3\gamma\gamma_8^6-63\gamma_8^7) \\
					&+t_3t_5(-\gamma^7+2\gamma^6\gamma_8+\gamma^5\gamma_8^2-11\gamma^4\gamma_8^3+6\gamma^3\gamma_8^4+5\gamma^2\gamma_8^5+6\gamma\gamma_8^6+39\gamma_8^7) \\
					&+\sigma_7(2\gamma^8+6\gamma^7\gamma_8+3\gamma^6\gamma_8^2-4\gamma^5\gamma_8^3-15\gamma^4\gamma_8^4+6\gamma^3\gamma_8^5+3\gamma^2\gamma_8^6
					 -40\gamma\gamma_8^7+59\gamma_8^8) \\
					&+t_3t_4(3\gamma^8+\gamma^6\gamma_8^2+11\gamma^5\gamma_8^3+14\gamma^4\gamma_8^4-20\gamma^3\gamma_8^5-4\gamma^2\gamma_8^6+118\gamma\gamma_8^7+3\gamma_8^8) \\
					&+t_6(-48\gamma^9+3\gamma^8\gamma_8-41\gamma^7\gamma_8^2+18\gamma^6\gamma_8^3+16\gamma^5\gamma_8^4-13\gamma^4\gamma_8^5-67\gamma^3\gamma_8^6
					 +125\gamma^2\gamma_8^7 \\
					&-15\gamma\gamma_8^8-291\gamma_8^9)
					+t_3^2(-18\gamma^9-3\gamma^8\gamma_8-16\gamma^7\gamma_8^2+10\gamma^6\gamma_8^3-4\gamma^5\gamma_8^4-8\gamma^4\gamma_8^5-16\gamma^3\gamma_8^6 \\
					&-23\gamma^2\gamma_8^7-10\gamma\gamma_8^8-115\gamma_8^9)
					+t_5(-6\gamma^{10}-3\gamma^9\gamma_8-9\gamma^8\gamma_8^2+5\gamma^7\gamma_8^3-5\gamma^6\gamma_8^4-14\gamma^4\gamma_8^6 \\
					&-52\gamma^3\gamma_8^7+6\gamma^2\gamma_8^8-60\gamma\gamma_8^9+117\gamma_8^10) 
					+t_4(18\gamma^{11}-3\gamma^{10}\gamma+5\gamma^9\gamma_8^2+11\gamma^8\gamma_8^3-28\gamma^7\gamma_8^4 \\
					&+8\gamma^6\gamma_8^5+20\gamma^5\gamma_8^6-64\gamma^4\gamma_8^7-15\gamma^3\gamma_8^8 
					 +54\gamma^2\gamma_8^9+178\gamma\gamma_8^{10}-177\gamma_8^{11}) \\
					&+t_3(-2\gamma^{12}+6\gamma^{11}\gamma_8+2\gamma^{10}\gamma_8^2-20\gamma^9\gamma_8^3+11\gamma^8\gamma_8^4+22\gamma^7\gamma_8^5-8\gamma^6\gamma_8^6 \\
					&+83\gamma^5\gamma_8^7+15\gamma^4\gamma_8^8+5\gamma^3\gamma_8^9-116\gamma^2\gamma_8^{10}+\gamma\gamma_8^{11}+117\gamma_8^{12}) \\
					&-12\gamma^{15}-\gamma^{14}\gamma_8-10\gamma^{13}\gamma_8^2+6\gamma^{12}\gamma_8^3+7\gamma^{11}\gamma_8^4-13\gamma^{10}\gamma_8^5-31\gamma^9\gamma_8^6
					 +9\gamma^8\gamma_8^7-\gamma^7\gamma_8^8 \\
					&-118\gamma^6\gamma_8^9-18\gamma^5\gamma_8^{10}+131\gamma^4\gamma_8^{11}-6\gamma^3\gamma_8^{12}-233\gamma^2\gamma_8^{13}+175\gamma\gamma_8^{14}
					 -58\gamma_8^{15}.
			\end{align*}
		\end{thm}

	\subsection{Based loop space cohomology of Lie groups}\label{sec:LoopLie}
		
		The Hopf algebra of the based loop space of Lie groups were studied by Bott in \cite{bott1958}.
		Here we give just the more straight forwardly produced results which we intend to use latter in this thesis.
		
		\begin{defn}
			Define the integral divided polynomial algebra on variables $x_1,\dots,x_n$ by
			\begin{equation*}
				\Gamma_{\mathbb{Z}}[x_1,\dots,x_n]=\frac{\mathbb{Z}[(x_i)_1,(x_i)_2,\dots]}{[(x_i)_k-k!x_i^k]},
			\end{equation*}
			for $1\leq i \leq n$ and $k\geq 1$ and where $x_i=(x_i)_1$.
		\end{defn}
	
		The following two theorems follow from Theorem \ref{thm:H*(SU(n))} and \ref{thm:H*(Sp(n))},
		using a Leray-Serre spectral sequence argument with the path space fibrations $\Omega SU(n) \to PSU(n) \to SU(n)$ and $\Omega Sp(n) \to PSp(n) \to Sp(n)$.
		
		\begin{thm}\label{thm:LoopSU(n)}
			For each $n\geq 1$, the cohomology of the based loop space of the classical simple Lie group $SU(n)$ is given by
			\begin{equation*}
				H^*(\Omega(SU(n));\mathbb{Z})=\Gamma_{\mathbb{Z}}[x_2,x_4,\dots,x_{2n-2}],
			\end{equation*}
			where $|x_i|=i$ for $i=2,4,\dots,2n-2$.
		\end{thm}
		
		\begin{proof}
			We proceed by induction on $n$.
			We have that $SU(1) = \{ pt \}$ hence by definition $\Omega SU(1) = \{ pt \}$, so has trivial cohomology ring.
			
			Now assume that $n \geq 2$. 
			We will apply the Leray-Serre spectral sequence to the path space fibration (\ref{eq:PathFibe}) for $X=SU(n)$,
			\begin{equation*}
				\Omega SU(n) \to PSU(n) \to SU(n).
			\end{equation*}
			Denote this spectral sequence by $\{ E_r,d^r \}$.
			Since $PSU(n)$ is contractible the spectral sequence will converge to the trivial algebra,
			which is $0$ in all entities except for $E_{\infty}^{0,0}$.
			Hence all non-zero entries are in the image of some differential $d^r$.
			
			In Figure \ref{pic:SS2} below, we identify the horizontal axis with $H^{*}(SU(n))$ and the vertical axis with $H^{*}(\Omega SU(n))$.
			Throughout the induction argument we obtain additional algebra generators of ${\Gamma}_{\mathbb{Z}}(x_{2n-2})$ in $H^{*}(\Omega SU(n);\mathbb{Z})$ not in 
			$H^{*}(\Omega SU(n-1);\mathbb{Z})$ using only the differential of degree $n$. 
			Hence we can assume all elements associated to generators of lower degree have all been annihilated before the $E^{*,*}_{n}$ page.
			
			When $n=2$ there are no non-zero differentials before page $E^{*,*}_{3}$ as the first non-trivial generator of $H^{*}(SU(2))$ has degree $3$.
			The only new generator of $H^{*}(SU(n))$ not in $H^{*}SU(n-1)$ is $x_{n}$.
			Since differentials have bidegree $(r,1-r)$, the only differential with domain in column $E^{0,*}_{r}$ to have image in column $E^{2n-1,*}_{r}$ is $d_{2n-1}$.
			The differential with image $E^{2n-1,0}_{2n-1}$ therefore must be an isomorphism and so we get a new generator of $H^{*}(\Omega SU(n))$
			in dimension $2n-2=2(n-1)$, which we will denote by $b_{1}$ with $d_{2n-1}(b_{1})=x_{1}$.
			
			Note that all products of $x_{2n-1}$ with the other generators $x_{3}, \dots x_{2n-3}$ are annihilated by differentials of degree less than $n$,
			with codomain $b_{1}$ multiplied byother elements in the multiplicative structure of $E_{2}^{*,*}$, which we will denote by $\cdot$.
			Annihilated by this differential due to the Leibnitz rule on differentials.
			Hence the only other potently non-zero entries on page $E_{2n-1}^{*,*}$ are in entries in $E_{2n-1}^{0,*}$
			and $E_{2n-1}^{2n-1,q}$ where $q= 2(n-1), 4(n-1), 6(n-1), \dots$.
			As all other entries are zero, the differentials with image $E_{2n-1}^{2n-1,q}$ on $E_{2n-1}^{*,*}$ are all isomorphisms.
			This gives new elements $b_{i}$ with $d_{2n-1}(b_{i})=x_{2n-1} \cdot b_{i-1}$ for each $i \geq 2$.
			We know that $b_{1}^{i}$ and $b_{i}$ have the same degree.
			
			From multiplication in $E^{*,*}_{2}$ and graded commutativity of the cup product,
			we deduce that
			\begin{equation*}
				\begin{aligned}
						d_{2n-1}(b_{1}^{2}) &=d_{2n-1}(b_{1})\cdot b_{1} +(-1)^{0 \cdot 2(n-1)}b_{1}d_{2n-1}(b_{1}) \\
						&= x_{2n-1} \cdot b_{1} + b_{1} \cdot x_{2n-1}                          \\
						&= x_{2n-1} \cdot b_{1} + (-1)^{2(n-1)(2n-1)} x_{2n-1} \cdot b_{1}            \\
						&= 2x_{2n-1} \cdot b_{1}                                    
				\end{aligned}
			\end{equation*}
			so $d_{2n-1}(b_{1}^{2})=2x_{2n-1} \cdot b_{1}$.
			Next we show by induction on $i$ that for each $i \geq 2$, $b_{1}^{i}=i!b_{i}$.
			Note that by definition of generators and applying isomorphisms $d_{2n-1}$,
			we have $b_{1}^{i}=i!x_{i}$ is equivalent to $d_{2n-1}(b_{1}^{i})= i! x_{2n-1} \cdot b_{i-1}$
			and $b_{i}= i b_{i-1} \cdot b_{1}$.
			Hence the following calculation is the induction step.
			\begin{equation*}
				\begin{aligned}
						d_{2n-1}(b_{1}^{i}) &=d_{2n-1}(b_{1}^{i-1})\cdot b_{1} +(-1)^{0 \cdot 2(n-1)}b_{1}^{i-1}d_{2n-1}(b_{1}) \\
						&=(i-1)! x_{2n-1} \cdot b_{i-2} \cdot b_{1} + b_{i-1} \cdot (i-1)!x_{2n-1} \\
						&=(i-1)! x_{2n-1} \cdot (i-1) b_{i-1} + (i-1)! b_{i-1} \cdot x_{2n-1} \\
						&=i! x_{2n-1} \cdot b_{i-1}
				\end{aligned}
			\end{equation*}
			This means that $\langle b_{1},b_{2},b_{3} \dots \rangle = {\Gamma}_{\mathbb{Z}}(b_{1})$.
			In addition these generators interact freely with all previous generators, as they are annihilated by differential of different degrees.
			Therefore $b_1$ is the additional element $x_{2n-2}$ in ${\Gamma}_{\mathbb{Z}}(x_2,x_4, \dots ,x_{2n-2})$ 
			not in ${\Gamma}_{\mathbb{Z}}(x_2,x_4, \dots ,x_{2n-4})$ for $H^*(SU(n-1))$, as in the statement of the theorem.
			
				\begin{center}
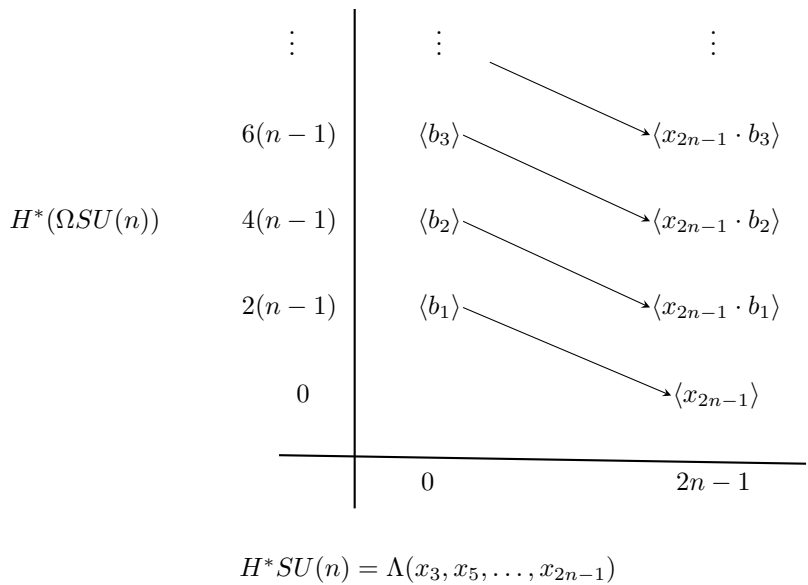

			\begin{tikzpicture}
				\matrix (m) [matrix of math nodes,
					nodes in empty cells,nodes={minimum width=0ex,
					minimum height=7ex,outer sep=-5pt},
					column sep=1ex,row sep=0ex]{
																&		\;\;\;\;\;\;\;\;\;\;  \vdots	\;\;\; \;\;\;\;\;\;\;\;\;\;\;\;\;\;\;\;\vdots\;\;\;\;\;\;\;			&	 \vdots	 						 & \\
																&		6(n-1)		\;\;\;\;\;\;\;\;\;\;\;  \langle b_{3} \rangle	\:								   &  \:\langle x_{2n-1} \cdot b_{3} \rangle  & \\
					H^{*}(\Omega SU(n))\!\!\!\!\!\!\!\!\!\!\!\!\!\!\!\!\!\!\!\!\!\!&		4(n-1)   \;\;\;\;\;\;\;\;\;\;\;   \langle b_{2} \rangle \:								 &  \:\langle x_{2n-1} \cdot b_{2} \rangle  & \\													
																&		2(n-1)	\;\;\;\;\;\;\;\;\;\;\;   \langle b_{1} \rangle \:	 								 &  \:\langle x_{2n-1} \cdot b_{1} \rangle  & \\
																&			 0\;\;\;\;\;\;\;\;\;\;\;\;\;\;																	 &  \:\langle x_{2n-1} \rangle  & \\
																&						  	            \;\;\;\;\;\;\;\;\;\;\;\;\;\;\;\;\;\;\;\;0				 &		2n-1	& \\
																&	 \;\;\;\;\;\;\;\;\;\;\;\;\;\;\;\;\;\;\;\; H^{*}SU(n) = \Lambda (x_{3},x_{5},\dots,x_{2n-1})&		& \strut \\};
				\draw[-stealth] (m-1-2.south east) -- (m-2-3.west);
				\draw[-stealth] (m-2-2.east) -- (m-3-3.west);
				\draw[-stealth] (m-3-2.east) -- (m-4-3.west);
				\draw[-stealth] (m-4-2.east) -- (m-5-3.west);
			\draw[thick] (m-1-2.north) -- (m-6-2.south) ;
			\draw[thick] (m-6-2.north west) -- (m-6-4.north east) ;
			\end{tikzpicture}
			\captionof{figure}{Serre spectral sequence for $\Omega SU(n) \to PSU(n) \to SU(n)$, $E_{2n-1}$-page.}
			\label{pic:SS2}
			\end{center}
		\end{proof}
		
		\begin{thm}\label{thm:LoopSp(n)}
			For each $n\geq 1$, the cohomology of the based loop space of the classical simple Lie group $Sp(n)$ is given by
			\begin{equation*}
				H^*(\Omega(Sp(n));\mathbb{Z})=\Gamma_{\mathbb{Z}}[x_2,x_6,\dots,x_{4n-2}],
			\end{equation*}
			where $|x_i|=i$ for $i=2,6,\dots,4n-2$.
		\end{thm}
		
		\begin{proof}
			The proof is the same as that of Theorem \ref{thm:LoopSU(n)} with the degrees of the $x_i$ shifted. 
		\end{proof}

	\subsection{Based loop space homology of complete flag manifolds}\label{sec:LoopFlag}

	In \cite{homology_Lflags}, Grbi{\'c} and Terzi{\'c} showed that the integral homology of the based loop space of a complete flag manifold is torsion free
	and found the integral Pontrjagin homology algebras the complete flag manifolds of compact connected simple Lie groups
	$SU(n)$, $Sp(n)$, $SO(n)$, $G_2$, $F_4$ and $E_6$.
	They achieved this by first using Sullivan minimal model theory to produce the rational homology algebras
	then used homotopy theory to extend these results to the integral case.
	The integral homology algebras are as follows.

	\begin{thm}[\cite{homology_Lflags}, Theorem $4.1$]
		The integral Pontrjagin homology ring of the based loop space on $SU(n+1)/T^n$ is given by
		\begin{equation*}
			H_*(\Omega(SU(n+1)/T^n);\mathbb{Z})=
			\frac{T(x_1,\dots,x_n)\otimes\mathbb{Z}[y_1,\dots,y_n]}
			{[x_k^2-x_px_q-x_qx_p, \;x_k^2-2y_1]}
		\end{equation*}
		for $1\leq k,p,q\leq n$ and $p\neq q$ where $|x_i|=1$ and $|y_i|=2i$ for each $1\leq i \leq n$.
	\end{thm}
	
	\begin{thm}[\cite{homology_Lflags}, Theorem $4.2$]
		For each $n\geq 1$ the integral Pontrjagin homology ring of the based loop space on $Sp(n)/T^n$ is given by
		\begin{equation*}
			H_*(\Omega(Sp(n)/T^n);\mathbb{Z})=
			\frac{T(x_1,\dots,x_n)\otimes\mathbb{Z}[y_2,\dots,y_n]}
			{[x_k^2-x_l^2,\; x_kx_l+x_lx_k]}
		\end{equation*}
		for $1\leq k < l \leq n$ where $|x_i|=1$ and $|y_j|=4j-2$ for each $1\leq i \leq n$ and $2\leq j \leq n$.
	\end{thm}
	
	\begin{thm}[\cite{homology_Lflags}, Theorem $4.3$]
		For each $n\geq 1$ the integral Pontrjagin homology ring of the based loop space on $SO(2n+1)/T^n$ is given by
		\begin{equation*}
			H_*(\Omega(SO(2n+1)/T^n);\mathbb{Z})=
			\frac{T(x_1,\dots,x_n)\otimes\mathbb{Z}[_1,\dots,y_{n-1},2y_n,\dots,2y_{2n-1}]}
			{[x_1^2-y_1, \; x_i^2-x_{i+1}^2, \; x_kx_l+x_lx_k, \; y_i^2-2y_{i-1}y_{i+1}+\cdots\pm2y_{2i}]}
		\end{equation*}
		for $1\leq i\leq n-1$ and $1\leq k<l \leq n$ where $y_0=1$, $|x_a|=1$, $|y_b|=2b$ and $|2y_c|=2c$
		for each $1\leq a \leq n$, $1\leq b \leq 2n-1$ and $n \leq c \leq 2n-1$.
	\end{thm}
	
	\begin{thm}[\cite{homology_Lflags}, Theorem $4.4$]
		For each $n\geq 1$ the integral Pontrjagin homology ring of the based loop space on $SO(2n)/T^n$ is given by
		\begin{equation*}
			H_*(\Omega(SO(2n)/T^n);\mathbb{Z})=
			\frac{T(x_1,\dots,x_n)\otimes\mathbb{Z}[y_1,\dots,y_{n-2},y_{n-1}+z,y_{n-1}-z,2y_n,\dots,2y_{2(n-1)}]}{I}
		\end{equation*}
		where
		\begin{align*}
			I=
			[&x_1^2-y_1, \; x_i^2-x_{i+1}^2, \; x_kx_l+x_lx_k, \; \\ &y_j^2y_{j-1}y_{j+1}+2y_{j-2}y_{j+2}-\cdots\pm2y_{2i},\; \\
			&(y_{n-1}+z)(y_{n-1-z})-2y_{n-1}y_{n+1}+\cdots\pm y_{2(n-1)}]
		\end{align*}
		for $1\leq i\leq n-1$, $1\leq j\leq n-2$ and $1\leq k<l \leq n$ where $y_0=1$, $|x_a|=1$, $|y_b|=2b$, $|y_{n-1}+z|=2(n-1)=|y_{n-1}-z|$ and $|2y_c|=2c$
		for each $1\leq a \leq n$, $1\leq b \leq n-2$ and $n \leq c \leq 2(n-1)$.
	\end{thm}
	
	\begin{thm}[\cite{homology_Lflags}, Theorem $4.5$]
		The integral Pontrjagin homology ring of the based loop space on $G_2/T^2$ is given by
		\begin{equation*}
			H_*(\Omega(G_2/T^2);\mathbb{Z})=
			\frac{T(x_1,x_2) \otimes \mathbb{Z}[y_1,y_2,y_3]}
			{[x_1^2-x_2^2,\; x_1^2-x_1x_2+x_2x_1,\; x_1^2-2y_1, \; 2y_2-x_1^4]}
		\end{equation*}
		where $|x_1|=1=|x_2|$, $|y_1|=2$, $|y_2|=4$ and $|y_5|=10$.
	\end{thm}
	
	\begin{thm}[\cite{homology_Lflags}, Theorem $4.6$]
		The integral Pontrjagin homology ring of the based loop space on $F_4/T^4$ is given by
		\begin{equation*}
			H_*(\Omega(F_4/T^4);\mathbb{Z})=
			\frac{T(x_1,x_2,x_3,x_4)\otimes\mathbb{Z}[y_1,y_2,y_3,y_5,y_7,y_{11}]}
			{[x_i^2-3y_1,\; x_px_q-x_qx_p,\; 2y_2-x_1^4,\; 3y_3-x_1^2y_2]}
		\end{equation*}
		for $1\leq i \leq 4$ and $1\leq p<q\leq 4$ where $|x_1|=|x_2|=|x_3|=|x_4|=1$ and $|y_a|=2a$ for each $a=1,2,3,5,7,11$.
	\end{thm}
	
	\begin{thm}[\cite{homology_Lflags}, Theorem $4.7$]
		The integral Pontrjagin homology ring of the based loop space on $E_6/T^6$ is given by
		\begin{equation*}
			H_*(\Omega(E_6/T^6);\mathbb{Z})=
			\frac{T(x_1,x_2,x_3,x_4,x_5,x_6)\otimes\mathbb{Z}[y_1,y_2,y_3,y_4,y_5,y_7,y_8,y_{ll}]}
			{[x_i^2-x_px_q-x_qx_p,\; x_i^2-12y_1,\; 2y_2-x_1^4,\; 3y_3-x_1^2y_2]}
		\end{equation*}
		for $1\leq i \leq 6$ and $1\leq p<q\leq 6$
		where $|x_a|=1$ and $|y_b|=2b$ for each $1\leq a\leq 6$ and $b=1,2,3,4,5,7,8,11$.
	\end{thm}

\newpage
\section{Combinatorics of polynomial symmetric quotients}\label{sec:CombiSQ}
		
		Before studying the the cohomology of the free loop space of $G/T$ in Sections \ref{sec:FreeLoopSU(n+1)/Tn} and \ref{sec:FreeLoopSp(n)/Tn}
		we first analyse some of the combinatorial structure of the cohomology algebras of the flag manifolds themselves.
		Understanding the structure of these algebras will be a major key to understanding the structure of the free loop cohomology.
		
		\subsection{Multiset coefficients}\label{subsec:MultiSet} 
		
			Recall that the binomial coefficients $\binom{n}{k}$ are defined to be the number of size $k$ subsets of an $n$ set.
			By separating the choice of an element of the $n$ set it is clear binomial coefficients satisfy the inductive formula $\binom{n}{k}=\binom{n-1}{k}+\binom{n-1}{k-1}$.
			It is easily shown by induction on $n$ that for $0\leq k \leq n$, $\binom{n}{k}=\frac{n!}{(n-k)!(k)!}$ and is zero otherwise.
			Also by induction on $n$, it is shown that binomial coefficients satisfy the well known formulas
			\begin{equation}\label{eq:binom}
				\sum_{k=0}^n{\binom{n}{k}}=2^n, \;\;\; \sum^{n}_{k=0}{(-1)^k \binom{n}{k}}=0.
			\end{equation}
		
			\begin{defn}\label{def:MutiSet}
				A {\it multiset}, unlike a set, can contain more than one of the same element. 
				The number of size $k$ multisets that can be formed from elements of a size $n$ set is denoted $\multiset{n}{k}$ and are called the {\it multiset coefficients}.
			\end{defn}
			
			It is well know that $\multiset{n}{k}=\binom{n+k-1}{k}$,
			hence $\multiset{n}{k}=\multiset{n-1}{k}+\multiset{n}{k-1}$.
			To the best of my knowledge the identity in the next Lemma has not been shown before.
			
			\begin{lem}\label{lem:combino}
				For each $n,m\geq 1$,
				\begin{equation*}
						\sum^n_{k=0}{(-1)^k\binom{n}{k}\multiset{n}{m-k}}=0.
				\end{equation*}
			\end{lem}
			\begin{proof}
				We proceed by induction on $n$.
				When $n=1$, 
				\begin{equation*}
						\sum^n_{k=0}{(-1)^k\binom{n}{k}\multiset{n}{m-k}}
						=\binom{1}{0}\multiset{1}{m}-\binom{1}{1}\multiset{1}{m-1}
						=\binom{m}{m}-\binom{m-1}{m-1}=0.
				\end{equation*}
				Suppose the lemma holds for $n=t-1\geq1$, then
				\begin{align*}
						&\sum^t_{k=0}{(-1)^k\binom{t}{k}\multiset{t}{m-k}}
						=\sum^t_{k=0}{(-1)^k\bigg(\binom{t-1}{k}+\binom{t-1}{k-1}\bigg)\multiset{t}{m-k}} \\
						&=\sum^t_{k=0}{(-1)^k\bigg(\binom{t-1}{k-1}\multiset{t}{m-k}+\binom{t-1}{k}\multiset{t-1}{m-k}+\binom{t-1}{k}\multiset{t}{m-k-1}\bigg)}=0
				\end{align*}
				as all terms cancel except for $\binom{t-1}{-1}\multiset{t}{m}$, $\binom{t-1}{t}\multiset{t-1}{m-t}$ and $\binom{t-1}{t}\multiset{t}{m-t-1}$
				all of which are zero, the middle sum $\sum^{t-1}_{k=0}{\binom{t-1}{k}\multiset{t-1}{m-k}}=0$ by assumption.
		\end{proof}
			
		\subsection{Alternative forms of the symmetric ideal}\label{subsec:IdeaForms}
			
			Recall from Section \ref{sec:elementary} that for $n\geq 1$ in $\mathbb{Z}[x_1,\dots,x_n]$,
			we define the elementary symmetric polynomials for $1\leq l\leq n$
			to be $\sigma_l=\sum_{1\leq i_1<\cdots<i_l\leq n}{x_{i_1}\cdots x_{i_l}}$
			and the elementary symmetric polynomials form a basis of the symmetric polynomials.
			We now consider two alternative expressions for the ideal $[\sigma_1,\dots,\sigma_n]$.
			
			\begin{lem}\label{lem:n+1elim}
				For each $n\geq 1$,
				\begin{equation*}
					[\sigma_1,\dots,\sigma_{n+1}]=[\sigma_1,\xi_2,\dots,\xi_{n+1}],
				\end{equation*}
				where for each $1\leq l\leq {n+1}$
				\begin{equation*}
					\xi_l =
					(1-l)\sum_{1\leq i_1<\cdots<i_l\leq n+1}{x_{i_1}\cdots x_{i_l}} -
					\sum_{\substack{1\leq i_1<\cdots<i_{l-2}\leq n+1\\
					1 \leq k \leq n+1,\; k \neq i_j}}
					{x_{i_1}\cdots x_{i_{l-2}}x_k^2}.
				\end{equation*}
				In particular
				\begin{equation*}
					\frac{\mathbb{Z}[x_1,\dots,x_{n+1}]}{[\sigma_1,\dots,\sigma_{n+1}]}=\frac{\mathbb{Z}[x_1,\dots,x_n]}{[\xi_2,\dots,\xi_{n+1}]}.
				\end{equation*}
			\end{lem}
			\begin{proof}
				Rewrite $\sigma_l$ as
				\begin{flalign*}
					\sigma_l &= \sum_{1\leq i_1<\cdots<i_l\leq n+1}{ x_{i_1}\cdots x_{i_l}}\\
					& \\
					&= \sum_{1\leq i_1<\cdots<i_l\leq n}{ x_{i_1}\cdots x_{i_l}} +
					\sum_{1\leq i_1<\cdots<i_{l-1}\leq n}{ x_{i_1}\cdots x_{i_{l-1}}  x_{n+1}}.
				\end{flalign*}
				By subtracting $\sum_{1\leq i_1<\cdots<i_{l-1}\leq n}{x_{i_1}\cdots x_{i_{l-1}}\sigma_1}$ from both sides we obtain
				\begin{flalign*}
					& \sum_{1\leq i_1<\cdots<i_l\leq n}{ x_{i_1}\cdots x_{i_l}} +
					\sum_{1\leq i_1<\cdots<i_{l-1}\leq n}{ x_{i_{1}}\cdots x_{i_{l-1}} (- x_1-\cdots- x_n)}\\
					& \\
					&=\sum_{1\leq i_1<\cdots<i_l\leq n}{ x_{i_1}\cdots x_{i_l}} -
					l\sum_{1\leq i_1<\cdots<i_l\leq n}{ x_{i_1}\cdots x_{i_l}} -
					\sum_{\substack{1\leq i_1<\cdots<i_{l-2}\leq n\\
					1 \leq k \leq n,\; k \neq i_j}}
					{ x_{i_1}\cdots x_{i_{l-2}} x_k^2}\\
					& \\
					&=(1-l)\sum_{1\leq i_1<\cdots<i_l\leq n}{ x_{i_1}\cdots x_{i_l}} -
					\sum_{\substack{1\leq i_1<\cdots<i_{l-2}\leq n\\
					1 \leq k \leq n,\; k \neq i_j}}
					{ x_{i_1}\cdots x_{i_{l-2}} x_k^2}=\xi_l.
				\end{flalign*}
				This proves that $[\sigma_1,\dots,\sigma_{n+1}]=[\sigma_1,\xi_2,\dots,\xi_{n+1}]$.
				The final statement of the lemma is obtained by rearranging the ideal as above and then removing the generator $x_{n+1}$ and ideal generator $\sigma_1$,
				which can be done since $x_{n+1}=\sigma_1-x_1-\cdots-x_n$ after quotienting out by $\sigma_1$.
			\end{proof}
			
			In addition to the elementary symmetric polynomials, recall from Section \ref{sec:homogeneous}
			another basis of the symmetric polynomials on $\mathbb{Z}[x_1,\dots,x_n]$ is given by the
			compete homogeneous symmetric polynomials, $h_l=\sum_{1\leq i_1,\dots,i_l\leq n}{x_{i_1}\cdots x_{i_l}}$ for each $1\leq l\leq n$.
			Starting with $h_l$ as generators of the of the symmetric ideal, leads to another simplification of the expression of the symmetric quotient,
			the usefulness of which will be demonstrated in the next section.
			
			For each integer $n\geq 1$ and all integers $1 \leq k' \leq k \leq n$,
			define $\Phi(k,k')$ to be the sum of all monomials in $\mathbb{Z}[x_1,\dots,x_n]$ of degree $k$ in variables $x_1,\dots,x_{n-k'+1}$.
			
			\begin{thm}\label{thm:monomial sum}
				In the ring $\frac{\mathbb{Z}[x_1,\dots,x_n]}{[h_1,\dots,h_n]}$,
				for each $1 \leq k' \leq k \leq n$, $\Phi(k,k')=0$.
				In addition 
				\begin{equation}\label{eq:OrigIdeal}
					[h_1,\dots,h_n]=[\Phi(1,1),\dots,\Phi(n,n)].
				\end{equation}
			\end{thm}
			\begin{proof}
				We replace the basis $\sigma_1,\dots,\sigma_n$ of symmetric polynomials by the complete homogeneous symmetric polynomials,
				where $h_k=\Phi(k,1)$.
				We will prove by induction on $k$ that, for each $1 \leq k' \leq k \leq n$, $\Phi(k,k')\in [h_1,\dots,h_n]$.
				When $k=1$, by definition
				\begin{equation*}
						h_1=\Phi(1,1).
				\end{equation*}
				Assume the theorem is true for all $k<m\leq n$.
				By induction $\Phi(m-1,m')\in [h_1,\dots,h_n]$ for all $1 \leq m'\leq m-1$.
				Note that $\Phi(m-1,m')x_{n-m'+1}$ is the sum of all monomials of degree $m$  in variables $x_1,\dots,x_{n-m'+1}$ divisible by $x_{n-m'+1}$.
				Hence, for each $1\leq m'\leq m-1$
				\begin{equation*}
						h_m-\Phi(m-1,1)x_n-\cdots-\Phi(m-1,m'-1)x_{n-m'+2}=\Phi(m,m').
				\end{equation*}
				At each stage of the proof the next $\Phi(k,k)$ is obtained as a sum of $h_k$ and polynomials obtained from $h_1,\dots,h_{k-1}$.
				Hence $[\Phi(1,1),\dots,\Phi(n,n)]$ and $[h_1,\dots,h_n]$ are equal.
			\end{proof}
			
			For integers $0\leq a\leq b$, denote by $h_a^b$ the complete homogeneous polynomial in variables $x_1,\dots,x_b$ of degree $a$.
			Then equation (\ref{eq:OrigIdeal}) can be written as
			\begin{equation}\label{eq:SipleRedusingHomogenious}
				[h_1^n,\dots,h_n^n]=[h_1^n,\dots,h_n^1].
			\end{equation}
			A useful intermediate form of Proposition \ref{thm:monomial sum} is given next.
			
			\begin{prop}\label{prop:Homogeneous-1}
				For each $n\geq 1$,
				\begin{equation*}
					[h_1^n,\dots,h_n^n]=[h_1^n,h_2^{n-1}\dots,h_n^{n-1}].
				\end{equation*}
			\end{prop}
			
			\begin{proof}
				For each $1\leq i\leq n-1$
				\begin{equation*}
					h^n_{i+1}-x_n h^n_i=h^{n-1}_{i+1}.
				\end{equation*}
				We can rearrange the ideal to achieve the desired result by performing the above elimination in sequence on the ideal for $i={n-1}$ to $i=1$.
			\end{proof}
			
			\begin{rmk}\label{remk:SymQotForms}
				By Theorem \ref{thm:monomial sum} and Proposition \ref{prop:Homogeneous-1} eliminating the last variable in $\mathbb{Z}[x_1,\dots,x_n]$,
				by rewriting $h_1$ as $x_n=-x_1-\cdots-x_{n-1}$ gives us
				\begin{equation*}
					\frac{\mathbb{Z}[x_1,\dots,x_n]}{[h_1^n,\dots,h_n^n]}
					\cong\frac{\mathbb{Z}[x_1,\dots,x_{n-1}]}{[h_2^{n-1},\dots,h_n^{n-1}]}
					\cong\frac{\mathbb{Z}[x_1,\dots,x_{n-1}]}{[h_2^{n-1},\dots,h_n^{1}]}.
				\end{equation*}
			\end{rmk}

		\subsection{Basis of representatives and degree-wise number of elements}\label{subsec:BettiNum}
		
			Using Remark \ref{remk:SymQotForms} following from Theorem \ref{thm:monomial sum}, we can deduce an additive basis of the symmetric quotient
			$\frac{\mathbb{Z}[x_1,\dots,x_n]}{[h_1^n,\dots,h_n^n]}$.
			
			\begin{thm}\label{thm:AddBasis}
				The elements $x_1^{a_1}\cdots x_{n-1}^{a_{n-1}}$ such that $0\leq a_i \leq n-i$, form an additive basis of
				$\frac{\mathbb{Z}[x_1,\dots,x_n]}{[h_1^n,\dots,h_n^n]}$.
			\end{thm}
			
			\begin{proof}
				By Theorem \ref{thm:monomial sum}, $\frac{\mathbb{Z}[x_1,\dots,x_n]}{[h_1^n,\dots,h_n^n]}\cong\frac{\mathbb{Z}[x_1,\dots,x_n]}{[h_1^n,\dots,h_n^1]}$.
				$h_1^n$ is the only generator of the ideal in which a summand is divisible by $x_n$ and $x_n$ is the unique summand in $h_1^n$ divisible by $x_n$.
				Hence any elements of $\frac{\mathbb{Z}[x_1,\dots,x_n]}{[h_1^n,\dots,h_n^n]}$ can be expressed with a representative not containing $x_n$
				by replacing $x_n$ with $-h_1^n+x_n$.
				Similarly apart from a multiple of $h_1^n$, $h_2^{n-1}$ is the only generator of the ideal containing a summand divisible $x_{n-1}^2$
				and $h_2^{n-1}$ contains the unique summand $x_{n-1}^2$ divisible by $x^2_{n-1}$.
				Hence any elements of $\frac{\mathbb{Z}[x_1,\dots,x_n]}{[h_1^n,\dots,h_n^n]}$ can be expressed by a representative not containing $x_{n}$ or $x_{n-1}^2$.
				The process can be continued with $h_3^{n-1}$ and $x_{n-3}^3$ through to $h_n^1$ and $x_1^n$ to give the desired result. 
			\end{proof}
			
			\begin{rmk}\label{rmk:AddBasis}
				The symmetry of the variables $x_1,\dots,x_n$ in $h_1^n,\dots,h_n^n$ implies that the basis of Theorem \ref{thm:AddBasis}
				can be chosen using any permutations of $\{ 1,\dots,n \}$.
				That is the elements $x_{\sigma(1)}^{a_1}\cdots x_{\sigma(n-1)}^{a_{n-1}}$ such that $0\leq a_i \leq n-i$ form an additive basis of
				$\frac{\mathbb{Z}[x_1,\dots,x_n]}{[h_1^n,\dots,h_n^n]}$ for any $\sigma\in S_n$.
			\end{rmk}
			
			We now address the problem of counting the number of elements in each degree of $\frac{\mathbb{Z}[x_1,\dots,x_{n+1}]}{[\sigma_1,\dots,\sigma_{n+1}]}$.
			These numbers are the Betti numbers of $H^*(SU(n+1)/T^n)\cong \frac{\mathbb{Z}[x_1,\dots,x_{n+1}]}{[\sigma_1,\dots,\sigma_{n+1}]}$.
			and have been well studied.  
			In particular as a consequence or work of Kostant, Macdonald and Steinberg in 
			\cite{Kostant2009}, \cite{Macdonald1972} and \cite{FiniteReflectionGroups} respectively,
			for simple Lie group $G$ with maximal torus $T$ the following are forms of the Poincar{\'e} series for $G/T$
			\begin{equation*}
				\sum_{w\in W}{t^{2l(w)}}=\prod_{\alpha \in \Phi^{+}}{\frac{1-t^{2ht(\alpha)+2}}{1-t^{2ht(\alpha)}}}=\prod_{i=1}^{l}(1+t^2+\cdots+t^{2m_i}),
			\end{equation*}
			where $W=N_G(T)/T$ The Weyl group of $G$,
			$l(w)$ the length of $w\in W$, $\Phi^+$ is the set of positive roots of $G$,
			$ht(\alpha)$ the hight of $\alpha \in \Phi^{+}$
			and $m_1,\dots,m_l$ the exponents of $G$.
			
			\begin{defn}\label{def:tri}
				Denote by $\triset{n}{k}$ the number of degree $k$ monomials of the form $x_1^{a_1}\cdots x_{n}^{a_{n}}$ such that $0\leq a_i \leq n-i$.
			\end{defn}
			
			\begin{rmk}\label{rmk:tribox}
				Alternatively $\triset{n}{k}$ can be described as the number of ways to construct a $k$ multiset $X$ from elements of $\{1,\dots,n\}$
				such that the element $i$ appear no more than $i$ times in $X$.
			\end{rmk}
			
			It is clear that if $k<0$ or $k>\frac{n(n+1)}{2}$ then $\triset{n}{k}=0$, since in either case such a multiset $X$ cannot exist.
			$\triset{n}{k}$ are known as the Mahonian numbers and were originally defined in terms of the inversion numbers of permutations,
			see for example \cite[page 239]{AdvanceCombintorics}.
			The next two propositions are well known properties of $\triset{n}{k}$, the second gives an inductive rule for computing $\triset{n}{k}$.
			In Theorem \ref{thm:trisetformula} we give an explicit formula for $\triset{n}{k}$, which is similar to the one given in \cite{ArtOfProgramming}.
			Through here in all cases I have given my own proofs.
			
			\begin{prop}\label{prop:trireflect}
			For each $n\geq 0$ and $0\leq k \leq \frac{n(n+1)}{2}$,	
				\begin{equation*}
					\triset{n}{k}=\triset{n}{\frac{n(n+1)}{2}-k}, \;\;\; \sum_{i=0}^{\frac{n(n+1)}{2}}{\triset{n}{i}=(n+1)!}.
				\end{equation*}
			\end{prop}
			\begin{proof}
				Both statements follow from Remark \ref{rmk:tribox}.
				The first is given by the clear bijection between the two multiset descriptions that 
				replaces the number of occurrences of $i$ in the multiset by $i$ minus this number. 
				The second statement follows from the fact that there are $(n+1)!$ ways to form any multiset
				from elements of $\{1,\dots,n\}$ such that the element $i$ appear no more than $i$ times.
			\end{proof}

			\begin{prop}\label{triinduct}
				The numbers $\triset{n}{k}$ for $n\geq 0$ and $0\leq k \leq \frac{n(n+1)}{2}$ are completely determined by the following inductive rule.
				\begin{equation*}
					\triset{0}{k}=
					\begin{cases}
						1, & k=0\\
						0, & k\neq 0
					\end{cases}
				\end{equation*}
				For each $n\geq 1$ and $0\leq k \leq \frac{n(n+1)}{2}$,
				\begin{equation*}
					\triset{n}{k}=\sum_{i=0}^{n}\triset{n-1}{k-i}.
				\end{equation*}
			\end{prop}
			\begin{proof}
				The case when $n=0$ is clear from the definition.
				Using the description from Remark \ref{rmk:tribox}, any $k$ multiset on $1,\dots,n$ satisfying the conditions
				can be obtained from a $(k-i)$-multiset on $1,\dots,n-1$ satisfying the conditions, by adding $i$, $n$'s to the multiset
				for some $0\leq i \leq n$.
			\end{proof}

			\begin{thm}\label{thm:trisetformula}
				For each $n\geq 1$ and $0\leq k \leq \frac{n(n+1)}{2}$,
				\begin{equation*}
					\triset{n}{k}=\multiset{n}{k}+\sum_{a=1}^{n}
					{(-1)^a\sum_{\substack{2\leq i_1<\cdots<i_a\leq n+1 \\ i_1+\cdots+i_a\leq k}}{\multiset{n}{k-i_1-\cdots-i_a}}}.
				\end{equation*}
			\end{thm}
			\begin{proof}
				Beginning with $\multiset{n}{k}$, the number of $k$ multisets on $\{1,\dots,n\}$
				we subtract the number of multisets not satisfying the condition element $i$ appear no more than $i$ times.
				For $2\leq i_1\leq n+1$, $\multiset{n}{k-{i_1}}$ corresponds to the number of multisets in which there are at lest $i_1$ occurrences of the element $i_1-1$.
				However if we subtract
				\begin{equation}\label{eq:stage1}
					\sum_{\substack{2\leq i_1 \leq n+1 \\ i_1\leq k}}{\multiset{n}{k-i_1}}
				\end{equation}
				from $\multiset{n}{k}$, we do not obtain the desired results because we have counted multiple combinations where more than $i$ of element $i$ occur in the multiset.
				For any $2\leq i_1<i_2 \leq n+1$, in equation (\ref{eq:stage1}), the number of multisets in which elements $i_1-1$ and $i_2-1$
				occur more than $i_1$ and $i_2$ times respectively are counted twice.
				Hence subtracting from $\multiset{n}{k}$, 
				\begin{equation}\label{eq:stage2}
					\sum_{\substack{2\leq i_1 \leq n+1 \\ i_1\leq k}}{\multiset{n}{k-i_1}}-\sum_{\substack{2\leq i_1<i_2 \leq n+1 \\ i_1+i_2\leq k}}{\multiset{n}{k-i_1-i_2}}
				\end{equation}
				counts correctly the number of multisets in which for any $2\leq i_1<i_2 \leq n+1$, only elements $i_1-1$ and $i_2-1$ occur more than $i_1$ and $i_2$ times.
				However equation (\ref{eq:stage2}) still counts multisets in which three or more elements occur more times than their value.
				For any $2\leq i_1<i_2<i_3 \leq n+1$ in equation (\ref{eq:stage1}), the number of multisets in which elements $i_1-1$, $i_2-1$ and $i_3-1$
				occur more than $i_1$, $i_2$ and $i_3$ times respectively are counted $\binom{3}{1}=3$ times.
				In
				\begin{equation*}
					\sum_{\substack{2\leq i_1<i_2 \leq n+1 \\ i_1+i_2\leq k}}{\multiset{n}{k-i_1-i_2}}
				\end{equation*}
				the number of multisets in which elements $i_1-1$, $i_2-1$ and $i_3-1$
				occur more than $i_1$, $i_2$ and $i_3$ times respectively is counted $\binom{3}{2}=3$ times and once in $\multiset{n}{k}$.
				Therefore in order to correct the count on triple occurrences we need to add
				\begin{equation*}
					\sum_{\substack{2\leq i_1<i_2<i_3 \leq n+1 \\ i_1+i_2+i_3\leq k}}{\multiset{n}{k-i_1-i_2-i_3}}
				\end{equation*}
				to equation (\ref{eq:stage2}).
				We continue this processes until we have considered combinations of all $n$ variables.
				At each stage, since $\sum_{k=0}^{n}{(-1)^k\binom{n}{k}}=0$, the multiplicity of the number of terms that need to be corrected is always one,
				hence we obtain the desired result.
			\end{proof}

		\subsection{Multiplicative rules}\label{subsec:MultRule}
			
			In this section we try to understand some of the multiplicative structure of the additive basis given in Theorem \ref{thm:AddBasis}. 
			
			\begin{prop}\label{prop:zeros}
				A representative $\gamma_1^{c_1}\cdots\gamma_{n}^{c_n}$ represents the zero class, if for any $1\leq k \leq n$ and $1\leq i_1 < \cdots < i_k \leq n$, 
				\begin{equation*}
					\sum_{j=1}^{k}{c_{i_j}}>\sum_{j=1}^{k}{n-i_j+1}.
				\end{equation*}
			\end{prop}
			\begin{proof}
				By symmetry of the variable $\gamma_1,\dots,\gamma_n$, the arguments of Theorem \ref{thm:monomial sum} and Theorem \ref{thm:AddBasis}
				can be applied to any permutation of the indices.
				Therefore we take a permutation $\phi\in S_n$ and denote $b_j=\phi(j)$ with $\phi(1)=i_1,\dots,\phi(k)=i_k$.
				Using the augment from Theorem \ref{thm:AddBasis}, the representative $\gamma_{b_1}^{c_{b_1}}\cdots\gamma_{b_n}^{c_{b_n}}$
				can be expressed as a sum of monomials 
				$\gamma_{b_1}^{a_1}\cdots \gamma_{b_{n}}^{a_{n}}$ such that $0\leq a_i \leq n-i$.
				In particular using the method given in the proof of Theorem \ref{thm:monomial sum} if $a_i\leq c_i$ then $a_1+\cdots a_{i-1}\geq c_i-a_i+c_1+\cdots+c_{i-1}$.
				So if $\sum_{j=1}^{k}{c_{i_j}}>\sum_{j=1}^{k}{n-i_j+1}$ then the sum of $\gamma_{b_1}^{a_1}\cdots \gamma_{b_{n}}^{a_{n}}$ must be empty.
				Hence $\sum_{j=1}^{k}{c_{i_j}}\leq\sum_{j=1}^{k}{n-i_j+1}$ or the expression is zero.
			\end{proof}
			
			We denote the representative ${\gamma}_1^{n} \cdots {\gamma}_{n-2}^{2}{\gamma}_{n-1}$
			of the unique $\frac{(n+1)n}{2}$ degree class by $\hat{\gamma}_{\emptyset}$.
			Denote by $\hat{\gamma}_{i}$ the class of ${\gamma}_1^{n}\cdots \gamma_{i+1}^{n-i+2} \gamma_{i}^{n-i} \gamma_{i-1}^{n-i} \cdots {\gamma}_{n-1}$
			in $\frac{\mathbb{Z}[\gamma_1,\dots,\gamma_{n}]}{[\sigma_2,\dots,\sigma_{n+1}]}$.
			That is the unique class of degree $\frac{(n+1)n}{2}-1$ represented by the monomial $\hat{\gamma}_{\emptyset}/\gamma_i$.

			\begin{lem}\label{lem:OneMissing}
				For any $1\leq i,j\leq n$,
				\begin{equation*}
					[\hat{\gamma_i}\gamma_j]=
					\begin{cases}
						\;\;\; [0]  &\mbox{if  } j < i \text{ or } j \geq i+2 
						\\
						\;\; [\hat{\gamma}_{\emptyset}] &\text{if  } j=i
						\\
						-[\hat{\gamma}_{\emptyset}] &\text{if  } j=i+1
						\end{cases}
					.
				\end{equation*}
			\end{lem}
			\begin{proof}
				If $j < i$, then $\hat{\gamma_i}\gamma_j=0$ by Proposition \ref{prop:zeros}.
				If $i = j$ then $\hat{\gamma_i}\gamma_j=\hat{\gamma}_{\emptyset}$ by definition.
				So for the rest of the proof assume $j>i$.
				By Theorem \ref{thm:monomial sum} we have $h_{j}^{n-j+1}\in [\sigma_2,\dots,\sigma_{n+1}]$.
				Hence we may replace $\gamma_j^{n-j+2}$ by
				\begin{equation}\label{eq:replace}
					\gamma_j^{n-j+2}-\sum_{1\leq i_1 \leq \cdots \leq i_{n-j+1} \leq j}{\gamma_{i_1} \cdots \gamma_{i_{n-j+2}}}.
				\end{equation}
				If any of the $i_k$ in equation (\ref{eq:replace}) are greater than $i$ or $i_k\neq j$ for $k \geq 2$,
				then multiplying that term by $\hat{\gamma_i}\gamma_j/\gamma_j^{n-j}$ 
				will result in a representative of the zero class by Proposition \ref{prop:zeros}.
				If $j \geq i+1$, again we may replace $\gamma_j^{n-j+2}$ with the expression in (\ref{eq:replace}).
				By Proposition \ref{prop:zeros} the only possible non-zero summand when this is multiplied by $\hat{\gamma}_i/\gamma_j^{n-j+2}$ are
				\begin{equation*}
				-\hat{\gamma}_i\gamma_{j-1}-\cdots-\hat{\gamma}_{i}\gamma_{i+1}-\hat{\gamma}_{\phi}.
				\end{equation*}
				If $j=i-1$, then this is just $\hat{\gamma}_{\emptyset}$.
				If $j \geq i+2$, then replace $-\hat{\gamma}_i\gamma_{j-1}$ with
				\begin{equation*}
					\hat{\gamma}_i\gamma_{j-2}+\cdots+\hat{\gamma}_{i}\gamma_{i+1}+\hat{\gamma}_{\phi},
				\end{equation*}
				which cancels with the other terms.
			\end{proof}

\newpage
\section{Cohomology of the free loop space of the complete flag manifold of $SU(n)$}\label{sec:FreeLoopSU(n+1)/Tn}
		
		In this chapter we investigate the cohomology of the free loop space of $SU(n+1)/T^n$ by studying the Leary-Serre spectral sequence associated to the free
		loop space fibration of $\Lambda(SU(n+1)/T^n)$.
		In particular in Section \ref{sec:L(SU(3)/T^2)}, we give the algebra structure of the $E_\infty$-page in the case when $n=2$
		and the module structure of $H^*(\Lambda(SU(3)/T^2);\mathbb{Z})$.
		
		\subsection{Differentials in the path space spectral sequence}\label{sec:evalSS}
			
					In this section we study $H^{*}(\Lambda(SU(n+1)/T^n);\mathbb{Z})$ for $n \geq 1$.
			The case when $n=0$ being trivial as $SU(1)$ is a point.
			The approach of the argument is similar to that of \cite{cohololgy_Lprojective},
			in which the cohomology of the free loop spaces of spheres and complex projective space are calculated using spectral sequence techniques. 
			However the details in the case of the complete flag of the special unitary group are considerably more complex.

			For any space $X$, the map $eval \colon Map(I,X) \to X\times X$ is given by $\alpha \mapsto (\alpha(0),\alpha(1))$.
			It can be shown directly that $eval$ is a fibration with fiber $\Omega X$.
			In this section we compute the differentials in the cohomology Serre spectral sequence of this fibration for the case $X=SU(n+1)/T^n$.
			The aim is to compute $H^{*}(\Lambda(SU(n+1)/T^n);\mathbb{Z})$.
			The map $eval \colon \Lambda X \to X$ given by evaluation at the base point of a free loop is also a fibration with fiber
			$\Omega X$.
			This is studied in section \ref{sec:diff} by considering a map of fibrations from the free loop fibration for $SU(n+1)/T^n$ to the evaluation fibration
			and hence the induced map on spectral sequences. For the rest of this section we consider the fibration
			
			\begin{equation}\label{eq:evalfib}
				\Omega (SU(n+1)/T^n) \to Map(I,SU(n+1)/T^n) \xrightarrow{eval} SU(n+1)/T^n\times SU(n+1)/T^n. 
			\end{equation}
			
			By extending the fibration $T^n \to SU(n+1) \to SU(n+1)/T^n$, we obtain the homotopy fibration sequence
			
			\begin{equation}\label{eq:SU/Tfib}
			\Omega(SU(n+1)) \to \Omega(SU(n+1)/T^n) \to T^n \to SU(n+1).
			\end{equation}
			
			It is well known see \cite{CohomologyOmega(G/U)}, that the furthest right map above of the inclusion of the maximal torus into $SU(n+1)$ is null-homotopic.
			Hence there is a homotopy section $T^n \to \Omega(SU(n+1)/T^n)$.
			Therefore, as the fibration $\Omega(SU(n+1)/T^n) \to T^n$ is a principle fibration, so $\Omega(SU(n+1)/T^n) \simeq \Omega(SU(n+1)) \times T^n$.
			Using the K\"{u}nneth formula and Theorem \ref{thm:H*(SU(n))} we obtain the algebra isomorphism
			
			\begin{center}
			$H^*(\Omega(SU(n+1)/T^n);\mathbb{Z}) \cong H^*(\Omega(SU(n+1);\mathbb{Z}) \otimes H^*(T^n;\mathbb{Z})
			\cong \Gamma_{\mathbb{Z}}[x_2,x_4,\dots,x_{2n}] \otimes \Lambda_{\mathbb{Z}}(y_1,\dots,y_n)$,
			\end{center}
			where $\Gamma_{\mathbb{Z}}[x_2,x_4,\dots,x_{2n}]$ is the integral divided polynomial algebra on $x_2,\dots,x_{2n}$
			with $|x_i|=i$ for each $i=2,\dots,2n$.
			$\Lambda(y_1,\dots,y_n)$ is an exterior algebra generated by $y_1,\dots,y_n$
			with $|y_j|=1$ for each $j=1,\dots,n$.
			It is well known that
			\begin{equation*}
				Map(I,SU(n+1)/T^n)\simeq SU(n+1)/T^n,
			\end{equation*}
			therefore by Theorem \ref{thm:H*SU/T} 
			all cohomology algebras of spaces in fibration (\ref{eq:evalfib}) are known.
			By studying the long exact sequence of homotopy groups associated to the fibration $T^n \to SU(n+1) \to SU(n+1)/T^n$,
			we obtain that $SU(n+1)/T^n$ hence $SU(n+1)/T^n \times SU(n+1)/T^n$ are simply connected.
			Therefore the cohomology Serre spectral sequence of fibration (\ref{eq:evalfib}), which we denote by $\{E_r,d^r\}$, 
			converges to $H^*(SU(n+1)/T^n;\mathbb{Z})$ with $E_2$-page $E^{p,q}_2=H^p(SU(n+1)/T^n\times SU(n+1)/T^n;H^q(\Omega(SU(n+1)/T^n);\mathbb{Z}))$, both of which are known.
			In the following arguments we will use the notation
			
			\begin{center}
			$H^*(Map(I,SU(n+1)/T^n);\mathbb{Z}) \cong \frac{\mathbb{Z}[\lambda_1,\dots,\lambda_{n+1}]}{[\sigma^{\lambda}_1,\dots,\sigma^{\lambda}_{n+1}]}$
			\end{center}
			and
			\begin{center}
			$H^*(SU(n+1)/T^n \times SU(n+1)/T^n;\mathbb{Z}) \cong 
			\frac{\mathbb{Z}[\alpha_1,\dots,\alpha_{n+1}]}{[\sigma^{\alpha}_1,\dots,\sigma^{\alpha}_{n+1}]} \otimes
			\frac{\mathbb{Z}[\beta_1,\dots,\beta_{n+1}]}{[\sigma^{\beta}_1,\dots,\sigma^{\beta}_{n+1}]}
			,$
			\end{center}
			where $|\alpha_i|=|\beta_i|=|\lambda_i|=2$ for each $i=1,\dots,n+1$ and $\sigma^{\lambda}_i,\sigma^{\alpha}_i$
			and $\sigma^{\beta}_i$ are the elementary symmetric polynomials in $\lambda_i,\alpha_i$ and $\beta_i$, respectively.
			
			
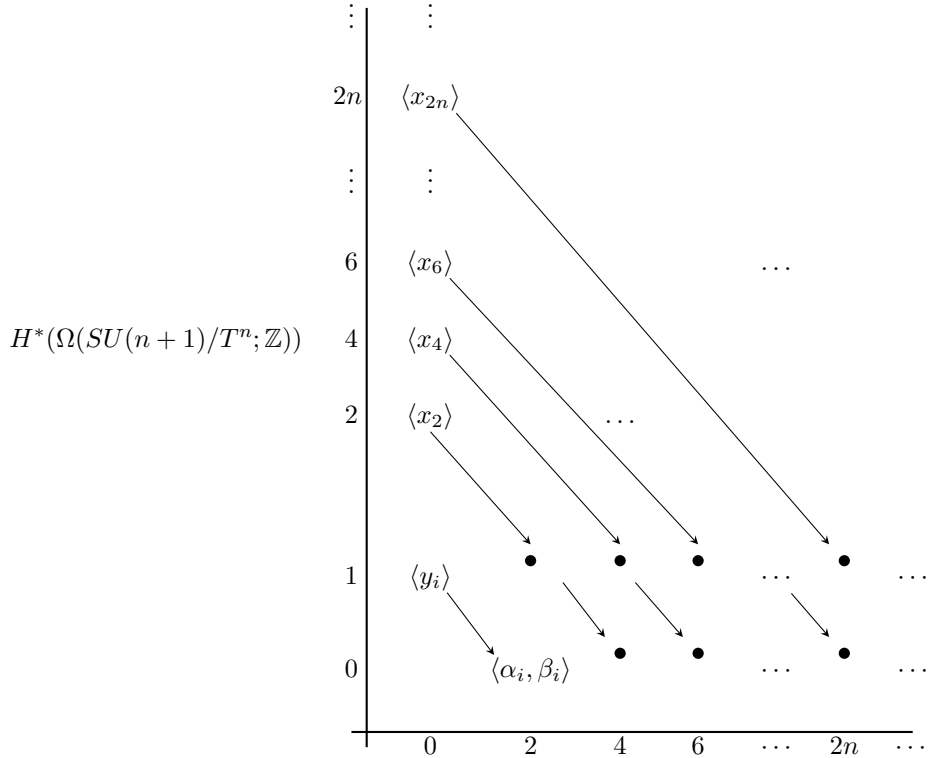
\begin{center}
				\begin{tikzpicture}
					\matrix (m) [matrix of math nodes,
						nodes in empty cells,nodes={minimum width=5ex,
						minimum height=5ex,outer sep=-5pt},
						column sep=1ex,row sep=1ex]{
																									&	 \vdots		&	 \vdots	 							&		 														 & 		  &						   &			 &					& \\
																									&	  2n\;  	& \langle x_{2n} \rangle&		 														 & 		  &						   &    	 &    			& \\
																									&	 \vdots   &  \vdots	 							&		 														 & 		  &						   &   		 &  			  & \\
																									&			6     &   \langle x_6 \rangle &		 														 & 		  & 						 &\dots &   			  & \\
						H^{*}(\Omega(SU(n+1)/T^n;\mathbb{Z})) &			4     &   \langle x_4\rangle  &		 														 & 		  & 						 & 			 &					& \\
																									&			2     &   \langle x_2\rangle  &																 &\dots& 		 				 & 			 &   			  & \\
																									&			      &  			   							&		 														 & 		  & 		 				 & 			 &          & \\
																									&			1			&		\langle y_i\rangle	&	\;\;\;			\lcdot	\;\;\;		 &\lcdot&\;\;\lcdot\;\;& \cdots& \lcdot		& \cdots \\	
																									&		  0     &  			   							& \langle\alpha_i,\beta_i\rangle &\lcdot&\lcdot				 & \cdots&  \lcdot  & \cdots \\						
																									&\quad\strut &     0    							&				2												 &		4	&	 6	 				 & \cdots&     2n   & \cdots \strut \\};
					\draw[-stealth] (m-2-3.south east) -- (m-8-8.north west);
					\draw[-stealth] (m-5-3.south east) -- (m-8-5.north);
					\draw[-stealth] (m-4-3.south east) -- (m-8-6.north);
					\draw[-stealth] (m-6-3.south) -- (m-8-4.north);
					\draw[-stealth] (m-8-3.south east) -- (m-9-4.north west);
					\draw[-stealth] (m-8-4.south east) -- (m-9-5.north west);
					\draw[-stealth] (m-8-5.south east) -- (m-9-6.north west);
					\draw[-stealth] (m-8-7.south east) -- (m-9-8.north west);
				\draw[thick] (m-1-2.east) -- (m-10-2.east) ;
				\draw[thick] (m-10-2.north) -- (m-10-9.north) ;
				\end{tikzpicture}
				\label{fig:evalSS}
				\end{center}
				\begin{center}
					$\;\;\;\;\;\;\;\;\;\;\;\;\;\;\;\;\;\;\;\;\;\;\;\;\;\;\;\;\;\;\;\;\;\;\;\;\;\;\;\;\; H^{*}(SU(n+1)/T^n\times SU(n+1)/T^n;\mathbb{Z})$
				\end{center}
				\begin{center}
				\captionof{figure}{Generators in integral cohomology Leray-Serre spectral sequence $\{E_r,d^r\}$ converging to $H^{*}(Map(I,SU(n+1));\mathbb{Z})$.}
				\end{center}
			
				In the remainder of this section we will describe explicitly the images of differentials shown in Figure \ref{fig:evalSS}
				and show that all other differential not generated by these differentials using the Leibniz rule are zero.
				It will often be useful to use the alternative basis
				\begin{center}
					$v_i=\alpha_i-\beta_i$ and $u_i=\beta_i$
				\end{center}
				for $H^{*}(SU(n+1)/T^n\times SU(n+1)/T^n;\mathbb{Z})$, where $i=1,\dots,n+1$. 
				The following lemma determines completely the $d^2$ differential on $E_2^{*,1}$.	
			
				\begin{lem}\label{lem:E^2_{*,1}d^2}
					With the notation above, in the cohomology Leray-Serre spectral sequence of fibration (\ref{eq:evalfib}),
					there is a choice of basis $y_1,\dots,y_n$ such that
					\begin{center}
							$d^2(y_i)=v_i$
					\end{center}
					for each $i=1,\dots,n$.
				\end{lem}
				
				\begin{proof}
					We have the homotopy commutative diagram
					
					\begin{equation*}\label{fig:evalcd}
						\xymatrix{
							{SU(n+1)/T^n} \ar[r]^(.37){\Delta}								   & {SU(n+1)/T^n\times SU(n+1)/T^n} \ar@{=}[d] \\
							{Map(I,SU(n+1)/T^n)} \ar[r]_(.42){eval} \ar[u]_{p_0} & {SU(n+1)/T^n\times SU(n+1)/T^n} ,}
					\end{equation*}
					where $p_0$, given by $\psi \mapsto \psi(0)$, is a homotopy equivalence and $\Delta$ is the diagonal map.
					As the cup product is induced by the diagonal map $eval^*$ has the same image as the cup product.
					For dimensional reasons, $d^2$ is the only possible non-zero differential ending at any $E_*^{2,0}$ and no non-zero differential have domain in any $E_*^{2,0}$.
					Therefore in order for the spectral sequence to converge to $H^{*}(Map(I,SU(n+1)/T^n))$,
					the image of $d^2 \colon E_2^{0,1} \to E_2^{2,0}$ must be the kernel of the cup product on $H^*(SU(n+1)/T^n\times SU(n+1)/T^n;\mathbb{Z})$, 
					which is generated by $v_1,\dots,v_n$.
				\end{proof}
				
				\begin{rmk}\label{rmk:unique}
				The only remaining differentials on generators left to determine are those with domain in
				$\langle x_2,x_4\dots,x_{2n} \rangle$, on some page $E_r$ for $r \geq 2$.
				For dimensional reasons, the elements $x_2,x_4,\dots,x_{2n}$ cannot be the image of any differential.
				By Lemma \ref{lem:E^2_{*,1}d^2}, the generators $u_1,\dots,u_n$ must survive to the $E_{\infty}$-page,
				so generators $x_2,x_4,\dots,x_{2n}$ cannot.	
				This is due to dimensional reasons combined with the fact that the spectral sequence must converge to $H^*(SU(n+1)/T^n)$.
				Now assume inductively for each $i=1,\dots,n$ that for each $1\leq j<i$, $d^{2j}$ is constructed.
				For dimensional reasons and due to all lower rows except $E_r^{*,2}$ and $E_r^{*,1}$ being annihilated by differentials already determined at lower values of
				$1\leq j<i$, the only possible non-zero differential beginning at $x_{2i}$,
				is $d^{2i}:E_{2i}^{0,2i}\to E_{2i}^{2i,1}$.
				The image of each of the differentials $d^{2i}$ will therefore be a unique class in $E_{2i}^{2i,1}$
				in the kernel of $d^2$ not already contained in the image of any $d^r$ for $r<2i$.
				\end{rmk}
					
				We have $d^2(u_i)=0=d^2(v_i)$ and by Lemma \ref{lem:E^2_{*,1}d^2} we may assume that $d^2(y_i)=v_i$ for each $i=1,\dots,n$.
				All non-zero generators $\gamma \in E_2^{*,1}$ can be expressed in form
				\begin{center}
					$\gamma = y_k u_{i_1} \cdots u_{i_s} v_{j_1} \cdots v_{j_t}$
				\end{center}
				for some $1\leq k \leq n$, $1\leq i_1<\cdots<i_s\leq n$ and $1\leq j_1<\cdots<j_t\leq n$.
				Therefore $d^2(\gamma)$ is zero only if it is contained in
				$[\sigma^{\alpha}_1,\dots,\sigma^{\alpha}_{n+1}, \sigma^{\beta}_1,\dots,\sigma^{\beta}_{n+1}]$.
				Hence it is important to understand the structure of the symmetric polynomials
				$\sigma^{\alpha}_1,\dots,\sigma^{\alpha}_{n+1}, \sigma^{\beta}_1,\dots,\sigma^{\beta}_{n+1}$.
				$\sigma^{\alpha}_1$ and $\sigma^{\beta}_1$ simply express $\alpha_{n+1}$ and $\beta_{n+1}$ in terms of the other generators of the ideal.
				Lemma \ref{lem:n+1elim} describes explicitly what the structure of 
				$\sigma^{\alpha}_2,\dots,\sigma^{\alpha}_{n+1}, \sigma^{\beta}_2,\dots,\sigma^{\beta}_{n+1}$
				is in terms of $\alpha_1,\dots,\alpha_{n}$
				and $\beta_1,\dots,\beta_{n}$.
				
				Using the next two lemmas, we will determine how $\sigma_l^\alpha$ and $\sigma_l^\beta$ lie in the image of $d^2$ and so determine other differentials.
				For each $n \geq 1$, $2 \leq l \leq n+1$ and $1\leq m\leq l$, define element $s_{l,n}^m$ of $E_2^{2l-1,1}$ by
				\begin{equation*}
					s_{l,n}^m=\sum_{\substack{1\leq i_1<\cdots<i_m\leq n\\ 1\leq i_{m+1}<\cdots<i_l\leq n\\ i_j \neq i_{j'} \; for \; j \neq j'}}
					{y_{i_1}v_{i_2}\cdots v_{i_m} u_{i_{m+1}}\cdots u_{i_l}}.
				\end{equation*}
				Define also $S_{l,n}=s_{l,n}^1+\cdots+s_{l,n}^l$.
			
				\begin{lem}\label{lem:term1}
				For each $n \geq 1$, $2 \leq l \leq n+1$ and $1\leq m\leq l$,
				\begin{equation*}
					d^{2}(S_{l,n})
					=\sum_{1\leq i_1<\cdots<i_l\leq n}{\alpha_{i_1}\cdots\alpha_{i_l}}-\sum_{1\leq i_1<\cdots<i_l\leq n}{\beta_{i_1}\cdots\beta_{i_l}}.
				\end{equation*}
				\end{lem}
				\begin{rmk}\label{rmk:unique1}
				In the course of the proof of the Lemma it is shown that
				\begin{equation*}
				\sum_{1\leq i_1<\cdots<i_l\leq n}{\alpha_{i_1}\cdots\alpha_{i_l}}-\sum_{1\leq i_1<\cdots<i_l\leq n}{\beta_{i_1}\cdots\beta_{i_l}}
				\end{equation*}
				is up to sign the unique generator for elements $E^2_{2l,0}$ in the image of $d^2$ containing either the terms
				$\sum_{1\leq i_1<\cdots<i_l\leq n}{\alpha_{i_1}\cdots\alpha_{i_l}}$ or $\sum_{1\leq i_1<\cdots<i_l\leq n}{\beta_{i_1}\cdots\beta_{i_l}}$.
				\end{rmk}
				
				\begin{proof}
				First note that
				\begin{changemargin}{-6mm}{40mm}
				\begin{flalign*}
					d^{2}(s_{l,n}^m) &= \sum_{\substack{1\leq i_1<\cdots<i_m\leq n\\ 1\leq i_{m+1}<\cdots<i_l\leq n\\ i_\sigma \neq i_{\sigma'} \; for \; \sigma \neq \sigma'}}
					{v_{i_1}v_{i_2}\cdots v_{i_m} u_{i_{m+1}}\cdots u_{i_l}} \\
					&=\sum_{\substack{1\leq i_1<\cdots<i_m\leq n\\ 1\leq i_{m+1}<\cdots<i_l\leq n\\ i_\sigma \neq i_{\sigma'} \; for \; \sigma \neq \sigma'}}
					{(\alpha_{i_1}-\beta_{i_1})\cdots (\alpha_{i_m}-\beta_{i_m}) \beta_{i_{m+1}}\cdots \beta_{i_l}} \\
					&=\sum_{\substack{0\leq t \leq m\\ 1\leq i_1<\cdots<i_m\leq n\\ 1\leq i_{m+1}<\cdots<i_l\leq n\\ i_\sigma \neq i_{\sigma'} \; for \; \sigma \neq \sigma'}}
					{(-1)^{m-t} \alpha_{i_1}\cdots \alpha_{i_t} \beta_{i_{t+1}}\cdots \beta_{i_l}} \\
					&=\sum_{\substack{0\leq t \leq m\\ 1\leq i_1<\cdots<i_t\leq n\\ 1\leq i_{t+1}<\cdots<i_l\leq n\\ i_\sigma \neq i_{\sigma'} \; for \; \sigma \neq \sigma'}}
					{(-1)^{m-t} \binom{l-t}{m-t} \alpha_{i_1}\cdots \alpha_{i_t} \beta_{i_{t+1}}\cdots \beta_{i_l}}.
				\end{flalign*}
				\end{changemargin}
				
				For each $1\leq m \leq l$, element $d^2(s_{l,n}^m)$ contains a term $\alpha_{i_1}\cdots \alpha_{i_t} \beta_{i_{t+1}}\cdots \beta_{i_l}$
				only when $0 \leq t \leq m$.
				None of the $d^2(s_{l,n}^m)$ are zero as they all at least contain a non-zero term of the form $\alpha_{i_1}\beta_{i_{2}}\cdots \beta_{i_l}$
				which is not contained in $[\sigma^\alpha_1,\dots,\sigma^\alpha_{l+1},\sigma^\beta_1,\dots,\sigma^\beta_{l+1}]$.
				The differential $d^{2}$ preserves the indices $i_1,\dots,i_l$.
				Hence the $d^{2}$ image of an element in $E_2^{2l-1,1}$ is given in terms of elements of the form
				\begin{flalign*}
				\sum_{\substack{1\leq i_1<\cdots<i_t\leq n\\ 1\leq i_{t+1}<\cdots<i_l\leq l\\ i_\sigma \neq i_{\sigma'} \; for \; \sigma \neq \sigma'}}
				{\alpha_{i_1}\cdots \alpha_{i_t} \beta_{i_{t+1}}\cdots \beta_{i_l}}
				\end{flalign*}
				if and only if it is a sum of elements of the form $s_{l,n}^m$ for $1\leq m\leq l$.
				As $m$ increases from $1$ to $l$, each successive $d^2(s^m_{l,n})$ contains a new term of the form
				\begin{flalign*}
				\sum_{\substack{1\leq i_1<\cdots<i_m\leq n\\ 1\leq i_{m+1}<\cdots<i_l\leq l\\ i_\sigma \neq i_{\sigma'} \; for \; \sigma \neq \sigma'}}
				{\alpha_{i_1}\cdots \alpha_{i_m} \beta_{i_{m+1}}\cdots \beta_{i_l}},
				\end{flalign*}
				which did not appear in any previous $d^2(s^i_{l,n})$ for $i<m$.
				When $m=l$, this new term is $\sum_{1\leq i_1 < \cdots< i_l \leq n}{\alpha_{i_1},\dots,\alpha_{i_l}}$.
				In order to cancel all terms not of the form $\sum_{1\leq i_1<\cdots<i_l\leq n}{\alpha_{i_1},\dots,\alpha_{i_l}}$
				or $\sum_{1\leq i_1<\cdots<i_l\leq n}{\beta_{i_1},\dots,\beta_{i_l}}$,
				we need a sum $c_1d^2(s^1_{l,n})+\cdots+c_ld^2(s^l_{l,n})$ where $c_1,\dots,c_l \in \mathbb{Z}\setminus\{0\}$.
				Since each successive $d^2(s^m_{l,n})$ contains a new term, the choice of $c_1=1$ uniquely determines $c_2,\dots,c_l$.
				Recall from the calculation at the beginning of the proof that if $d^2(s^m_{l,n})$ contains terms of the form
				\begin{flalign*}
				\sum_{\substack{1\leq i_1<\cdots<i_t\leq n\\ 1\leq i_{t+1}<\cdots<i_l\leq l\\ i_\sigma \neq i_{\sigma'} \; for \; \sigma \neq \sigma'}}
				{\alpha_{i_1}\cdots \alpha_{i_t} \beta_{i_{t+1}}\cdots \beta_{i_l}}
				\end{flalign*}
				and the constant multiplied by each of these terms is $(-1)^{m-t}\binom{l-t}{m-t}$.
				It is well know that the alternating sum of rows greater than $0$ in Pascal's triangle is zero,
				more precisely this is $\sum_{i=0}^{n}{(-1)^{n-i}\binom{n}{i}}=0$ for $n\geq1$.
				Hence $c_2,\dots,c_l$ are also $1$ and therefore $S_{l,n}$ is the unique sum in $s^m_{l,n}$ such that $S_{l,n}$ has no cancellation but $d^2(S_{l,n})$
				can be expressed with a single term of the form 
				$\sum_{1\leq i_1<\cdots<i_l\leq n}{\alpha_{i_1}\cdots\alpha_{i_l}}$, only containing other terms of the form
				$\sum_{1\leq i_1<\cdots<i_l\leq n}{\beta_{i_1}\cdots\beta_{i_l}}$.
				Finally the constant for the $\sum_{1\leq i_1<\cdots<i_l\leq n}{\beta_{i_1}\cdots\beta_{i_l}}$ terms in $d^2(S_{l,n})$ is $-1$ as 
				\begin{center}
				$(-1)^l\binom{l}{0}+(-1)^{l-1}\binom{l}{1}+\cdots+(-1)^i\binom{l}{l-1}=\sum_{i=0}^{l-1}{(-1)^{l-i}\binom{l}{i}}=\sum_{i=0}^{l}{(-1)^{l-i}\binom{l}{i}}-1=-1$.
				\end{center}
				\end{proof}

					For each $n \geq 2$, $2 \leq l \leq n+1$ and $0 \leq m \leq l-2$, define elements $\tilde{s}_{l,n}^m,\tilde{s}_{l,n}^{'m}$ of $E_2^{2l,1}$ by
					\begin{equation*}
						\tilde{s}_{l,n}^m=\sum_{\substack{1\leq k \leq n\\ 1\leq i_1<\cdots<i_m\leq n\\ 1\leq i_{m+1}<\cdots<i_{l-2}\leq n\\
						k \neq i_j \neq i_{j'} \; for \; j \neq j'}}
						{y_{k}v_k v_{i_1}\cdots v_{i_m} u_{i_{m+1}}\cdots u_{i_{l-2}}}, \;\;\;\;
						\tilde{s}_{l,n}^{'m}=\sum_{\substack{1\leq k \leq n\\ 1\leq i_1<\cdots<i_m\leq n\\ 1\leq i_{m+1}<\cdots<i_{l-2}\leq n\\
						k \neq i_j \neq i_{j'} \; for \; j \neq j'}}
						{y_{k}u_k v_{i_1}\cdots v_{i_m} u_{i_{m+1}}\cdots u_{i_{l-2}}}.
					\end{equation*}
					For each $1\leq m \leq l-2$ and $3 \leq l \leq n+1$, define
					\begin{equation*}
						\tilde{s}_{l,n}^{''m}=\sum_{\substack{1\leq k \leq n\\ 1\leq i_1<\cdots<i_m\leq n\\ 1\leq i_{m+1}<\cdots<i_{l-2}\leq n\\
						k \neq i_j \neq i_{j'} \; for \; j \neq j'}}
						{u_k^2 y_{i_1} v_{i_2}\cdots v_{i_m} u_{i_{m+1}}\cdots u_{i_{l-2}}},
					\end{equation*}
					in addition set $\tilde{s}_{2,n}^{''m}=0$.
					Define also $\tilde{S}_{l,n}=\tilde{s}_{l,n}^{0}+\cdots+\tilde{s}_{l,n}^{l-2}$, $\tilde{S}^{'}_{l,n}=\tilde{s}_{l,n}^{'0}+\cdots+\tilde{s}_{l,n}^{'l-2}$
					and $\tilde{S}^{''}_{l,n}=\tilde{s}_{l,n}^{''1}+\cdots+\tilde{s}_{l,n}^{''l-2}$ with $\bar{S}_{l,n}=\tilde{S}_{l,n}+2\tilde{S}^{'}_{l,n}+\tilde{S}^{''}_{l,n}$.
					
					\begin{lem}\label{lem:term2}
					For each $n \geq 1$, $2 \leq l \leq n-1$,
					\begin{flalign*}
						d^2(\bar{S}_{l,n})=d^{2}(\tilde{S}_{l,n}+2\tilde{S}^{'}_{l,n}+\tilde{S}^{''}_{l,n})
						=\sum_{\substack{1\leq i_1<\cdots<i_l\leq n \\ 1 \leq k\leq n, i_k \neq i_{j}}}
						{\alpha_k^2\alpha_{i_1}\cdots\alpha_{i_{l-2}}}
						-\sum_{\substack{1\leq i_1<\cdots<i_l\leq n \\ 1 \leq k\leq n, i_k \neq i_{j}}}
						{\beta_k^2\beta_{i_1}\cdots\beta_{i_{l-2}}}.
					\end{flalign*}
					\end{lem}
					\begin{rmk}\label{rmk:unique2}
						In the course of the proof of the Lemma it is shown that
						\begin{flalign*}
							\sum_{\substack{1\leq i_1<\cdots<i_l\leq n \\ 1 \leq k\leq n, i_k \neq i_{j}}}
							{\alpha_k^2\alpha_{i_1}\cdots\alpha_{i_{l-2}}}
							-\sum_{\substack{1\leq i_1<\cdots<i_l\leq n \\ 1 \leq k\leq n i_k \neq i_{j}}}
							{\beta_k^2\beta_{i_1}\cdots\beta_{i_{l-2}}},
						\end{flalign*}
						is up to sign the unique generator for elements $E_2^{2l,0}$ in the image of $d^2$ containing either the terms
						\begin{flalign*}
							\sum_{\substack{1\leq i_1<\cdots<i_l\leq n \\ 1 \leq k\leq n, i_k \neq i_{j}}}
							{\alpha_k^2\alpha_{i_1}\cdots\alpha_{i_{l-2}}}
							 \;\; \textrm{or} \;\; 
							\sum_{\substack{1\leq i_1<\cdots<i_l\leq n \\ 1 \leq k\leq n, i_k \neq i_{j}}}
							{\beta_k^2\beta_{i_1}\cdots\beta_{i_{l-2}}}.
						\end{flalign*}	
					\end{rmk}
				
				\begin{proof}
				The proof of the lemma will in places be similar to the proof of Lemma \ref{lem:term1}, hence in these parts details will be omitted.
				First note that for each $0 \leq m \leq l-2$,
				
				\begin{flalign*}
					d^{2}(\tilde{s}_{l,n}^{m}) &= \sum_{\substack{1\leq i_1<\cdots<i_t\leq n\\ 1\leq i_{t+1}<\cdots<i_{l-2}\leq n\\ 
					0 \leq t \leq m, 0 \leq k \leq n  \\ k \neq i_j \neq i_{j'} \; for \; j \neq j'}}
					{(\alpha_k^2-2\alpha_k\beta_k+\beta_k^2)(-1)^{m-t} \binom{m-t}{l-t-2} \alpha_{i_1}\cdots \alpha_{i_t} \beta_{i_{t+1}}\cdots \beta_{l-2}},
				\end{flalign*}
				
				\begin{flalign*}
					d^{2}(\tilde{s}_{l,n}^{'m}) &= \sum_{\substack{1\leq i_1<\cdots<i_t\leq n\\ 1\leq i_{t+1}<\cdots<i_{l-2}\leq n\\ 
					0 \leq t \leq m, 0 \leq k \leq n  \\ k \neq i_j \neq i_{j'} \; for \; j \neq j'}}
					{(\alpha_k\beta_k-\beta_k^2)(-1)^{m-t} \binom{m-t}{l-t-2} \alpha_{i_1}\cdots \alpha_{i_t} \beta_{i_{t+1}}\cdots \beta_{l-2}},
				\end{flalign*}
				
				\begin{flalign*}
					d^{2}(\tilde{s}_{l,n}^{''m}) &= \sum_{\substack{1\leq i_1<\cdots<i_t\leq n\\ 1\leq i_{t+1}<\cdots<i_{l-2}\leq n\\ 
					0 \leq t \leq m, 0 \leq k \leq n  \\ k \neq i_j \neq i_{j'} \; for \; j \neq j'}}
					{\beta_k^2(-1)^{m-t} \binom{m-t}{l-t-2} \alpha_{i_1}\cdots \alpha_{i_t} \beta_{i_{t+1}}\cdots \beta_{l-2}}.
				\end{flalign*}
				Using the same argument given in Lemma \ref{lem:term1}, we obtain $d^{2}(\tilde{S}_{l,n}), d^{2}(\tilde{S}^{'}_{l,n})$ and $d^{2}(\tilde{S}^{''}_{l,n})$.
				The only difference is for $d^{2}(\tilde{S}_{l,n})$ and $d^{2}(\tilde{S}^{'}_{l,n})$, where we begin with $\tilde{s}_{l,n}^{0}$ and $\tilde{s}_{l,n}^{'0}$
				rather than $\tilde{s}_{l,n}^{1}$ and $\tilde{s}_{l,n}^{'1}$.
				Hence the $\beta_{i_1},\dots,\beta_{i_{l-2}}$ terms give the alternating sum over the entire row of Pascal's triangle, so all such terms cancel.
				Therefore
				\begin{flalign*}
					d^{2}(\tilde{S}_{l,n}) &= \sum_{\substack{1\leq i_1<\cdots<i_{l-2}\leq n\\ 
					0 \leq k \leq n, k \neq i_\sigma}}
					{(\alpha_k^2-2\alpha_k\beta_k+\beta_k^2)\alpha_{i_1}\cdots\alpha_{i_{l-2}}},
				\end{flalign*}
				
				\begin{flalign*}
					d^{2}(\tilde{S}_{l,n}^{'}) &= \sum_{\substack{1\leq i_1<\cdots<i_{l-2}\leq n\\ 
					0 \leq k \leq n, k \neq i_\sigma}}
					{(\alpha_k\beta_k-\beta_k^2)\alpha_{i_1}\cdots\alpha_{i_{l-2}}},
				\end{flalign*}
				
				\begin{flalign*}
					d^{2}(\tilde{S}_{l,n}^{''}) &= \sum_{\substack{1\leq i_1<\cdots<i_{l-2}\leq n\\ 
					0 \leq k \leq n, k \neq i_\sigma}}
					{\beta_k^2(\alpha_{i_1}\cdots\alpha_{i_{l-2}}-\beta_{i_1}\cdots \beta_{i_{l-2}}}).
				\end{flalign*}
				In addition, as Remark \ref{rmk:unique1} was respected in Lemma \ref{lem:term1}, so the statements are maintained in the expressions above.
				Finally calculating $d^2(\tilde{S}_{l,n}+2\tilde{S}^{'}_{l,n}+\tilde{S}^{''}_{l,n})$ using the expressions above proves the lemma.
				\end{proof}

				\begin{thm}\label{thm:finaldiff}
				For each $n \geq 1$ and $2 \leq l \leq n+1$, in the spectral sequence $\{E_n,d^n\}$,
				up to class representative in $E_2^{2l,1}$, we have
				\begin{center}
					$d^{2(l-1)}(x_{2(l-1)})=(1-l)S_{l,n}-\bar{S}_{l,n}$
				\end{center}		
				using the notation preceding Lemmas \ref{lem:term1} and \ref{lem:term2}.
				More precisely, for $3 \leq l \leq n+1$
				\begin{changemargin}{-7mm}{40mm}
				\begin{flalign*}
					d^{2(l-1)}(x_{2(l-1)})= 
					(1-l)\sum_{\substack{1 \leq m \leq l \\ 1\leq i_1<\cdots<i_m\leq n\\ 1\leq i_{m+1}<\cdots<i_{l}\leq n
					\\ i_j \neq i_{j'} \; for \; j \neq j'}}
				{y_{i_1}v_{i_2}\cdots v_{i_m} u_{i_{m+1}}\cdots u_{i_l} 
				-\sum_{\substack{1 \leq m \leq l, 1 \leq k \leq n \\ 1\leq i_1<\cdots<i_m\leq n\\ 1\leq i_{m+1}<\cdots<i_{l-2}\leq n
					\\ k \neq i_j \neq i_{j'} \; for \; j \neq j'}}
				u_k^2 y_{i_1} v_{i_2}\cdots v_{i_m} u_{i_{m+1}}\cdots u_{i_{l-2}}}
					\\ -
					\sum_{\substack{0 \leq m \leq l, 1 \leq k \leq n \\ 1\leq i_1<\cdots<i_m\leq n\\ 1\leq i_{m+1}<\cdots<i_{l-2}\leq n
					\\ k \neq i_j \neq i_{j'} \; for \; j \neq j'}}
					(2{y_{k}u_k v_{i_1}\cdots v_{i_m} u_{i_{m+1}}\cdots u_{i_{l-2}}}+{y_{k}v_k v_{i_1}\cdots v_{i_m} u_{i_{m+1}}\cdots u_{i_{l-2}}})
				\end{flalign*}
				\end{changemargin}
				and
				\begin{flalign*}
					d^{2}(x_{2})=
					-\sum_{1 \leq i_1<i_2 \leq n}{y_{i_1}v_{i_2}}
					-\sum_{1 \leq i_1,i_2 \leq n, i_1 \neq i_2}{y_{i_1}u_{i_2}} 
					-\sum_{1\leq k \leq n}{(2{y_{k}u_k}+{y_{k}v_k})}.
				\end{flalign*}
				
				\end{thm}
				
				\begin{proof}
				The generators $\sigma_1^\alpha$ and $\sigma_1^\beta$ in the ideals $[\sigma^{\alpha}_1,\dots,\sigma^{\alpha}_{n+1}]$ and $[\sigma^{\beta}_1,\dots,\sigma^{\beta}_{n+1}]$
				are $\alpha_1+\cdots+\alpha_{n+1}$ and $\beta_1+\cdots+\beta_{n+1}$, receptively.
				So $\sigma_1^\alpha$ and $\sigma_1^\beta$ just express elements $\alpha_{n+1}$ and $\beta_{n+1}$ in terms of minimal generating sets $\alpha_1,\dots,\alpha_{n}$
				and $\beta_1,\dots,\beta_{n}$ of $\frac{\mathbb{Z}[\alpha_1,\dots,\alpha_{n+1}]}{[\sigma^{\alpha}_1,\dots,\sigma^{\alpha}_{n+1}]}$ and
				$\frac{\mathbb{Z}[\beta_1,\dots,\beta_{n+1}]}{[\sigma^{\beta}_1,\dots,\sigma^{\beta}_{n+1}]}$, respectively.
				Each $\sigma^\alpha_l$ and $\sigma^\beta_l$ has degree $2l$.
				Since each $E_2^{2l,0}$ contains only the elements of $ 
				\frac{\mathbb{Z}[\alpha_1,\dots,\alpha_{n+1}]}{[\sigma^{\alpha}_1,\dots,\sigma^{\alpha}_{n+1}]} \otimes
				\frac{\mathbb{Z}[\beta_1,\dots,\beta_{n+1}]}{[\sigma^{\beta}_1,\dots,\sigma^{\beta}_{n+1}]}$ of degree $2l$,
				so generators $\sigma^\alpha_l$ and $\sigma^\beta_l$ only become relevant to $E_2^{2i,0}$ if $i\geq l$.
				By Lemmas \ref{lem:term1}, \ref{lem:term2} and \ref{lem:n+1elim} we have
				\begin{center}
				$d^{2}((1-l)S_{l,n}-\bar{S}_{l,n})=\sigma_{l,n}^\beta-\sigma_{l,n}^\alpha=0$.
				\end{center}
				Recall from Remark \ref{rmk:unique} that the image of each of the differentials $d^{2i}$, $i \geq 1$ in $E_{2i}^{2i,1}$ will be a unique class
				in the kernel of $d^2$ not already contained in the image of any $d^r$ for $r<2i$.
				The simplicity conditions of Remarks \ref{rmk:unique1} and \ref{rmk:unique2} will ensure that if $(1-l)S_{l,n}-\bar{S}_{l,n}$
				is in the kernel of $d^2$ previously mentioned, then it will be a generator. 
				We now proceed to determine $d^{2(l-1)}(x_{2(l-1)})$ by induction for $2\leq l\leq n+1$.
				First note that the only non-zero elements of $E_2^{*,1}$ mapped identically to zero in $\langle u_1,\dots,u_n,v_1,\dots,v_n \rangle$
				are those obtained from elements of the form
				\begin{center}
				$y_\gamma v_{\gamma'}-y_{\gamma'}v_\gamma$
				\end{center}
				for some $1\leq \gamma < \gamma' \leq n$.
				Since $\sigma^\alpha_l$ and $\sigma^\beta_l$ only become relevant to $E_2^{0,2i}$ if $i\geq l$, the element $(1-l)S_{l,n}-\bar{S}_{l,n}$
				is not contained in the image of $d^{2i}$ for $i<2l$.
				For $l=2$, the only relevant $\sigma^\alpha_i,\sigma_i^\beta$ are $\sigma_2^\alpha,\sigma_2^\beta$.
				Since $-S_{1,n}-\bar{S}_{1,n}$ is not a sum containing any terms of the form $y_\gamma v_{\gamma'}-y_{\gamma'}v_\gamma$,
				so $d^{2}(x_{2})$ is $-S_{1,n}-\bar{S}_{1,n}$ up to sign.
				For $l>2$, by induction and the Leibniz rule, the images of differentials $d^{2(i-1)}$ for $2 \leq i<l$,
				correspond to $\sigma^\alpha_i,\sigma^\beta_i$ for $2\leq i<l$ 
				or $y_\gamma v_{\gamma'}-y_{\gamma'}v_\gamma$ for some $1\leq \gamma < \gamma' \leq n$
				and $\sigma^\alpha_l,\sigma^\beta_l$ cannot be expressed in terms of $\sigma^\alpha_1,\dots,\sigma^\alpha_{l-1},\sigma^\beta_1,\dots,\sigma^\beta_{l-1}$.
				Hence $(1-l)S_{l,n}-\bar{S}_{l,n}$ must be $d^{2(l-1)}(x_{2(l-1)})$ up to a choice of class representative and sign.
				Therefore by changing the sign of $x_{2l}$ if necessary we obtain 
				$d^{2(l-1)}(x_{2(l-1)})=(1-l)S_{l,n}-\bar{S}_{l,n}$.
				\end{proof}

			\subsection{Differentials for the free loop spectral sequence}\label{sec:diff}
			
			Throughout the following arguments we consider the map $\phi$ of fibrations between the free loop fibration of $SU(n+1)/T^n$ for $n \geq 1$ 
			and the evaluation fibration studied in section \ref{sec:evalSS},
			given by the following commutative diagram
			
			\begin{equation*}\label{fig:fibcd}
						\xymatrix{
							{\Omega(SU(n+1)/T^n)} \ar[r]^(.5){} \ar[d]^(.45){id} & {\Lambda(SU(n+1)/T^n)} \ar[r]^{eval} \ar[d]^(.45){exp}  & {SU(n+1)/T^n} \ar[d]^(.45){\Delta} \\
							{\Omega(SU(n+1)/T^n)} \ar[r]^(.45){}   							 & {Map(I,SU(n+1)/T^n)} 	 \ar[r]^(.425){eval}  					 & {SU(n+1)/T^n\times SU(n+1)/T^n} ,}
			\end{equation*}
			
			where $\exp$ is given on elements by $\exp(\alpha)(t)=\alpha(e^{2\pi i t})$.
			As $SU(n+1)/T^{n}$ is simply connected, the free loop fibration induces a cohomology Leray-Serre spectral sequence $\{ \bar{E}_r,\bar{d}^r \}$.
			Hence $\phi$ indices a map of spectral sequences $\phi^*:\{ E_r,d^r \} \to \{ \bar{E}_r,\bar{d}^r \}$. 
			More precisely for each $r\geq 2$ and $a,b \in \mathbb{Z}$, we have the commutative diagram
			\begin{equation}\label{fig:phicd}
						\xymatrix{
							{E_r^{a,b}} \ar[r]^(.4){d^r} \ar[d]^(.46){\phi^*} & {E_r^{a+r,b-r+1}} \ar[d]^(.46){\phi^*}  \\
							{\bar{E}_r^{a,b}} \ar[r]^(.4){\bar{d}^r}   			& {\bar{E}_r^{a+r,b-r+1}} 	  			 	  ,}
					\end{equation}
			where $\phi^*$ for each successive $r$ is the induced map on the homology of the previous page, beginning as the map induced on the tensor on the $E_2$-pages
			by the maps $id \colon \Omega(SU(n+1)/T^n)\to\Omega(SU(n+1)/T^n)$ and $\Delta \colon SU(n+1)/T^n \to SU(n+1)/T^n \times SU(n+1)/T^n$.
			For the rest of the section we will use the notation
			\begin{center}
			$H^*(\Omega(SU(n+1)/T^n);\mathbb{Z})\cong \Gamma_{\mathbb{Z}}(x'_2,x'_4,\dots,x'_{2n})\otimes\Lambda_{\mathbb{Z}}(y'_1,\dots,y'_{n}),$
			\end{center}
			\begin{center}
			$\;\;\; H^*(SU(n+1)/T^n;\mathbb{Z})\cong \frac{\mathbb{Z}[\gamma_1,\dots,\gamma_{n+1}]}{[\sigma^\gamma_1,\dots,\sigma_{n+1}^\gamma]}$,
			\end{center}
			where $|y'_i|=1, |\gamma_j|=2, |x'_{2i}|=2i$ for each $1\leq i\leq n,1\leq j\leq n+1$ and $\sigma^\gamma_1,\dots,\sigma_{n+1}^\gamma$
			are a basis of the symmetric functions on $\gamma_i$.
			Now we determine all the differentials in $\{ \bar{E}_r,\bar{d}^r \}$.	
			
			\begin{thm}\label{thm:allDiff}
			For each $n\geq 1$, the only non-zero differentials on generators of the $\bar{E}_2$-page of $\{ \bar{E}_r,\bar{d}^r \}$ are up to class representative and sign, 
				\begin{equation*}
					\bar{d}^{2}(x'_{2})
					=-\sum_{1 \leq i_1,i_2 \leq n, i_1 \neq i_2}{y'_{i_1}\gamma_{i_2}}-\sum_{1\leq k \leq n}{2{y'_{k}\gamma_k}}
				\end{equation*}
				and for $3 \leq l \leq n+1$,
				\begin{align*}
					&\bar{d}^{2(l-1)}(x'_{2(l-1)})
					= \\
					&(1-l)\sum_{\substack{1\leq i_1<\cdots<i_{l-1}\leq n
					\\ 1\leq k\leq n,\; i_j \neq k}}
					{y'_{k}\gamma_{i_1}\cdots \gamma_{i_{l-1}}} 
					-\sum_{\substack{1\leq i_1<\cdots<i_{l-3}\leq n
					\\ 1\leq k,k'\leq n,\; i_j \neq k \neq k'}}
					{y'_{k}\gamma_{k'}^2\gamma_{i_1}\cdots\gamma_{i_{l-3}}}
					-2
					\sum_{\substack{1\leq i_1<\cdots<i_{l-2}\leq n
					\\ 1\leq k\leq n,\; i_j \neq k}}
					{y'_k \gamma_k \gamma_{i_1}\cdots\gamma_{i_{l-2}}}
					.
				\end{align*}
			\end{thm}
			
			\begin{proof}
				Throughout the proof it may be useful to refer to Figure \ref{fig:freeSS}, showing differentials in the spectral sequence.
				The identity $id \colon \Omega(SU(n+1)/T^n)\to\Omega(SU(n+1)/T^n)$ induces the identity map on cohomology.
				The diagonal map $\Delta \colon SU(n+1)/T^n \to SU(n+1)/T^n \times SU(n+1)/T^n$ induces the cup product on cohomology.
				Hence by choosing generators in $\{ \bar{E}_r,\bar{d}^r \}$, we may assume that
				\begin{center}
				$\phi^*(y_i)=y'_i,\;\;\phi^*(x_i)=x'_i\;\;$and$\;\;\phi^*(\alpha_i)=\gamma_i=\phi^*(\beta_i)=\phi^*(u_i),\;\;$so$\;\;\phi^*(v_i)=0$.
				\end{center}
				For dimensional reasons, the only possibly non-zero differential on generators ${y'}_i$ in $\{ \bar{E}_r,\bar{d}^r \}$ is $\bar{d}^{2}$.
				However for each $1\leq i\leq n$ using commutative diagram (\ref{fig:phicd}) and Lemma \ref{lem:E^2_{*,1}d^2}, we have
				\begin{center}
					$\bar{d}^2(y'_i)=\bar{d}^2(\phi^*(y_i))=\phi^*(d^2(y_i))=\phi^*(v_i)=0$.
				\end{center}
				Hence all elements of $\bar{E}_2^{(*,1)}$ and $\bar{E}_2^{(*,0)}$ survive to $\bar{E}_{\infty}$, unless they are in the image of some differential
				$\bar{d}^r$ for $r\geq 2$.
				Using commutative diagram (\ref{fig:phicd}) and Theorem \ref{thm:finaldiff}, we have up to class representative and sign
				\begin{equation*}
					\bar{d}^{2}(x'_{2})=\phi^*(d^2(x_2))=\phi^*(-S_{2,n}-\bar{S}_{2,n})
					=-\sum_{1 \leq i_1,i_2 \leq n, i_1 \neq i_2}{y'_{i_1}\gamma_{i_2}}-\sum_{1\leq k \leq n}{2{y'_{k}\gamma_k}}
				\end{equation*}
				and for $3 \leq l \leq n+1$,
				\begin{flalign*}
					\bar{d}^{2(l-1)}(x'_{2(l-1)})=\phi^*(d^2(x_{2(l-1)}))=\phi^*((1-l)S_{l,n}-\bar{S}_{l,n})
				\end{flalign*}
				\begin{flalign*}
					=(1-l)\sum_{\substack{1\leq i_1<\cdots<i_{l-1}\leq n
					\\ 1\leq k\leq n,\; i_j \neq k}}
					{y'_{k}\gamma_{i_1}\cdots \gamma_{i_{l-1}}} 
					-\sum_{\substack{1\leq i_1<\cdots<i_{l-3}\leq n
					\\ 1\leq k,k'\leq n,\; i_j \neq k \neq k'}}
					{y'_{k}\gamma_{k'}^2\gamma_{i_1}\cdots\gamma_{i_{l-3}}}
					-2
					\sum_{\substack{1\leq i_1<\cdots<i_{l-2}\leq n
					\\ 1\leq k\leq n,\; i_j \neq k}}
					{y'_k \gamma_k \gamma_{i_1}\cdots\gamma_{i_{l-2}}}
					.
				\end{flalign*}		
				All differentials on generators $\gamma_i$, for each $1\leq i \leq n+1$, are zero for dimensional reasons.
			\end{proof}

			\begin{center}
				\begin{tikzpicture}
					\matrix (m) [matrix of math nodes,
						nodes in empty cells,nodes={minimum width=5ex,
						minimum height=5ex,outer sep=-5pt},
						column sep=1ex,row sep=1ex]{
																									&	 \vdots		&	 \vdots	 							 &		 												& 		  &						   &			 &					& \\
																									&	  2n\;  	& \langle x'_{2n} \rangle&		 												& 		  &						   &    	 &    			& \\
																									&	 \vdots   &  \vdots	 							 &		 												& 		  &						   & 			 & 				  & \\
																									&			6     &   \langle x'_6 \rangle &		 												& 		  & 						 &\dots &   			  & \\
						H^{*}(\Omega(SU(n+1)/T^n);\mathbb{Z}) &			4     &   \langle x'_4\rangle  &		 												& 			& 						 & 			 &					& \\
																									&			2     &   \langle x'_2\rangle  &													  & \dots& 		 				 & 			 &   			  & \\
																									&			      &  			   							 &		 											  & 		  & 		 				 & 			 &          & \\
																									&			1			&		\langle y'_i\rangle	 &			\lcdot			 				  &\lcdot &\lcdot				 & \cdots& \lcdot		& \cdots \\	
																									&		  0     &  			   							 & \langle\gamma_i\rangle 	  &\lcdot &\lcdot				 & \cdots&  \lcdot  & \cdots \\
																									&\quad\strut&     0    							 &				2										&		4	  &	 6	 				 & \cdots&     2n   & \cdots \strut \\};
					\draw[-stealth] (m-2-3.south east) -- (m-8-8.north);
					\draw[-stealth] (m-4-3.south east) -- (m-8-6.north);
					\draw[-stealth] (m-5-3.south east) -- (m-8-5.north);
					\draw[-stealth] (m-6-3.south) -- (m-8-4.north);

				\draw[thick] (m-1-2.east) -- (m-10-2.east) ;
				\draw[thick] (m-10-2.north) -- (m-10-9.north) ;
				\end{tikzpicture}
				\label{fig:freeSS}
				\end{center}
				\begin{center}
					$\;\;\;\;\;\;\;\;\;\;\;\;\;\;\;\;\;\;\;\;\;\;\;\;\;\;\;\;\;\;\;\;\;\;\;\;\;\;\;\;\; H^{*}(SU(n+1)/T^n;\mathbb{Z})$
				\end{center}
				\begin{center}
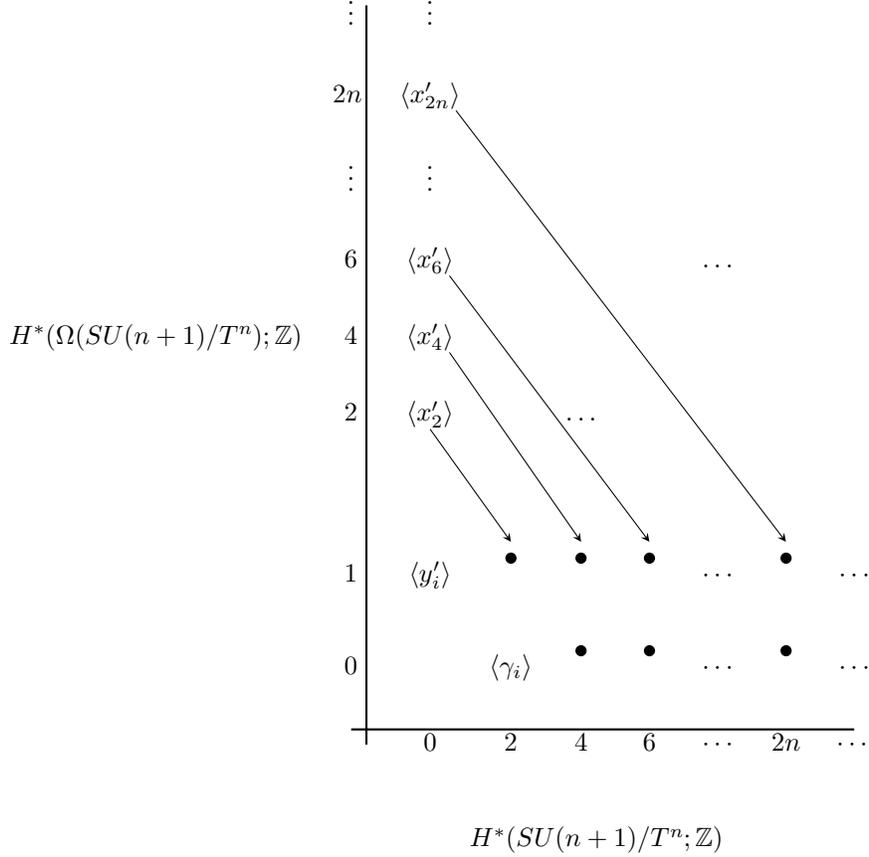

				\captionof{figure}{Generators in integral cohomology Leray-Serre spectral sequence $\{\bar{E}_r,\bar{d}^r\}$ converging to $H^{*}(\Lambda(SU(n+1)/T^n);\mathbb{Z})$.}
				\end{center}

		\subsection{Basis}\label{sec:basis}
			
			By considering a basis of $\mathbb{Z}[\gamma_1,\dots,\gamma_n]$ that resembles the image of the $d^2$ differential in Theorem \ref{thm:allDiff},
			it becomes easier to study the $E_3$-page of the spectral sequence.
			
			\begin{rmk}\label{rmk:TildeBasis}
				In $\mathbb{Z}[\gamma_1,\dots,\gamma_{n}]$, let $\bar{\gamma}=\gamma_1+\cdots+\gamma_n$ and $\tilde{\gamma}_i=\bar{\gamma}+\gamma_i$ for each $1\leq i\leq n$. 
				We may rearrange the standard basis $\gamma_1,\dots,\gamma_n$ of $\mathbb{Z}[\gamma_1,\dots,\gamma_{n}]$ to $\gamma_1,\dots,\gamma_{n-1},\bar{\gamma}$.
				Then rearrange to $\tilde{\gamma}_1,\dots,\tilde{\gamma}_{n-1},\bar{\gamma}$,
				by adding $\bar{\gamma}$ to all other basis elements.
				Notice that the replacement $\gamma_i \mapsto \tilde{\gamma}_i$ for $1\leq i \leq n-1$, $\gamma_n \mapsto \bar{\gamma}$
				could have been chosen $\gamma_j \mapsto \bar{\gamma}$ for any $1\leq j \leq n$ and $\gamma_i \mapsto \tilde{\gamma}_i$ for any $i\neq j$ instead.
				
				Replacing $\bar{\gamma}$ by $(n+1)\bar{\gamma}-\tilde{\gamma}_1-\cdots-\tilde{\gamma}_{n-1}$ gives $\tilde{\gamma}_n$, 
				hence $\tilde{\gamma}_1,\dots,\tilde{\gamma}_{n}$ forms a rational basis.
			\end{rmk}
			
			\begin{prop}
				Using the notion of (\ref{eq:SipleRedusingHomogenious}),
				we can rewrite $h^{n-l+2}_i$ for each $3\leq l \leq n$ in the basis of Remark \ref{rmk:TildeBasis} as
				\begin{equation*}
					h^{n}_2=\sum_{\substack{0\leq k \leq 2 \\ 1\leq i_j \leq n-1}}
					{(-1)^{2-k} \binom{n+1}{2-k} \tilde{\gamma}_{i_1}\cdots\tilde{\gamma}_{i_{k}}\bar{\gamma}^{2-k}}
				\end{equation*}
				and
				\begin{equation*}
					h^{n-l+2}_l=\sum_{\substack{0\leq k \leq l \\ 1\leq i_j \leq n-l+2}}
					{(-1)^{l-k} \binom{n+1}{l-k} \tilde{\gamma}_{i_1}\cdots\tilde{\gamma}_{i_{k}}\bar{\gamma}^{l-k}}
				\end{equation*}
			\end{prop}
			\begin{proof}
				First note that in the basis of Remark \ref{rmk:TildeBasis} we can rewrite the original basis in terms of the new one
				\begin{equation}\label{eq:BaseExpression}
					\gamma_i=\tilde{\gamma_i}-\bar{\gamma} \text{  for  } 1\leq i\leq n-1, \;\;\; \gamma_n=n\bar{\gamma}-\sum_{i=1}^{n-1}{\tilde{\gamma_i}}.
				\end{equation}
				When $l=2$ using (\ref{eq:BaseExpression})
				\begin{flalign}
					h_2^{n}&=\sum_{a=0}^{2}{\big{(}(n\bar{\gamma}-\sum_{j=1}^{n-1}{\tilde{\gamma}_j})^{2-a}
					\sum_{1\leq i_1\leq i_2 \leq n-1}{\prod^a_{k=1}{(\tilde{\gamma}_{i_k}-\bar{\gamma})}}\big{)}} \nonumber \\
					&= (n\bar{\gamma}-\sum_{j=1}^{n-1}{\tilde{\gamma}_j})^{2}
						+\sum^{n-1}_{a=1}{(n\bar{\gamma}-\sum_{j=1}^{n-1}{\tilde{\gamma}_j})(\tilde{\gamma}_a-\bar{\gamma})}
						+\sum^{n-1}_{a=1}{(\tilde{\gamma}_{a}-\bar{\gamma})^2}
						+\sum_{1\leq i_1 < i_2 \leq n-1}{(\tilde{\gamma}_{i_1}-\bar{\gamma})(\tilde{\gamma}_{i_2}-\bar{\gamma})}. \label{eq:l=2}
				\end{flalign}
				For $1\leq k,k_1, k_2\leq n-1$, $k_1 \neq k_2$,
				we consider the terms of the form 
				\begin{equation*}
					\bar{\gamma}^2, \; \tilde{\gamma}_k\bar{\gamma}, \; \tilde{\gamma}_k^2, \; \tilde{\gamma}_{k_1}\tilde{\gamma}_{k_2}
				\end{equation*}
				in tern and count their occurrences in the summands of (\ref{eq:l=2}).
				In total $n^2$ element of the form $\bar{\gamma}^2$ are produced by the first summand of (\ref{eq:l=2}),
				minus $n(n-1)$ times in the second, $n-1$ in the third and $\binom{n-1}{2}$ in the last.   
				Hence in total
				\begin{equation*}
					n^2-n(n-1)+(n-1)+\binom{n-1}{2}=n+\binom{n-1}{1}+\binom{n-1}{2}=\binom{n}{1}+\binom{n}{2}=\binom{n+1}{2}.
				\end{equation*}
				In total $-2n$ elements of the form $\tilde{\gamma}_k\bar{\gamma}$ are produced in the first summand of (\ref{eq:l=2}),
				$2n-1$ in the second, minus $2$ in the third and $2-n$ in the last.
				Hence in total
				\begin{equation*}
					-2n+(2n-1)-2+(2-n)=n+1=\binom{n+1}{1}.
				\end{equation*}
				The terms $\tilde{\gamma}_k^2$ are produced once in the first summand of (\ref{eq:l=2}), once in the third and negative once in the second,
				hence once in total.
				The terms $\tilde{\gamma}_{k_1}\tilde{\gamma}_{k_2}$ are produced twice in the first summand,
				minus twice in the the second and once in the last,
				hence once in total.
				Therefore the conditions of the proposition are satisfied.
				
				For $l\geq 3$ using (\ref{eq:BaseExpression})
				\begin{equation}\label{eq:l>3}
					h^{n-l+2}_l=\sum_{1\leq i_1\leq \cdots\leq i_l \leq n-l+2}{\prod_{k=1}^{l}{(\tilde{\gamma}_{i_k}-\bar{\gamma})}}.
				\end{equation}
				For any choice of $1\leq i_1\leq \cdots \leq i_k \leq n-l+2$ and non-negative integers $b,a_1,\dots,a_k$ such that $b+a_1+\cdots+a_k=l$,
				terms of the form
				\begin{equation}\label{eq:ProdChoice}
					\tilde{\gamma}_{i_1}^{a_1}\cdots\tilde{\gamma}_{i_k}^{a_k}\bar{\gamma}^{b}
				\end{equation}
				describe up to multiplicity all possible summand in the expansion of equation (\ref{eq:l>3}).
				Define $h^{n-l+2}_l\{ \tilde{\gamma}_{i_1}^{a_1}\cdots\tilde{\gamma}_{i_k}^{a_k}\bar{\gamma}^{b} \}$ to be the multiplicity of the summand containing
				$\tilde{\gamma}_{i_1}^{a_1}\cdots\tilde{\gamma}_{i_k}^{a_k}\bar{\gamma}^{b}$ in the expansion of equation (\ref{eq:l>3}).
				We will show that if $h^{n-l+2}_l$ has of the form of equation (\ref{eq:l>3}) for all $n+1\geq l \geq 2$ satisfies the statement of the proposition
				for $3\geq l\geq n$. 
				In particular
				\begin{equation}\label{eq:BorckekenExpress}
					h^{n-l+2}_l\{ \tilde{\gamma}_{i_1}^{a_1}\cdots\tilde{\gamma}_{i_k}^{a_k}\bar{\gamma}^{b} \} = (-1)^{b} \binom{n+1}{b}
				\end{equation}
				where $k+b=l$,
				which would complete the proof of the proposition.
				
				Considering each summand of equation (\ref{eq:l>3}) in tern and counting the number of 
				$\tilde{\gamma}_{i_1}^{a_1}\cdots\tilde{\gamma}_{i_k}^{a_k}\bar{\gamma}^{b}$
				produced in each product, we obtain
				\begin{equation*}
					h^{n-l+2}_l\{ \tilde{\gamma}_{i_1}^{a_1}\cdots\tilde{\gamma}_{i_k}^{a_k}\bar{\gamma}^{b} \} =
					(-1)^b
					\sum_{\theta=0}^{b}{\multiset{n-l+2-k}{b-\theta}\sum_{\substack{\alpha_1+\cdots+\alpha_k=\theta \\ \alpha_j\geq 0}}
					{\prod^{\theta}_{\beta=1}{\binom{a_\beta+\alpha_\beta}{\alpha_\beta}}}}.
				\end{equation*}
				We proceed by induction on $n$ and will prove (\ref{eq:BorckekenExpress}) for all $n\geq 1$ and $2\leq l\leq n+1$.
				When $n=1$, the only valid value of $l$ is $2$ and $h^{n-l+2}_{l}=(\tilde{\gamma}_1-\bar{\gamma})^2$ whose expansions satisfies (\ref{eq:BorckekenExpress}). 
				Assume that (\ref{eq:BorckekenExpress}) holds for all $\phi\leq n$.
				It is clear that $h^{n-l+1}_{l}\{ \bar{\gamma}^{n+1} \}=(-1)^{n+1}$ and $h^{n-l+1}_{l}\{ \tilde{\gamma}_{i_1}^{a_1}\cdots\tilde{\gamma}_{i_k}^{a_k} \}=1$
				for any choice of $a_1,\dots,a_{k}$ since in the expansion of equation (\ref{eq:BorckekenExpress}) there would be only one way to obtain the element.
				For $1\leq b \leq n$, by induction
				\begin{equation}\label{eq:Case(n,b)}
					\binom{n}{b}
					=\sum_{\theta=0}^{b}{\multiset{n-l+1-k}{b-\theta}\sum_{\substack{\alpha_1+\cdots+\alpha_k=\theta \\ \alpha_j\geq 0}}
					 {\prod^{\theta}_{\beta=1}{\binom{a_\beta+\alpha_\beta}{\alpha_\beta}}}}
				\end{equation}
				and
				\begin{equation}\label{eq:Case(n,b-1)}
					\binom{n}{b-1}=(-1)^{b-1}h^{n-l+2}_{l-1}\{ \tilde{\gamma}_{i_1}^{a_1}\cdots\tilde{\gamma}_{i_k}^{a_k}\bar{\gamma}^{b-1} \}
					=\sum_{\theta=0}^{b-1}{\multiset{n-l+2-k}{b-1-\theta}\sum_{\substack{\alpha_1+\cdots+\alpha_k=\theta \\ \alpha_j\geq 0}}
					 {\prod^{\theta}_{\beta=1}{\binom{a_\beta+\alpha_\beta}{\alpha_\beta}}}}.
				\end{equation}
				For each $0\leq\theta\leq b-1$ the sum of values from (\ref{eq:Case(n,b)}) and (\ref{eq:Case(n,b-1)}) corresponds to the $\theta$ summand in the expression for  
				$h^{n-l+2}_l\{ \tilde{\gamma}_{i_1}^{a_1}\cdots\tilde{\gamma}_{i_k}^{a_k}\bar{\gamma}^{b} \}$,
				since the binomial expressions agree and the multi set expression sum to the correct result.
				The only reaming summand in $h^{n-l+2}_l\{ \tilde{\gamma}_{i_1}^{a_1}\cdots\tilde{\gamma}_{i_k}^{a_k}\bar{\gamma}^{b} \}$ is the one corresponding to $\theta=b$.
				However this is same as that in (\ref{eq:Case(n,b-1)}) because $\multiset{n-l+2-k}{0}=1=\multiset{n-l+1-k}{0}$ and the binomial parts agree.
			\end{proof}

		\subsection{Pre-quotient spectral sequence}\label{sec:PreQuotient}
			
			In this section we simplify the problem of studying the $E_3$-page of $\{ E_r, d^r \}$ by considering the differential bigraded algebra
			$E_2$ with differential $d^2$, factored thorough the polynomial algebra, removing the quotient by symmetric ideal.
			In section \ref{sec:RationalPre} we consider a differential bigraded algebra that turns out to be a rational version of the this differential bigraded algebra.
			In the rational case the problem is further simplified and so is more easily dealt with.
			Then in section \ref{sec:IntegralPre} we extend the rational result to the integral situation.
			
			\subsubsection{Rational pre-quotient spectral sequence}\label{sec:RationalPre}
			
			Given a sequence indexed by natural numbers $i_1,\dots,i_j$, we denote by $i_1,\dots,\hat{i_s},\dots,i_j$ the same sequence with $i_s$ missing. 
			In the free commutative graded algebra $\Lambda(y_1,\dots,y_n)$ for any $1\leq i_1<\cdots<i_j\leq n$, denote by $\hat{y}_{i_1,\dots,i_j}$
			the elements of $\Lambda(y_1,\dots,y_n)$ given by the multiplication in ascending order of indices of all elements $y_k$ except $y_{i_1},\dots,y_{i_j}$.
			
			Let $(E,d)$ be a differential bigraded algebra with $E^{p,q}=A^p\otimes B^q$, where $A$ and $B$ are graded algebras.  
			Given elements $x_1,\dots,x_n\in E^2_{0,q}=B^q$, we will want to refer to all elements in the row $E^{*,q}$ involving generators $x_1,\dots,x_n$ and
			hence we denote by $E^{p,q}(x_1,\dots,x_n)$ the graded algebra $A^p \otimes \langle x_1,\dots,x_n \rangle^q$
			and let $H^*E^{p,q}(x_1,\dots,x_n)$ be the image of the inclusion of $E^{p,q}(x_1,\dots,x_n)$ into the homology of $(E,d)$.
			Similarly we may extend this notation to as spectral sequence where the second pages statistics the initial condition. 
			
			\begin{lem}\label{lem:1to1spec}
					For any $n\geq 1$, let $A=\mathbb{Z}[\gamma_1,\dots,\gamma_n]$ and $B=\Lambda_\mathbb{Z}(y_1,\dots,y_n)\otimes\Gamma_{\mathbb{Z}}(x)$
					be the graded algebras with $|\gamma_i|=2=|x|$ and $|y_i|=1$.
					For each integer $i\geq 1$, denote by $x_i$ the element of $\Gamma_{\mathbb{Z}}(x)$ such that $x^i = i!x_i$.
					There is a differential bigraded algebra $(E,d)$
					with $E^{p,q}=A^p\otimes B^q$, differential of bidegree $(2,-1)$ given by $d(x)=y_1\gamma_1-y_2\gamma_2+\dots+(-1)^{n+1}y_n\gamma_n$.
					The homology of $(E,d)$ is given by
					\begin{align*}
						H^*E^{0,n+2m}(x_m y_1\cdots y_n)&\cong\mathbb{Z}, \\
						H^*E^{p,n-j}(\hat{y}_{i_1,\dots,i_j})&\cong\mathbb{Z}^{\sum_{k=0}^{n-j}{(-1)^k\binom{n}{j+k}\supermultiset{n}{p-k}}},\\
						H^*E^{p,0}(1)&\cong\mathbb{Z}^{\supermultiset{n}{p}}
					\end{align*}
					for each $m,p\geq 0$, $1\leq j\leq n-1$, $1\leq i_1<\cdots<i_j\leq n$ and all other elements are trivial.
			\end{lem}
			\begin{proof}
					For $m\geq 1$, due to the divide polynomial structure on $\Gamma_{\mathbb{Z}}(x)$,
					\begin{equation}\label{eq:difdiv}
							d(x_m)=\frac{1}{m!}d(x^m)= \frac{1}{(m-1)!} x^{m-1} d(x)= x_{m-1}d(x).
					\end{equation}
					Algebraa $E$ is generated additively by elements of the form $x_m\hat{y}_{i_1,\dots i_j}P$
					with $m\geq 0$, $0\leq j \leq n$, $1\leq i_1<\cdots<i_j\leq n$ and $P\in \mathbb{Z}[\gamma_1,\dots,\gamma_n]$.
					For $1\leq j \leq n$ and $m\geq 1$,
					\begin{flalign}\label{eq:boundary}
						\begin{split}
								d(x_m\hat{y}_{i_1,\dots,i_j}P)
								&=x_{m-1}d^2(x)\hat{y}_{i_1,\dots,i_j}P \\
								&=\sum_{t=1}^j{(-1)^{i_t+1}(-1)^{i_t+t-2}x_{m-1}\hat{y}_{i_1,\dots,\hat{i}_t,\dots,i_j}\gamma_{i_t}P} \\
								&=\sum_{t=1}^j{(-1)^{t-1}x_{m-1}\hat{y}_{i_1,\dots,\hat{i}_t,\dots,i_j}\gamma_tP},
						\end{split}
					\end{flalign}
					where the additional $(-1)^{i_t+t-2}$ sign changes come from reordering the $y_i$.
					The generator $y_t$ swaps places with $y_i$, $i_t-1$ times for $i<t$ changing the sign each time,
					however $t-1$ of these $y_i$ are missing.
					
					Ignoring $x_m,x_{m-1},\gamma_t$ and $P$ in (\ref{eq:boundary})
					and thinking of $\hat{y}_{i_1,\dots,i_j}$ as simplices in an $n$ vertex simplicial complex, $d$ is the usual boundary map.
					In particular, this implies that the differential and hence the differential bigraded algebra is well defined. 
					With this idea in mind, we construct the following CW-complex $X$.
					For each $m\geq 0$, $1\leq j \leq n$, $1\leq i_1<\cdots<i_j\leq n$ and $P\in \mathbb{Z}[\gamma_1,\dots,\gamma_n]$, there is a corresponding cell of dimension $j-1$
					and one additional zero-cell $*$.
					For each cell of dimension $\geq 1$, if $m=0$ the attaching map for the boundary of the cell will be $*$, as $d$ of these element in $E$ is zero.
					For $m\geq 1$, the attaching map is given by the $d$ in (\ref{eq:boundary}) tacking the cell as a simplex of corresponding dimension.
					
					For $j\geq3$, every $x_m\hat{y}_{i_1,\dots,i_j}P$ has the image of its $d$ differential represented in $X$.
					So for $j\geq3$, a non-zero element in the homology of $(E,d)$ corresponds to an element in  $H_j(X;\mathbb{Z})$. 
					
					First consider the cells corresponding to generators
					\begin{equation*}
						x_m\hat{y}_{i_1,\dots,i_j}P,
					\end{equation*}
					where $2\leq j\leq n$, $0\leq m \leq j-1$, $1\leq i_1<\cdots<i_j\leq n$ and $P\in\mathbb{Z}[\gamma_1,\dots,\gamma_n]$.
					In this case some boundary component of the cell will be attached to $*$.
					If $m>j-1$ or $j=1$, then the cell is not connected to $*$ since all 0-cells in its boundary are not $*$
					and all cells with $*$ in their boundary have $*$ as their only 0-cell in the boundary.
					Let $X_*$ be the connected component of $*$ in $X$.
					
					Now consider cells corresponding to generators
					\begin{equation}\label{eq:StandarForm}
						x_m\hat{y}_{i_1,\dots,i_j}\gamma_{a_1}\cdots\gamma_{a_b}P,
					\end{equation}
					where $2\leq j \leq n$, $0 \leq b \leq n$, $1\leq i_1<\cdots<i_j\leq n$, $1\leq a_1<\cdots<a_b\leq n$, $i_s\neq a_t$,
					$P\in\mathbb{Z}[\gamma_{i_1},\dots,\gamma_{i_j},\gamma_{a_1},\dots,\gamma_{a_b}]$
					and $m>j-1$.
					Notice that all elements can be expressed uniquely in this form.
					In this form the corresponding cell is contained in the boundary of the unique $(b+j)$-cell corresponding to $x_{m+b}\hat{y}_{i_1,\dots,i_j,a_1,\dots,a_b}P$.
					Therefore all connected component other than $X_*$ are contractable.
					Hence for $j\geq3$, the only non-trivial elements in $H_j(X;\mathbb{Z})$ correspond to cycles in $H_j(X_*;\mathbb{Z})$.
					
					All cells of $X_*$ correspond to an element of the form of (\ref{eq:StandarForm}) but with $m\leq j-1$.
					Again each such cell is contained in the boundary of a unique $(b+j)$-cell corresponding to $x_{m+b}\hat{y}_{i_1,\dots,i_j,a_1,\dots,a_b}P$.
					Each such top cell is a simplex whose $j-m-1$ faces have been identified to $*$.
					Hence all homology classes of $X$ are generated by cells whose boundary is exactly $*$.
					These correspond to generators of the form
					\begin{equation*}
						\hat{y}_{i_1,\dots,i_j}P,
					\end{equation*}
					where $j\geq 2$, $1\leq i_1<\cdots<i_j\leq n$ and $P\in \mathbb{Z}[\gamma_1,\dots,\gamma_n]$.
					For $j\geq 3$, at $E^{p,n-j}(\hat{y}_{i_1,\dots,i_j})$ there are $\binom{n}{j}$ possible choices for $i_1,\dots,i_j$ and $\multiset{n}{p}$ choices for $P$.
					However if $j\leq n-1$, there are $\binom{n}{j+1}\multiset{n}{p-1}$
					cells of dimension one higher whose boundary contain cells corresponding to generators of the form 
					$x\hat{y}_{i_1,\dots,i_{j+1}}P$, where $P$ has degree $p-1$.
					Again if $j\leq n-2$, in one dimension higher there are $\binom{n}{j+2}\multiset{n}{p-2}$ cells with boundary contained in the previous cells
					and so on until the top cells in dimension $n-1$.
					The lemma is now proved for all elements containing a multiple of $\hat{y}_{i_1,\dots,i_j}$ when $j\geq 3$.
					It remains to be deduced what happens to generators with $0\leq j\leq 2$.
					
					For $m\geq 1$, $1\leq i\leq n$ and $P\in \mathbb{Z}[\gamma_1,\dots,\gamma_n]$, using (\ref{eq:boundary})
					\begin{equation*}
						d(x_m\hat{y}_iP)=x_{m-1}\gamma_i P.
					\end{equation*}
					Therefore the kernel of $d$ on generators of the form $x_m\hat{y}_iP$ is generated by elements of the form $x_m(\hat{y}_i \gamma_j-\hat{y}_j \gamma_i)P$
					for some $1\leq j \leq n$ and $j\neq i$.
					Again from (\ref{eq:boundary}) this is exactly the image of generators of the form $x_{m+1}\hat{y}_{i,j}P$.
					Therefore the only elements that may survive in the homology of $(E,d)$ are generated by those of the form
					\begin{center}
						$\hat{y}_{i_1,i_2}P$, $\hat{y}_i P$ or $x_m y_1\dots y_n P$
					\end{center}
					for $m\geq 1$, $1\leq i_1<i_2 \leq n$, $1\leq i\leq n$ and $P\in \mathbb{Z}[\gamma_1,\dots,\gamma_n]$.
					The generators of the form $\hat{y}_{i_1,i_2}P$ correspond to 1-cells in $X_*$
					and since they are not affected by $\hat{y}_{i_1,\dots,i_j}P$ for $j\leq 2$ they can be dealt with in the same way we did for $j\geq 3$.
					At $(E_{2}^{n-j}(\hat{y}_i))^p$ there are $\binom{n}{1}\multiset{n}{p}$ generators of the form $\hat{y}_i P$.
					The image of $d$ is generated by $\binom{n}{2}\multiset{n}{p-1}$ elements of the form $d(x_1\hat{y}_{i_1,i_2})P$.
					In $X$ there are $\binom{n}{1}\multiset{n}{p}$ 2-cells in $X\setminus X_*$
					whose boundary lie in cells corresponding to the generators of the form $x_1\hat{y}_{i_1,i_2}$ and so on as in previous cases.
					Finally at $E^{p,n-j}(x_m y_1\cdots y_n)^p$ there are $\binom{n}{0}\multiset{n}{p}$ generators of the form $x_m y_1\cdots y_n P$.
					The image of $d$ is generated by $\binom{n}{1}\multiset{n}{p-1}$ elements of the form $d(x_{m+1}\hat{y}_{i})P$ and so on as in previous cases.
					Hence $E^{n+2m,p}(x_m y_1\cdots y_n)^p\cong \mathbb{Z}^{\sum_{k=0}^{n}{(-1)^k\binom{n}{k}\supermultiset{n}{p-k}}}$.
					However for $p\geq 1$, by Lemma \ref{lem:combino}, we have $\sum_{k=0}^{n}{(-1)^k\binom{n}{k}\multiset{n}{p-k}}=0$.
					
			\end{proof}
			
			\subsubsection{Integral pre-quotient spectral sequence}\label{sec:IntegralPre}
			
			We now continue the study of the cohomology Leray-Serre spectral sequence associated to the free loop fibration of $SU(n+1)/T^n$ for $n\geq 2$ 
			we began in section \ref{sec:diff}.
			We now refer to the Serre spectral sequence associated to the free loop fibration as $(E_r,d^r)$.
			In addition to simplify notation,
			we remove the notation required to differentiate elements in the free loop spectral sequence from those of the path space spectral sequence, letting

			\begin{center}
					$H^*(\Omega(SU(n+1)/T^n);\mathbb{Z})= \Gamma_{\mathbb{Z}}(x_2,x_4,\dots,x_{2n})\otimes\Lambda_{\mathbb{Z}}(y_1,\dots,y_{n}),$
			\end{center}
			\begin{center}
					$\;\;\; H^*(SU(n+1)/T^n;\mathbb{Z})= \frac{\mathbb{Z}[\gamma_1,\dots,\gamma_{n+1}]}{[\sigma^\gamma_1,\dots,\sigma_{n+1}^\gamma]}$,
			\end{center}
			
			where $|y_i|=1, |\gamma_j|=2, |x_{2i}|=2i$ for each $1\leq i\leq n,1\leq j\leq n+1$ and $\sigma^\gamma_1,\dots,\sigma_{n+1}^\gamma$
			the elementary symmetric polynomials in $\gamma_i$. 
			Recall that in Theorem \ref{thm:allDiff} all differentials of $(E_r,d^r)$ were determined.
			In particular by choosing the sign of our generators, we may assume
			\begin{equation}\label{eq:d^2}
				d^2(x_2)
				=\sum_{i=1}^n(-1)^{i+1}{y_i(\gamma_1+\cdots+\hat{\gamma}_i+\cdots+\gamma_n+2\gamma_i)}.
			\end{equation}
			
			To begin with we ignore the symmetric quotient by the ideal $[\sigma_1,\dots,\sigma_n]$ and study the differential bigraded algebra
			$(\bar{E},D)$, with
			\begin{center}
				$\bar{E} = H^*(\Omega(SU(n+1)/T^n);\mathbb{Z}) \otimes \frac{\mathbb{Z}(\gamma_1,\dots,\gamma_{n+1})}{[\sigma^\gamma_1]}
				= H^*(\Omega(SU(n+1)/T^n);\mathbb{Z}) \otimes \mathbb{Z}(\gamma_1,\dots,\gamma_n)$
			\end{center}
			and $D$ is defined as $d^2$. 
			
			\begin{thm}\label{thm:SymFree}
				The homology of $(\bar{E},D)$, as a module is given by
					\begin{align*}
						H^*\bar{E}^{0,n+2m+\dim(X)}((x_2)_m X y_1\cdots y_n)&\cong\mathbb{Z}, \\
						H^*\bar{E}^{p,n+2m+\dim(X)}((x_2)_m X y_1\cdots y_n)&\cong\mathbb{Z}_{n+1} \; for \; p>0, \\
						H^*\bar{E}^{p,n-j+\dim(X)}(X \hat{y}_{i_1,\dots,i_j})&\cong\mathbb{Z}^{\sum_{k=0}^{n-j}{(-1)^k\binom{n}{j-k}\supermultiset{n}{p-k}}},\\
						H^*\bar{E}^{p,\dim(X)}(X)&\cong\mathbb{Z}^{\supermultiset{n}{p}}
					\end{align*}
				for each $m\geq 1$, $p\geq 0$, $1\leq j\leq n-1$, $1\leq i_1<\cdots<i_j\leq n$, $X\in\langle \Gamma_\mathbb{Z}[x_4,\dots,x_{2n}] \rangle$ a monomial
				and all other elements trivial.	
			\end{thm}
			\begin{proof}
				Consider the homomorphism of abelian groups $f\colon \mathbb{Z}[\gamma_1,\dots,\gamma_n]\to\mathbb{Z}[\gamma_1,\dots,\gamma_n]$ given by
				\begin{equation*}
					\gamma_i\mapsto\gamma_1+\cdots\hat{\gamma_i}+\cdots+\gamma_n+2\gamma_i.
				\end{equation*}
				For simplicity we use the notation $\tilde{\gamma_i}=\gamma_1+\cdots\hat{\gamma_i}+\cdots+\gamma_n+2\gamma_i$.
				The matrix with respect to basis $\gamma_1,\dots,\gamma_n$ of $f$ is given by the top left hand $n\times n$ matrix below.
				\begin{center}
					$
					\begin{bmatrix}
						2 &   							&   		 &  							 &   \\
							&   							&   		 & \text{\huge{1}} &   \\
							&   							& \ddots &   							 &   \\
							& \text{\huge{1}} &   		 &   							 &   \\
							&  								&   		 &   							 & 2
					\end{bmatrix}
					\to
					\begin{bmatrix}
						1 			&   	\cdots			& \cdots &  			1				 & 2 			 \\
						\vdots  &   		2					&   		 & \text{\huge{1}} & 1 			 \\
						\vdots  &   							& \ddots &   							 & \vdots  \\
						1 			& \text{\huge{1}} &   		 &   			2				 & \vdots  \\
						2				&  			1					& \cdots &   	\cdots			 & 1
					\end{bmatrix}
					\to
					\begin{bmatrix}
						1 			&   		1					& \cdots &  			1				 & 2 			 \\
						0  			&   		1					&   		 & \text{\huge{0}} & -1 		 \\
						\vdots  &   							& \ddots &   							 & \vdots  \\
						0 			& \text{\huge{0}} &   		 &   			1				 & -1  		 \\
						0				&  			-1				& \cdots &   		 -1				 & -3
					\end{bmatrix}
					\to
					\begin{bmatrix}
						1 			&   							& 			 &  							 & n 		   \\
										&   	\ddots			&   		 & \text{\huge{0}} & -1		   \\
										&   							& \ddots &   							 & \vdots  \\
										& \text{\huge{0}} &   		 &   			1				 & -1  		 \\
						0				&  		\cdots			& \cdots &   		 	0				 & -(n+1)
					\end{bmatrix}
					\to
					\begin{bmatrix}
						1 			&   							& 			 &  							 & n 		   \\
										&   	\ddots			&   		 & \text{\huge{0}} & 0 		   \\
										&   							& \ddots &   							 & \vdots  \\
										& \text{\huge{0}} &   		 &   			1				 & 0   		 \\
						0				&  		\cdots			& \cdots &   		 	0				 & -(n+1)
					\end{bmatrix}
					\to
					\begin{bmatrix}
						1 &   							&   		 &  							 &   \\
							&   	\ddots			&   		 & \text{\huge{0}} &   \\
							&   							& \ddots &   							 &   \\
							& \text{\huge{0}} &   		 &   			1				 &   \\
							&  								&   		 &   							 & n+1
					\end{bmatrix}
					$
				\end{center}
				Obtain the second matrix from the first matrix by swapping the first and last rows.
				Obtain the third matrix from the second by eliminating all entries in the first column except the first, by row operations using the top row.
				Obtain the fourth matrix from the third by row operations on the middle $n-2$ rows to eliminate the $1$'s and $-1$'s in the top and bottom rows.
				Obtain the fifth matrix from the fourth by using column operations on the middle $n-2$ columns to eliminate the $-1$'s in the final column.
				Finally obtain the sixth matrix from the fifth by subtracting $n$ times the first column from the last and changing the sign on the final row.
				
				Over a field of characteristic $0$, $f$ would be an isomorphism of vector spaces.
				Hence considering $(\bar{E},D)$ with coefficients in $\mathbb{Q}$ instead of $\mathbb{Z}$,
				up to multiplication by a factor in $\Gamma_{\mathbb{Q}}[x_4,x_6,\dots,x_{2n}]$,
				the homology of $(\bar{E},D)$ is described exactly as the one in Lemma \ref{lem:1to1spec},
				since rationally the $D$ differential is the same up to isomorphism $f$.
				
				Integrally the image and kernel of $D$ are finite, so $D$ still has the same rank as the differential in Lemma \ref{lem:1to1spec}.
				In particular, consider the case of generators of the form
				\begin{equation*}
					(x_2)_m\hat{y}_{i_1,\dots,i_j}P,
				\end{equation*}
				for $m\geq 1$, $1\leq j\leq n-1$ and
				$P\in\mathbb{Z}[{\gamma}_1,\dots,{\gamma}_n]\otimes\Gamma_\mathbb{Z}[x_4,\dots,x_{2n}]$.
				The image of $D$ is a subgroup of the kernel.
				Using (\ref{eq:boundary}), 
				the image of $D$ from the span of such elements and is of the form
				\begin{equation}\label{eq:ImageSpan}
					\langle D(x_m\hat{y}_{i_1,\dots,i_j}P) \rangle
					=\langle\sum_{t=1}^j{(-1)^{t-1}x_{m-1}\hat{y}_{i_1,\dots,\hat{i}_t,\dots,i_j}\tilde{\gamma}_{i_t}P}\rangle,
				\end{equation}
				where we consider $2\leq j\leq n$.
				Since the elements of the $\bar{E}$ are additivity generated as
				$\langle x_{m-1}\hat{y}_{i_1,\dots,i_{j-1}}\tilde{\gamma}_tP \rangle_{\mathbb{Z}}$,
				both the image and kernel are subgroups.
				We will show that
				\begin{equation}\label{eq:linSpans}
				\resizebox{1\hsize}{!}{$
					\langle x_{m-1}\hat{y}_{i_1,\dots,i_{j-1}}\tilde{\gamma}_tP \rangle_{\mathbb{Z}}
					\bigcap
					\langle\sum_{t=1}^j{(-1)^{t-1}x_{m-1}\hat{y}_{i_1,\dots,\hat{i}_t,\dots,i_j}\tilde{\gamma}_{i_t}P}\rangle_{\mathbb{Q}}
					=
					\langle\sum_{t=1}^j{(-1)^{t-1}x_{m-1}\hat{y}_{i_1,\dots,\hat{i}_t,\dots,i_j}\tilde{\gamma}_{i_t}P}\rangle_{\mathbb{Z}},$}
				\end{equation}
				where for ring $R$, $\langle z_1,\dots,z_a \rangle_{R}$ means the linear span of elements $z_1,\dots,z_a$ as an $R$-module. 
				(\ref{eq:linSpans}) implies that the kernel must be equal to the image.
				
				Take any element
				$A\in\langle x_{m-1}\hat{y}_{i_1,\dots,i_{j-1}}\tilde{\gamma}_tP \rangle_{\mathbb{Z}}
				\bigcap\langle\sum_{t=1}^j{(-1)^{t-1}x_{m-1}\hat{y}_{i_1,\dots,\hat{i}_t,\dots,i_j}\tilde{\gamma}_{i_t}P}\rangle_{\mathbb{Q}}$.
				Then there are $a_{i_1,\dots,i_j}\in \mathbb{Q}$ such that
				\begin{equation*}
					A=\sum_{1\leq i_1<\cdots<i_j\leq n}a_{i_1,\dots,i_j}\sum_{t=1}^j{(-1)^{t-1}x_{m-1}\hat{y}_{i_1,\dots,\hat{i}_t,\dots,i_j}\tilde{\gamma}_{i_t}P}
					\in \langle x_{m-1}\hat{y}_{i_1,\dots,i_{j-1}}\tilde{\gamma}_{i_t}P \rangle_{\mathbb{Z}}.
				\end{equation*}
				We will show that for any choice of $1\leq k_1 <\cdots<k_{j}\leq n$, $a_{k_1,\dots,k_j}\in \mathbb{Z}$.
				Since $j \geq 2$,
				we can consider the non-empty set $B=\{ (i_1,\dots,i_j) | 1\leq i_1<\cdots<i_j\leq n,\; \{k_1,\dots,k_{j-1}\} \subseteq \{i_1,\dots,i_j\} \}$.
				Note that $\sum_{t=1}^j{(-1)^{t-1}x_{m-1}\hat{y}_{i_1,\dots,\hat{i}_t,\dots,i_j}\tilde{\gamma}_{i_t}P}$
				contains a term with $\hat{y}_{k_1,\dots,k_{j-1}}$ if and only if $(i_1,\dots,i_j)\in B$.
				Assume that we have chosen the signs of $x_{m-1}\hat{y}_{i_1,\dots,\hat{i}_t,\dots,i_j}\tilde{\gamma}_{i_t}P$
				so that $(-1)^{t-1}x_{m-1}\hat{y}_{i_1,\dots,\hat{i}_t,\dots,i_j}\tilde{\gamma}_{i_t}P$ 
				have positive sign and change the signs on the $a_{i_1,\dots,i_j}$ accordingly. 
				
				Recall that $\tilde{\gamma_t}=\gamma_1+\cdots+\hat{\gamma_t}+\cdots+\gamma_n+2\gamma_t$.
				So for each $(i_1,\dots,i_j)\in B\setminus(k_1,\dots,k_j)$,
				$\pm\sum_{t=1}^j{(-1)^{t-1}x_{m-1}\hat{y}_{i_1,\dots,\hat{i}_t,\dots,i_j}\tilde{\gamma}_{i_t}P}$
				contains a unique term $x_{m-1}\hat{y}_{k_1,\dots,k_{j-1}}\tilde{\gamma}_{k_j}P$.
				\newline
				$\pm\sum_{t=1}^j{(-1)^{t-1}x_{m-1}\hat{y}_{k_1,\dots,\hat{k}_t,\dots,k_j}\tilde{\gamma}_{k_t}P}$
				contains a unique term $2x_{m-1}\hat{y}_{k_1,\dots,k_{j-1}}\tilde{\gamma}_{k_j}P$.
				Therefore
				\begin{equation}\label{eq:WightedSum}
					2a_{k_1,\dots,k_j}+\sum_{b\in B\setminus(k_1,\dots,k_j)}{a_b}\in \mathbb{Z}.
				\end{equation}
				In addition for each $(i_1,\dots,i_j)\in B$, since $j \geq 2$,
				$\pm\sum_{t=1}^j{(-1)^{t-1}x_{m-1}\hat{y}_{i_1,\dots,\hat{i}_t,\dots,i_j}\tilde{\gamma}_{i_t}P}$
				contains a unique term $x_{m-1}\hat{y}_{k_1,\dots,k_{j-1}}\tilde{\gamma}_{k_1}P$.
				Therefore
				\begin{equation}\label{eq:LevelSum}
					\sum_{b\in B}{a_b}\in \mathbb{Z}.  
				\end{equation}
				Subtracting (\ref{eq:LevelSum}) from (\ref{eq:WightedSum}) gives $a_{k_1,\dots,k_j}\in \mathbb{Z}$.
				
				It remains to deduce what effect $D^2$ has on generators of the form
				\begin{equation*}
					(x_2)_m P, \; \hat{y}_{i_1,\dots,i_j}P \text{  and  } (x_2)_m y_1\cdots y_n P
				\end{equation*}
				for $m\geq 0$, $1\leq j\leq n-1$, $1\leq i_1<\cdots,i_j\leq n$ and
				$P\in\mathbb{Z}[\gamma_1,\dots,\gamma_n]\otimes\Gamma_\mathbb{Z}[x_4,\dots,x_{2n}]$.
				
				Considering generators $(x_2)_m P$ for $m>0$.
				Since the equivalent elements in Lemma \ref{lem:1to1spec} are not contained in the kernel of $D$,
				the kernel is zero rationally therefore must also be zero integrally.
				If $m=0$, then all elements are sent to zero by $D$ and there are $\multiset{n}{p}$ in each horizontal dimension $p$.
				
				In the case $m=1$ the image of $D$ in $\langle \hat{y}_{i_1,\dots,i_j}P\rangle$ will be the same as in (\ref{eq:ImageSpan}).
				We will show that the quotient by the image still contains no torsion, hence has the same structure as Lemma \ref{lem:1to1spec}.
				For each $j\geq 1$, suppose $P$ is of the degree $p$ and $P'$ is of degree $p-1$ in their $\mathbb{Z}[\gamma_1,\dots,\gamma_n]$ components,
				with $P,P'\in \mathbb{Z}[\gamma_1,\dots,\gamma_{n}]\otimes\Gamma(x_4,x_6,\dots,x_{2n})$ monomials.
				After a choice of basis, the differential $D$ whose image lies in $\langle \hat{y}_{i_1,\dots,i_j}P\rangle$
				is represented by a matrix whose rows represent the image of a basis of the domain and columns a basis of the co-domain. 
				The quotient of the co-domain by the image is torsion free if and only if the the integral Smith normal form of this matrix
				has only ones and zeros on the leading diagonal.
				
				Recall from Remark \ref{rmk:TildeBasis} that 
				\begin{center}
					$\mathbb{Z}[\gamma_1,\dots,\gamma_{n}]=\mathbb{Z}[\tilde{\gamma}_1,\dots,\tilde{\gamma}_{n-1},\bar{\gamma}]$ and
					$\mathbb{Q}[\gamma_1,\dots,\gamma_{n}]=\mathbb{Q}[\tilde{\gamma}_1,\dots,\tilde{\gamma}_{n}]$.
				\end{center}
				
				Choosing the rational basis of the domain $\{x_2 \hat{y}_{i_{1},\dots,i_{j+1}}\tilde{\gamma}_{a_1}\cdots\tilde{\gamma}_{a_{p-1}} \}$,
				where $1\leq a_1\leq \cdots \leq a_{p-1} \leq n$
				and the rational basis $\{ \hat{y}_{i_{1},\dots,i_{j}}\tilde{\gamma}_{a_1}\cdots \tilde{\gamma}_{a_{p}} \}$ for the image,
				where $1\leq a_1\leq \cdots \leq a_{p} \leq n$.
				With this choice of basis the image of the differential are the same as that of ${d}$ in Lemma \ref{lem:1to1spec} given in equation (\ref{eq:boundary}),
				when $\gamma_i$ are replaced by $\tilde{\gamma}_i$.
				As there is no torsion in Lemma \ref{lem:1to1spec},
				using integral row and column operations the matrix corresponding to these basis can be brought to the smith normal form
				with only ones and zeros on the leading diagonal.
				
				Now choose a basis of the image using the integral basis of Remark \ref{rmk:TildeBasis} generated by
				\begin{equation*}
					x_2\hat{y}_{i_{1},\dots,i_{j+1}}\tilde{\gamma}_{a_1}\cdots\tilde{\gamma}_{a_{p-k}}\bar{\gamma}^{k}
				\end{equation*}
				with $0\leq k \leq p$ for the domain and
				\begin{equation*}
					\hat{y}_{i_{1},\dots,i_{j}}\tilde{\gamma}_{a_1}\cdots\tilde{\gamma}_{a_{p-k'-1}}\bar{\gamma}^{k'}
				\end{equation*}
				with $1\leq k' \leq p-1$ for the co-domain.
				Rearrange the rows and columns of the matrix corresponding to these bases such that
				the columns of the form $\hat{y}_{i_1,\dots,i_j}P$ for $i_k\neq n$ are on the left
				and the columns of the form $\hat{y}_{i_1,\dots,i_{j-1},n}P$ are on the right.
				The rows of the form $x_2 \hat{y}_{i_1,\dots,i_{j+1}}P'$ are at the top and the rows of the form $x_2 \hat{y}_{i_1,\dots,i_j,n}P'$ are at the bottom.
				The sub-matrix in the intersection of rows $x_2 \hat{y}_{i_1,\dots,i_{j+1}}P'$ and columns $\hat{y}_{i_1,\dots,i_{j-1},n}P$ is zero
				because none of the $i_k$ are equal to $n$, hence the image of the differential contains no summand divisible by a $\hat{y}_{i_1,\dots,i_{j-1},n}$.
				Label the remaining three sub-matrices $A$, $B$ and $C$ as in the diagram below.
				\[
				\begin{blockarray}{ccc}
					 & \hat{y}_{i_1,\dots,i_j}P & \hat{y}_{i_1,\dots,i_{j-1},n}P \\
					\begin{block}{c(c|c)}
						x_2\hat{y}_{i_1,\dots,i_{j+1}}P' & {A} & {0} \\
						\cline{2-3}
						x_2\hat{y}_{i_1,\dots,i_j,n}P' 	 & {B} & {C} \\
					\end{block}
				\end{blockarray}
				\]
				The sub-matrix $A$ can be further broken down as a diagonal sum as follows
				\[
				\begin{blockarray}{cccccc}
					 & \hat{y}_{i_1,\dots,i_j}\tilde{P} & \hat{y}_{i_1,\dots,i_j}\tilde{P}\bar{\gamma} & \cdots & \hat{y}_{i_1,\dots,i_j}\tilde{P}\bar{\gamma}^{p-2} & \hat{y}_{i_1,\dots,i_j}\tilde{P}\bar{\gamma}^{p-1} \\
					\begin{block}{c(ccccc)}
						x_2\hat{y}_{i_1,\dots,i_{j+1}}\tilde{P}						 			 	& {A}_0  & {0}	 & \cdots	& {0}				& {0}	\\
						x_2\hat{y}_{i_1,\dots,i_{j+1}}\tilde{P}\bar{\gamma} 			& {0}	   & {A}_1 & 				& {0}				& {0}	 \\
						\vdots																			 		 			 		& \vdots & 			 & \ddots & 					& \vdots			\\
						x_2\hat{y}_{i_1,\dots,i_{j+1}}\tilde{P}\bar{\gamma}^{p-2} & {0}	   & 	{0}  & 				& {A}_{d-2} & {0}		 \\
						x_2\hat{y}_{i_1,\dots,i_{j+1}}\tilde{P}\bar{\gamma}^{p-1} & {0}	   & 	{0}  & \cdots & {0}				& {A}_{d-1} \\
					\end{block}
				\end{blockarray}
				\]
				where $\tilde{P}$ is some monomial in $\mathbb{Z}[\tilde{\gamma}_1,\dots,\tilde{\gamma}_{n-1}]\otimes\Gamma[x_2,x_4,\dots,x_{2n}]$.
				The sub-matrix in the intersection of 
				$x_2\hat{y}_{i_1,\dots,i_{j+1}}\tilde{P}\bar{\gamma}^{a}$
				and
				$\hat{y}_{i_1,\dots,i_{j}}\tilde{P}\bar{\gamma}^{b}$
				for $a\neq b$
				will be zero since the image of $D$ on $x_2\hat{y}_{i_1,\dots,i_{j+1}}$ in the basis $\hat{y}_{i_1,\dots,i_j}P$
				will not be divisible by $\bar\gamma$,
				hence summands in the image of $d^2$ on $x_2\hat{y}_{i_1,\dots,i_{j+1}}\tilde{P}\bar{\gamma}^{a}$ will each contain a multiple of exactly $\bar{\gamma}^a$.
				After dividing the $A_k$ by $\bar{\gamma}^k$, each $A_k$
				is the same as the matrix with respect to the rational basis $\tilde{\gamma}_1,\dots,\tilde{\gamma}_{n-1}$
				if we reduce the value of $n$ by 1 and the degree of the polynomial components in $\mathbb{Z}[\tilde{\gamma}_1,\dots\tilde{\gamma}_{n-1}]$ by $k$.
				
				Similarly the sub-matrices $C$ is the same as a diagonal sum of matrices with respect to the rational basis with rows interchanged
				$x_2\hat{y}_{i_1,\dots,i_j,n}P\mapsto x_2\hat{y}_{i_1,\dots,i_j}P$
				and columns interchanged $\hat{y}_{i_1,\dots,i_{j-1},n}P\mapsto \hat{y}_{i_1,\dots,i_{j-1}}P$.
				
				Hence there exists integral row and column operations on the whole matrix that bring $A$ and $C$ to the Smith normal form
				with only ones and zeros on the leading diagonal.
				
				Every row of the form $x_2\hat{y}_{i_1,\dots,i_j,n}P$ has a non-zero entry in $C$.
				Every row reduced to zero while putting $C$
				into its Smith normal form corresponds to an element of the image of $d^2$.
				Using the previous part of the proof, we know that the kernel of $D$ on $x_2\hat{y}_{i_1,\dots,i_j,n}P$ is exactly the image of $D$
				whose image is the previous domain.
				Given a row in $C$ that was in the kernel implies it is the image of some element of the form $x_2\sum{\hat{y}_{i_1,\dots,i_{j}}\tilde{P}}$
				of the previous differential under the correspondence used to obtain the Smith normal form.
				In this case for some $1\leq k\leq p-1$, the image of $x_2\sum{\hat{y}_{i_1,\dots,i_{j},n}\tilde{P}\bar{\gamma}^k}$
				under $d^2$ is the row inducing this row of $C$ in the larger matrix.
				Hence corresponding row in the larger matrix will still be in the image of $D$ and therefore in the kernel.
				So the whole row in the matrix can be is reduced to zero not just the row in $C$.
				Any remaining entries in $B$ can then be reduced to zero by column operations cancelling them with using a column in $C$.
				Therefore $B$ is reduced to zero, while $A$ and $C$ are reduced to the Smith normal form with only ones and zeros on the leading diagonal. 
				Hence the whole matrix has a Smith normal form with only ones and zeros on the leading diagonal
				so has the same Smith normal form as that with respect to the rational basis.

				Finally consider generators of the form $(x_2)_m y_1\cdots y_n P$. Their image under $D$ is also zero.
				When $\deg(P)=0$, there is no differential with image in $(x_2)_m y_1\cdots y_n P$, so it contributes a copy of $\mathbb{Z}$ to the homology of $\bar{E}$.
				Note that up to sign for $m'\geq 1$, $1\leq i\leq n$ and
				$P'\in\mathbb{Z}[\gamma_1,\dots,\gamma_n]\otimes\Gamma_\mathbb{Z}[x_4,\dots,x_{2n}]$ a monomial,
				\begin{equation*}
					D((x_2)_{m'}\hat{y}_i P)=(x_2)_{{m'}-1}\tilde{\gamma}_i P'.
				\end{equation*}
				In particular $D((x_2)_{m'}(\hat{y}_i-\hat{y}_j)P)=(x_2)_{m'-1}(\gamma_i-\gamma_j)P'$.
				So for fixed $\deg{P}\geq 1$ on in the homology of $(\bar{E},D)$, all elements of the form $(x_2)_m y_1\cdots,y_n P$ become identified.
				The number of terms in $D((x_2)_{m'}\hat{y}_i P)$ is the number of terms in $\tilde{\gamma}_i$, is $n+1$.
				So the elements $(x_2)_m y_1\cdots,y_n P$ contribute a copy of $\mathbb{Z}_{n+1}$ in the homology of $(\bar{E},D)$.
			\end{proof}
			
		\subsection{Third page}\label{sec:Third}
			
			We now turn our attention to applying the results of sections \ref{sec:basis}, \ref{sec:PreQuotient}
			and  \ref{subsec:IdeaForms}
			to produce information about the spectral sequence $\{ E_r,d^r \}$. 
			Determining The $E_3$-page everywhere would be difficult, however in special cases the problem is considerably simplified.
			Throughout this section assume $X\in \Gamma_{\mathbb{Z}}[x_4,\dots,x_{2n}]$ is a monomial.
			
			\begin{thm}\label{thm:BottemRow}
				For each $m\geq 0$,
				\begin{align*}
					E_3^{0,n+2m+\dim(X)}((x_2)_m X y_1\cdots y_n)&\cong\mathbb{Z}, \\
					E_3^{p,n+2m+\dim(X)}((x_2)_{m}X y_1\cdots y_n)&\cong\mathbb{Z}_{\gcd(\binom{n+1}{1},\dots,\binom{n+1}{p})} \; \text{for} \; p>0.
				\end{align*}
			\end{thm}
			\begin{proof}
				By Theorem \ref{thm:SymFree}, in the differential bigraded algebra $(\bar{E},D)$, which is the same as $\{ E_2,d^2\}$ 
				before quotienting out by the symmetric ideal 
				\begin{align*}
					\bar{E}_{3}^{0,n+2m+\dim(X)}((x_2)_m X y_1\cdots y_n)&\cong\mathbb{Z}, \\
					\bar{E}_{3}^{p,n+2m+\dim(X)}((x_2)_m X y_1\cdots y_n)&\cong\mathbb{Z}_{n+1} \; \text{for} \; p>0.
				\end{align*}
				Since the smallest degree of $\sigma_2,\dots,\sigma_n$ is degree $2$,
				$\bar{E}_{3}^{0,n+2m+dim(X)}$ will remain unchanged after tacking the symmetric quotient. 
				Recall from (\ref{eq:d^2}) that for any $1\leq i \leq n$
				\begin{equation*}
					d^2((x_2)_{m+1}X\hat{y}_{i})=(x_2)_mX\tilde{\gamma}_i.
				\end{equation*}
				In particular this implies that for any $1\leq i,j \leq n$
				\begin{equation}\label{eq:exchange}
					d^2((x_2)_{m+1}X(\hat{y}_i-\hat{y}_j))=(x_2)_mX(\gamma_i-\gamma_j).
				\end{equation}
				Hence in the quotient by the image of the differential there is at most one generator, as all generators of the from $(x_2)_mX\gamma_i$ are identified.
				This remains true for all $E_3^{p,n+2m+dim(X)}((x_2)_{m}X y_1\cdots y_n)$ when $p>0$ since (\ref{eq:exchange})
				can be multiplied by any element of $\mathbb{Z}[\gamma_1,\dots,\gamma_n]$.
				Consequently in the quotient by the image of the differential the expressions for $\sigma_1,\dots,\sigma_{n+1}$
				can be identifies with an expression in just one generator.
				Such an expression would consist of the number of summands in $\sigma_i$ times a generator for each $1\leq i \leq n+1$.
				By Remark \ref{remk:SymQotForms}, we may assume that $\sigma_2,\dots,\sigma_{n+1}=h_2^{n},\dots,h_{n+1}^1$.
				The number of summand in $h_j^{n-j+2}$ is $\multiset{n-j+2}{j}=\binom{n+1}{j}$.
				Tacking into account the degrees of $h_2^{n-1},\dots,h_n^1$, we arrive at the statement of the theorem.
				
			\end{proof}

			\begin{thm}\label{thm:LastColumn}
				For each $2 \leq k\leq n$, $1\leq i_1 <\cdots< i_k \leq n$, $1\leq i \leq n$ and $m\geq 0$
				\begin{align*}
					E_3^{(n+1)n/2,2m+\dim(X)}((x_2)_mX)													&	\cong		\mathbb{Z}, \\
					E_3^{(n+1)n/2,n-1+2m+\dim(X)}((x_2)_mXy_i)										&	\cong		\mathbb{Z}_{n+1}, \\
					E_3^{(n+1)n/2,k+2m+\dim(X)}((x_2)_mXy_{i_1}\cdots y_{i_k})		&	\cong		0.
				\end{align*}
			\end{thm}
			\begin{proof}
				By Theorem \ref{thm:AddBasis}
				any element in $E_2^{(n+1)n/2,*}$ is always in the kernel of $d^2$, since the domain of the differentiate will be zero.
				Hence we consider the quotient of $E_2^{(n+1)n/2,-}$ by the image of $d^2$.
				Using (\ref{eq:boundary}) in the proof of Lemma \ref{lem:1to1spec}, for any $1\leq b\leq n$
				\begin{equation*}
					d^2((x_2)_{m}X\hat{y}_{{i_1},\dots,{i_j}}\hat{\gamma}_b)=(x_2)_{m-1}X\sum_{a=1}^{j}{(-1)^a\hat{y}_{i_1,\dots\hat{i}_a,\dots,i_j}\tilde{\gamma}_a\hat{\gamma}_b}.
				\end{equation*}
				Recall from Lemma \ref{lem:OneMissing}, that for any $1\leq i,j\leq n$
				\begin{equation*}
					\gamma_j\hat{\gamma_i}=
					\begin{cases}
							\;\;\; [0]  &\mbox{if  } j < i \text{ or } j \geq i+2 
							\\
							\;\; [\hat{\gamma}_{\emptyset}] &\text{if  } j=i
							\\
							-[\hat{\gamma}_{\emptyset}] &\text{if  } j=i+1
							.
					\end{cases}
				\end{equation*}
				Therefore for any $1\leq i,\leq n$ and $1 \leq j \leq n-1$
				\begin{equation*}
					\tilde{\gamma_j}\hat{\gamma_i}=
					\begin{cases}
							\;\;\; [0]  &\mbox{if  } j < i \text{ or } j \geq i+2 
							\\
							\;\; [\hat{\gamma}_{\emptyset}] &\text{if  } j=i
							\\
							-[\hat{\gamma}_{\emptyset}] &\text{if  } j=i+1
							.
					\end{cases}
				\end{equation*}
				Hence we deduce that for each $2\leq j \leq n$, $1\leq i_1<\cdots<i_j\leq n$, $m\geq 1$ and $1\leq b \leq n-1$ such that $b\neq i_k,n$ for any $1\leq k \leq j$,
				\begin{align*}
					d^2&([{x_2}_mX\hat{y}_{i_1,\dots,i_k,b,i_{k+2},\dots,i_j}\hat\gamma_b])
					\\
					\\
					&=
					\begin{cases}
							\;\; [\hat{\gamma}_{\emptyset}(-1)^{k}(\hat{y}_{i_1,\dots,i_j}-\hat{y}_{i_1,\dots,i_{k-2},b,i_k,\dots,i_j})] &\text{if  } i_{k}=b+1 \text{ for some } 1\leq k\leq j
							\\
							\;\; [\hat{\gamma}_{\emptyset}(-1)^{k}\hat{y}_{i_1,\dots,i_j})] &\text{if  } i_{k}\neq b+1 \text{ for some } 1\leq k\leq j.
					\end{cases}
				\end{align*}
				Therefore for $j\leq n-2$, in the quotient by the image, if there exits $1\leq b\leq n-1$ such that for all $1\leq k \leq j$,
				$i_k\neq b$ and $i_k \neq b+1$ then $[\hat{y}_{i_1,\dots,i_j}]=0$.
				If not, then there exists a smallest $1\leq b\leq n-1$ such that for any $1\leq k \leq j$, $i_k\neq b$ in which case
				\begin{equation}\label{eq:step}
				[\hat{y}_{i_1,\dots,i_j}]=[\hat{y}_{i_1,\dots,i_{k},b,i_{k+1},\dots,i_j}].
				\end{equation}
				We can think of this as moving up the position of the missing integer in sequence $i_1,\dots,i_j$. 
				Since we assume $j\leq n-2$, there are at least $2$ integers between $1$ and $n$ that do not occur in the sequence $i_1,\dotso,i_j$.
				The index $b$ was chosen to be the smallest such integer so if in $i_1,\dots,i_j$ does not have two missing elements next to each other by repeated application of
				(\ref{eq:step}),
				$[\hat{y}_{i_1,\dots,i_j}]$=$[\hat{y}_{s_1,\dots,s_j}]$ where $s_1,\dots,s_j$ does have two consecutive gaps and therefore $[\hat{y}_{i_1,\dots,i_j}]=0$.
				Hence $[\hat{\gamma}_{i_1,\dots,i_j}]=0$ for any choice of $i_1,\dots,i_j$.
				Which proves that $E_3^{(n+1)n/2,k+2m+\dim(X)}((x_2)_mXy_{i_1}\cdots y_{i_k})	=	0$ 
				since $\hat{\gamma}_{i_1,\dots,i_j}$ span $E_3^{(n+1)n/2,k+2m+\dim(X)}((x_2)_mXy_{i_1}\cdots y_{i_k})$.
				
				For each $1\leq i \leq n$,
				\begin{equation*}
					d^2((x_2)_mX\hat{\gamma}_i)=
					\hat{\gamma}_{\emptyset}\sum_{1\leq j\leq n, j\neq i}{(-1)^j\hat{y}_{1,\dots,\hat{j},\dots,n}+2(-1)^i\hat{y}_{1,\dots,\hat{i},\dots,n}},
				\end{equation*}
				Therefore, for each $1\leq i <j \leq n$
				\begin{equation*}
					d^2((x_2)_mX{\hat{\gamma}_i-\hat{\gamma}_j})
					=\hat{\gamma}_{\emptyset}((-1)^j\hat{y}_{1,\dots,\hat{j},\dots,n}+(-1)^i\hat{y}_{1,\dots,\hat{i},\dots,n}).
				\end{equation*}
				Hence $E_3^{(n+1)n/2,n-1+2m+dim(X)}((x_2)_mXy_i)$ has a single generator of which each $(x_2)_mX{\hat{\gamma}_i}$ is a representative.
				As the number of summands in the image of the differential on each generator $(x_2)_mX\hat{\gamma_i}$ of $E_2^{(n+1)n/2-2,2m+dim(X)}((x_2)_mX)$
				is $n+1$, the generator of $E_3^{(n+1)n/2,n-1+2m+dim(X)}((x_2)_mXy_i)$ is torsion and has multiplicity $n+1$.
				A generator of $E_2^{(n+1)n/2,2m+dim(X)}((x_2)_mX)$ is not in the image of any differential hence survives to the next page.
			\end{proof}

		\subsection{Free loop cohomology of $SU(3)/T^2$}\label{sec:L(SU(3)/T^2)}
		
			When $n=0$, $SU(n+1)/T^n$ is a point and when $n=1$, it has the homotopy type of $S^2$.
			Hence in the first case the free loops cohomology is trivial and in the second the cohomology ring is known.
			We now use some of the tools developed in Sections \ref{subsec:BettiNum}, \ref{subsec:IdeaForms}
			and Section \ref{sec:diff}, \ref{sec:PreQuotient} and \ref{sec:Third}
			to study $H^*(\Lambda(SU(3)/T^{2});\mathbb{Z})$.
			
			\begin{thm}\label{thm:L(SU(3)/T^2)}
				
				The integral algebra structure of the $E_\infty$-page of the Leray-Serre spectral sequence associated to the free loop space fibration of $\Lambda(SU(3)/T^2)$
				is $A/I$, where
				\begin{equation*}
					A=\Lambda_{\mathbb{Z}}(\gamma_i,\; (x_4)_m,\; y_i,\; (x_2)_m(y_1(\gamma_1+\gamma_2)-y_2\gamma_2),\; (x_2)_my_2(\gamma_1^2-\gamma_1\gamma_2),\;
					(x_2)_m\gamma_1^2\gamma_2)
				\end{equation*}
				and
				\begin{align*}
					I=[(x_2)_1^m-m!(x_2)_m,\; (x_4)_1^m-m!(x_4)_m,\; \gamma_1^2+\gamma_2^2+\gamma_1\gamma_2,\;
					\gamma_1^3,\; y_1(2\gamma_1+\gamma_2)-y_2(\gamma_1+2\gamma_2),\; \\
					3(x_2)_m(y_1\gamma_1^2+y_2\gamma_2^2),\; (x_2)_my_1y_2(\gamma_1-\gamma_2),\; 3(x_2)_my_1y_2\gamma_1,\; (x_2)_my_1y_2\gamma_1^2\gamma_2]
				\end{align*}
				where $m\geq 1$,
				$|\gamma_i|=2$, $|y_i|=1$, $|(x_2)_k|=2k$ and $|(x_4)_k|=4k$ for $1\leq i \leq n$ and $1\leq k$.
				Furthermore all additive extension problems are trivial, hence the algebra has the same module structure as $H^*(\Lambda(SU(3)/T^{2});\mathbb{Z})$.
			\end{thm}
			
			\begin{proof}
				We consider the cohomology Leray-Serre spectral $\{ E_r,d^r \}$ sequence associated to the free loop fibration of $SU(n+1)$ studies in section \ref{sec:diff},
				in the case $n=2$, that is
				\begin{equation*}
					\Omega(SU(3)/T^2)\to\Lambda(SU(3)/T^2)\to SU(3)/T^2.
				\end{equation*}
				The cohomology of the base space $SU(3)/T^2$ is $\frac{\mathbb{Z}[\gamma_1,\gamma_2]}{[\sigma_2,\sigma_3]}$,
				where $|\gamma_1|=2=|\gamma_2|$.
				By Remark \ref{remk:SymQotForms} we may replace
				$\sigma_2$ with $\gamma_1^2+\gamma_2^2+\gamma_1\gamma_2$ and $\sigma_3$ with $\gamma_1^3$.
				Noteing also by symmetry that $\gamma_2^3\in[\sigma_2,\sigma_3]$
				and that $\gamma_1^2\gamma_2+\gamma_1\gamma_2^2=\gamma_1\sigma_2-\sigma_3\in[\sigma_2,\sigma_3]$.
				
				The cohomology of the fibre $\Omega(SU(3)/T^2)$ is $\Lambda(y_1,y_2)\otimes\Gamma[x_2,x_4]$
				where $|y_1|=1=|y_2|$, $|x_2|=2$ and $|x_4|=4$.
				In particular $\Lambda(y_1,y_2)$ is an exterior algebra
				and 
				\begin{equation*}
					\Gamma[x_2,x_4]=\frac{\mathbb{Z}[(x_2)_1,(x_2)_2,\dots,(x_4)_1,(x_4)_2,\dots]}{[(x_2)_1^m-m!(x_2)_m, (x_4)_1^m-m!(x_4)_m]}
				\end{equation*}
				is a divided polynomial algebra, where $(x_2)_1=x_2$ and $(x_4)_1=x_4$.
				Hence elements on the $E_2$-page of the spectral sequence are generated additively by representative elements
				of the form
				\begin{equation*}
					(x_2)_a(x_4)_bP, \;\; (x_2)_a(x_4)_by_iP, \;\; (x_2)_a(x_4)_by_1y_2P
				\end{equation*}
				where $0\leq a,b$, $1\leq i \leq n$ and $P\in \mathbb{Z}[\gamma_1,\gamma_2]$ is a monomial of degree between $0$ and $3$.
				By Theorem, \ref{thm:allDiff} the only non-zero differentials are $d^2$ and $d^4$,
				which are non-zero only on generators $x_2$ and $x_4$ respectively.
				The differentials up to sign are given by
				\begin{equation*}
					d^2([x_2])=[y_1(\gamma_1+2\gamma_2)+y_2(2\gamma_1+\gamma_2)], \;\;\;
					d^4([x_4])=[y_1(\gamma_1^2+2\gamma_1\gamma_2)+y_2(\gamma_2^2+2\gamma_1\gamma_2)].
				\end{equation*}
				However, 
				\begin{align*}
					&d^2([(\gamma_1+\gamma_2)x_2]) \\
					&=[y_1(2\gamma_1^2+3\gamma_1\gamma_2+\gamma_2^2)+y_2(\gamma_1^2+3\gamma_1\gamma_2+2\gamma_2^2)] \\
					&=[y_1(\gamma_1^2+2\gamma_1\gamma_2)+y_2(\gamma_2^2+2\gamma_1\gamma_2)] \\
					&=d^4([x_4])
				\end{align*}
				where the second equality is given by subtracting element of the symmetric ideal $y_i(\gamma_1^2+\gamma_2^2+\gamma_1\gamma_2)$ for $i=1,2$, 
				from $y_1(2\gamma_1^2+3\gamma_1\gamma_2+\gamma_2^2)+y_2(\gamma_1^2+3\gamma_1\gamma_2+2\gamma_2^2)$.
				Hence $d^4$ is trivial, and the spectral sequence converges by the third page.
				The generators $\gamma_i$, $x_4$ and $y_i$ occur in $E_2^{*,0}$ and are always in the kernel of the differentials,
				so are generators of the $E_\infty$-page.
				The relations $x_2^m-m!(x_2)_m$, $x_4^m-m!(x_4)_m$ from the divide polynomial algebra in $H^*(\Omega(SU(3)/T^2);\mathbb{Z})$
				and $\gamma_1^2+\gamma_2^2+\gamma_1\gamma_2$, $\gamma_1^3$ generators of the symmetric ideal in $H^*(SU(3)/T^2;\mathbb{Z})$
				will also be relations in $H^*(\Lambda(SU(3)/T^2);\mathbb{Z})$,
				so are generators of the ideal $I$.
				
				We choose the opposite sign on $y_1$ so that
				\begin{equation}\label{eq:diff2}
					d^2([x_2])=[y_2(2\gamma_1+\gamma_2)-y_1(\gamma_1+2\gamma_2)],
				\end{equation}
				which means that 
				\begin{equation}\label{eq:diff2y}
					d^2([x_2y_1])=[y_1y_2(2\gamma_1+\gamma_2)], \;\; d^2([x_2y_2])=[y_1y_2(\gamma_1+2\gamma_2)].
				\end{equation}
				
				We first consider the image and kernel of the differential $d^2$ on elements of the $E^2$-page of the form
				\begin{equation*}
					[(x_2)_a(x_4)_by_1y_2P].
				\end{equation*}
				By Theorem \ref{thm:BottemRow}, on the $E_3$-page when the degree of $P$ is zero all generators of the $E_2$-page survive.
				When the degree of $P$ is the $1$ or $2$ the only non-trivial elements are $3$-torsions generated by the class of any non-trivial representative from the $E_2$-page
				and when the degree of $p$ is $3$ all elements represent the trivial element.
				Hence on the $E_\infty$-page
				requires generator of the form
				\begin{equation*}
					(x_2)_ay_1y_2
				\end{equation*}
				and the ideal $I$ should contain generators of the form
				\begin{equation*}
					y_1y_2(\gamma_1-\gamma_2)(x_2)_a,\;\; 3y_1y_2(x_2)_a,\;\; y_1y_2(x_2)_a\gamma_1^2\gamma_2.
				\end{equation*}
				It remains to deduce the kernel of $d^2$ with codomain in $\langle (x_2)_a(x_4)_by_1y_2P \rangle$.
				
				By Theorem \ref{thm:SymFree}, before the quotient of the symmetric ideal on subgroups of $E_2^{p,q}$ where $d^2$ is non-trivial for both the differentials
				\begin{align*}
													&d^2\colon E_2^{p-2,q+1} \to E_2^{p,q} \\
					\text{and} \;\; &d^2\colon E_2^{p,q} \to E_2^{p+2,q-1}
				\end{align*} 
				the kernel of $d^2$ is exactly the image of $d^2$.
				Hence kernel elements that can be represented by a non-trivial element on the $E_3$-page are those that have image under $d^2$
				of summands dividable by non-trivial element of the symmetric ideal.
				
				For elements of the form $[(x_2)_a(x_4)_by_1y_2P]$, when the degree of $P$ is $0$ or $1$, the kernel of $d^2$ quotient the image of $d^2$ must be trivial
				since the degree of the first generator of the symmetric ideal has degree $2$.
				When the degree of $P$ is $2$, by (\ref{eq:diff2y}), the image of $d^2$ is generated by
				\begin{align*}
					d^2([(x_2)_{a+1}(x_4)_by_1\gamma_1])=[(x_2)_{a}(x_4)_by_1y_2(2\gamma_1^2+\gamma_1\gamma_2)], \\ 
					d^2([(x_2)_{a+1}(x_4)_by_1\gamma_2])=[(x_2)_{a}(x_4)_by_1y_2(2\gamma_1\gamma_2+\gamma_2^2)], \\
					d^2([(x_2)_{a+1}(x_4)_by_2\gamma_1])=[(x_2)_{a}(x_4)_by_1y_2(\gamma_1^2+2\gamma_1\gamma_2)], \\
					d^2([(x_2)_{a+1}(x_4)_by_2\gamma_2])=[(x_2)_{a}(x_4)_by_1y_2(\gamma_1\gamma_2+2\gamma_2^2)].
				\end{align*}
				The rank of the codomain is $\triset{2}{2}=2$ and the dimension of the domain $2\triset{2}{1}=4$.
				We know that the dimension of the image is $2$, so by the rank nullity theorem the dimension of the kernel must be $2$.
				By (\ref{eq:diff2}), $[(x_2)_{a+1}(x_4)_b(y_2(2\gamma_1+\gamma_2)-y_1(\gamma_1+2\gamma_2))]$ is the image of the previous differential and so is in the kernel.
				Since $d^2([(x_2)_{a+1}(x_{4})_b(y_1(\gamma_1+\gamma_2)-y_2\gamma_2)])=h_2^2$, $[(x_2)_{a+1}(x_{4})_b(y_1(\gamma_1+\gamma_2)-y_2\gamma_2)]$
				can be taken to be the other generator of the kernel.
				Hence  
				\begin{equation*}
					(y_1(\gamma_1+\gamma_2)-y_2\gamma_2)(x_2)_m
				\end{equation*}
				is a generator of $H^*(\Lambda(SU(3)/T^2);\mathbb{Z})$.
				When the degree of $P$ is $3$, by (\ref{eq:diff2y}) the image of $d^2$ is generated by
				\begin{align}
					d^2([(x_2)_{a+1}(x_4)_by_1\gamma_1^2])			&=[(x_2)_{a}(x_4)_by_1y_2(\gamma_1^3+2\gamma_1^2\gamma_2)]				&=&[2(x_2)_{a}(x_4)_by_1y_2\gamma_1^2\gamma_2]					,\label{1}\\ 
					d^2([(x_2)_{a+1}(x_4)_by_1\gamma_1\gamma_2])&=[(x_2)_{a}(x_4)_by_1y_2(\gamma_1^2\gamma_2+2\gamma_1\gamma_2^2)]&=&[-(x_2)_{a}(x_4)_by_1y_2\gamma_1^2\gamma_2]					,\label{2}\\
					d^2([(x_2)_{a+1}(x_4)_by_1\gamma_2^2])			&=[(x_2)_{a}(x_4)_by_1y_2(\gamma_1\gamma_2^2+2\gamma_2^3)]				&=&[-(x_2)_{a}(x_4)_by_1y_2\gamma_1^2\gamma_2]					,\label{3}\\
					d^2([(x_2)_{a+1}(x_4)_by_2\gamma_1^2])			&=[(x_2)_{a}(x_4)_by_1y_2(2\gamma_1^3+\gamma_1^2\gamma_2)]				&=&[(x_2)_{a}(x_4)_by_1y_2\gamma_1^2\gamma_2]					 ,\label{4}\\
					d^2([(x_2)_{a+1}(x_4)_by_2\gamma_1\gamma_2])&=[(x_2)_{a}(x_4)_by_1y_2(2\gamma_1^2\gamma_2+\gamma_1\gamma_2^2)]&=&[(x_2)_{a}(x_4)_by_1y_2\gamma_1^2\gamma_2]					 ,\label{5}\\
					d^2([(x_2)_{a+1}(x_4)_by_2\gamma_2^2])			&=[(x_2)_{a}(x_4)_by_1y_2(2\gamma_1\gamma_2^2+\gamma_2^3)]        &=&[-2(x_2)_{a}(x_4)_by_1y_2\gamma_1^2\gamma_2] 				  \label{6}
				\end{align}
				where the last equalities are given by adding an element of the symmetric ideal to the representatives.
				Using the numbering of the equations to represent the generators in the domain of $d^2$, we may take the kernel to be generated by
				\begin{equation*}
					(\ref{6})+(\ref{1}),\;\;(\ref{5})+(\ref{3}),\;\;(\ref{4})+(\ref{2}),\;\;(\ref{4})-(\ref{5}),\;\;(\ref{4})+(\ref{5})+(\ref{6}).
				\end{equation*}
				The symmetric ideal in the domain is generated by
				\begin{equation*}
					(\ref{1})+(\ref{2})+(\ref{3}),\;\;(\ref{4})+(\ref{5})+(\ref{6}).
				\end{equation*}
				By (\ref{eq:diff2}), the image of the previous differential is generated by
				\begin{align*}
					d^2([(x_2)_{a+1}(x_4)_b\gamma_1])
					&=[(x_2)_{a}(x_4)_b(y_2(2\gamma_1^2+\gamma_1\gamma_2)-y_1(\gamma_1^2+2\gamma_1\gamma_2))] \\
					&=[(x_2)_{a}(x_4)_b(y_2(\gamma_1^2-\gamma_2^2)-y_1(\gamma_1\gamma_2-\gamma_2^2))] \\
					&=(\ref{6})+(\ref{1})-(\ref{4})-(\ref{2}), \\
					d^2([(x_2)_{a+1}(x_4)_b\gamma_2])
					&=[(x_2)_{a}(x_4)_b(y_2(2\gamma_1\gamma_2+\gamma_2^2)-y_1(\gamma_1\gamma_2+2\gamma_2^2))]\\
					&=[(x_2)_{a}(x_4)_b(y_2(\gamma_1\gamma_2-\gamma_1^2)-y_1(\gamma_2^2-\gamma_1^2))] \\
					&=(\ref{5})+(\ref{3})-(\ref{6})-(\ref{1}).
				\end{align*}
				Hence the quotient of the kernel by the image is given by
				\begin{equation*}
					\resizebox{1\hsize}{!}{$
					\frac{\langle(\ref{6})+(\ref{1}),\;\;(\ref{5})+(\ref{3}),\;\;(\ref{4})+(\ref{2}),\;\;(\ref{4})-(\ref{5}),\;\;(\ref{4})+(\ref{5})+(\ref{6})\rangle}
					{\langle(\ref{1})+(\ref{2})+(\ref{3}),\;(\ref{4})+(\ref{5})+(\ref{6}),
					\;(\ref{6})+(\ref{1})-(\ref{4})-(\ref{2}),\;(\ref{5})+(\ref{3})-(\ref{6})-(\ref{1})\rangle}.$}
				\end{equation*}
				Subtracting $(\ref{6})+(\ref{1})$ from $(\ref{5})+(\ref{3})$, adding $(\ref{6})+(\ref{1})$ to $-((\ref{4})+(\ref{2}))$ in the generators of the kernel
				and adding $(\ref{1})+(\ref{2})+(\ref{3})$, $(\ref{5})+(\ref{3})-(\ref{6})-(\ref{1})$ and $-((\ref{6})+(\ref{1})-(\ref{4}))$ 
				to $(\ref{4})+(\ref{5})+(\ref{6})$ in the generators of the image gives
				\begin{equation*}
					\resizebox{1\hsize}{!}{$
					\frac{\langle(\ref{6})+(\ref{1}),\;(\ref{5})+(\ref{3})-(\ref{6})-(\ref{1}),
					\;(\ref{6})+(\ref{1})-(\ref{4})-(\ref{2}),\;(\ref{4})-(\ref{5}),\;(\ref{4})+(\ref{5})+(\ref{6})\rangle}
					{\langle(\ref{1})+(\ref{2})+(\ref{3}),\;3((\ref{6})+(\ref{1})),
					\;(\ref{6})+(\ref{1})-(\ref{4})-(\ref{2}),\;(\ref{5})+(\ref{3})-(\ref{6})-(\ref{1})\rangle}.$}
				\end{equation*}
				Therefore the kernel of $d^2$ is generated by
				\begin{align*}
					(\ref{6})+(\ref{1}) \;\; \text{and} \;\; (\ref{4})-(\ref{5}).
				\end{align*}
				Recall that $[(y_1(\gamma_1+\gamma_2)-y_2\gamma_2)(x_2)_m]$ generated the kernel when the degree of $P$ was 2.
				Notice that 
				\begin{align*}
					[\gamma_2(y_1(\gamma_1+\gamma_2)-y_2\gamma_2)(x_2)_m]
					&=[(y_1(\gamma_1\gamma_2+\gamma_2^2)-y_2\gamma_2^2)(x_2)_m] \\
					&=[-y_1\gamma_1^2-y_2(\gamma_2^2))(x_2)_m] \\
					&=-(\ref{6})+(\ref{1})
				\end{align*}
				hence the generator $(\ref{6})+(\ref{1})$ is algebraically redundant.
				Assuming all torsion on the $E^3$-page of the spectral sequence remains in the cohomology algebra, the algebra will contain the generator
				\begin{equation*}
					[y_2(\gamma_1^2-\gamma_1\gamma_2)(x_2)_m]
				\end{equation*}
				and $I$ contains the generator
				\begin{equation*}
					3(y_1\gamma_1^2+y_2\gamma_2^2)(x_2)_m.
				\end{equation*}
				
				Next we consider elements of the $E^2$-page of the form
				\begin{equation*}
					[(x_2)_a(x_4)_by_iP].
				\end{equation*}
				We have already considered the case when $a\geq1$ and $\deg(P)\leq 2$
				by studying the quotient of the kernel of $d^2$ on elements of the form $[(x_2)_a(x_4)_by_1y_2P]$.
				When $a=0$ or $\deg(p)=3$ all elements of the form $[(x_2)_a(x_4)_by_iP]$ are in the kernel of $d^2$.
				It remains to deduce the quotient of such generators by the image of $d^2$ and the kernel of $d^2$ whose codomain lies in the span of such elements.
				
				When the degree of $P$ is $0$, $[(x_4)_by_i]$ is not in the image of $d^2$, so services to the third page.
				However $(x_4)_by_i$ is already a product of generators $(x_4)_m$ and $y_i$.
				When the degree of $P$ is $1$, the image of $d^2$ on $x_2(x_4)_b$ is given by (\ref{eq:diff2y}).
				Since the image is spanned by just the one generator $[(x_4)_b(y_2(2\gamma_1+\gamma_2)-y_1(\gamma_1+2\gamma_2))]$,
				the kernel is trivial and the quotient is generated by $[y_2\gamma_1(x_4)_b],[y_2\gamma_2(x_4)_b]$ and $[y_1\gamma_2(x_4)_b]$
				all of which are products of $(x_4)_m,y_i$ and $\gamma_i$. 
				In addition
				\begin{equation*}
					y_2(2\gamma_1+\gamma_2)-y_1(\gamma_1+2\gamma_2)
				\end{equation*}
				is a generator of $I$.
				Since the image of $d^2$ and the symmetric ideal are in $I$ and $(x_4)_m,y_i,\gamma_i$ are generators of the algebra
				any generator of $I$ not containing an $(x_2)_m$ term is redundant. 
				When the degree of $P$ is $2$, by (\ref{eq:diff2y}) the image of $d^2$ with codomain in $[(x_4)_by_iP]$ is generated by
				\begin{align*}
					d^2([x_2(x_4)_b\gamma_1])
					&=[(y_2(2\gamma_1^2+\gamma_1\gamma_2)-y_1(\gamma_1^2+2\gamma_1\gamma_2))(x_4)_b] \\
					&=[(y_2(2\gamma_1^2+\gamma_1\gamma_2)-y_1(\gamma_1\gamma_2-\gamma_2^2))(x_4)_b], \\
					d^2([x_2(x_4)_b\gamma_2])
					&=[(y_2(2\gamma_1\gamma_2+\gamma_2^2)-y_1(\gamma_1\gamma_2+2\gamma_2^2))(x_4)_b] \\
					&=[(y_2(\gamma_1\gamma_2-\gamma_1^2)-y_1(\gamma_1\gamma_2+2\gamma_2^2))(x_4)_b].
				\end{align*}
				Subtracting the second generator from the first gives 
				\begin{equation*}
					[3(y_1\gamma_1^2+y_2\gamma_2^2)(x_4)_b].
				\end{equation*}
				Hence the generators of the image are independent and the kernel of $d^2$ with codomain in $\langle [(x_4)_by_iP] \rangle$ is trivial.
				In addition the quotient by the image is isomorphic to $\mathbb{Z}^2\oplus \mathbb{Z}_3$ as a group
				and assuming all $3$-torsion survives the cohomology algebra already contained all necessary generators and relations.
				When the degree of $P$ is $3$, by Theorem \ref{thm:LastColumn} the quotient by the image of $d^2$ is isomorphic to $\mathbb{Z}_3$,
				generated by any $[(x_2)_a(x_4)_by_i\gamma_1^2\gamma_2]$ or $-[(x_2)_a(x_4)_by_i\gamma_1\gamma_2^2]$.
				However
				\begin{equation*}
					[(x_2)_a(x_4)_b\gamma_2(y_1\gamma_1^2+y_2\gamma_2^2)]=[(x_2)_a(x_4)_b(y_1\gamma_1^2\gamma_2+y_2\gamma_2^3]=[(x_2)_a(x_4)_by_i\gamma_1^2\gamma_2].
				\end{equation*}
				So
				$[(x_2)_a(x_4)_by_i\gamma_1^2\gamma_2]$ is contained on the $E_\infty$-page.
				Since $\triset{2}{2}=2=2\triset{2}{3}$ the kernel of $d^2$ with codomain in $\langle [(x_2)_a(x_4)_by_iP] \rangle$ is trivial.

				All necessary generators and relations are already contained in the algebra.
				
				Finally elements of the form $[(x_2)_m(x_4)_bP]$ in the $E_3$-page are all trivial,
				since the kernel of $d^2$ on elements of the form $[(x_2)_a(x_4)_by_1yP]$ was always trivial.
				Elements of the form $[(x_4)_bP]$ survive to the third page and are already included 
				on the $E_\infty$-page
				as a product of generators $(x_4)_m$ and $\gamma_i$.
				
				All the torsion on the $E_\infty$-page of the spectral sequence is $3$ torsion.
				In order to resolve any extension problems that arise, we will consider the spectral sequence $\{ E_r, d^r \}$ over the field of order $3$.
				
				None of the generators in the integral spectral sequence are divisible by $3$,
				hence in the modulo $3$ spectral sequence all of the integral generators remain non-trivial. 
				In addition when the kernel of $d^2$ at $E_2^{p,q}$ is all of $E_2^{p,q}$, the rank plus rank of the torsion in the integral spectral sequence
				must be greater than or equal to the rank in the modulo $3$ spectral sequence.
				So in these cases the rank in modulo $3$ spectral sequence is exactly the rank plus the rank of the torsion in the integral case.
				Hence it remains to consider the kernels of the $d^2$ differential in the cases when the integral kernel is not the entire domain.
				By the rank nullity theorem, the rank of the image plus the nullity, the dimension of the kernel is the dimension of the domain.
				
				When considering the spectral sequence modulo $3$ the rank of any differential is the same as in the integral case
				when the quotient of the preceding kernel by the image contains no torsion.
				When integral $3$-torsion exists, there is are generators of the image which are $3$ times a generators of the kernel.
				In the modulo $3$ spectral sequence these generators of the image are now generators of the kernel.
				Hence in the modulo $3$ sepulchral sequence the rank is reduced by the dimension of the integral torsion and the nullity increased by the same number.
				
				Since the modulo $3$ spectral sequence has coefficients in a field, there are no extension problems.
				As the the total degree of the $d^2$ differential is $-1$ and $E_3=E_\infty$, $\dim(H^i(SU(3)/T^2;\mathbb{Z}_3))$
				is the sum of the ranks of the total degree $i$ coordinated of the integral $E_3$-page plus the sum of the torsion rank in total degrees $i$ and $i+1$.
				By Corollary \ref{cor:UniversalModp}, 
				the modulo $3$ cohomology algebra is only consistent with the case when all torsion on the $E_\infty$-page of the spectral sequence
				is contained in the integral cohomology module
				Therefore all additive extension problems are resolved and all the torsion elements in the spectral sequence are present in the integral cohomology.
			\end{proof}

\newpage
\section{Cohomology of the free loop space of the complete flag manifold of $Sp(n)$}\label{sec:FreeLoopSp(n)/Tn}
	
	In this chapter we apply the method used in Chapter \ref{sec:FreeLoopSU(n+1)/Tn} to study the free loop cohomology of $SU(n+1)/T^n$
	and apply them to study the free loops cohomology of $Sp(n)/T^n$.
	The Lie groups $Sp(n)$ is simply connected, hence $Sp(n)/T^n$ is too.
	In addition the integral cohomology of $Sp(n)$ like that of $SU(n)$ has no torsion,
	so the process of adapting the methods is relatively straightforward.
	However these properties are not shared by the other simple Lie groups,
	meaning that generalising the arguments of Chapter \ref{sec:FreeLoopSU(n+1)/Tn} to their cases would require more work.
	
	\subsection{Differentials in the path space spectral sequence}	
		
		Just as in Section \ref{sec:FreeLoopSU(n+1)/Tn}
		we begin by studying the cohomology Leray-Serre spectral sequence associated to the fibration
		\begin{equation}\label{eq:SpMapFib}
			\Omega (Sp(n)/T^n) \to Map(I,Sp(n)/T^n) \xrightarrow{eval} Sp(n)/T^n\times Sp(n)/T^n,
		\end{equation}
		where $eval \colon Map(I,Sp(n)/T^n)\to Sp(n)/T^n\times Sp(n)/T^n$ is given by $\alpha\mapsto (\alpha(0),\alpha(1))$ and $Map(I,Sp(n)/T^n)\simeq Sp(n)/T^n$.
		By the same reasoning as for $\Omega(SU(n+1)/T^n)$,
		\begin{equation*}
			\Omega(Sp(n)/T^n)\simeq \Omega Sp(n) \times T^n.
		\end{equation*}
		Using the K\"{u}nneth formula and Theorem \ref{thm:H*(Sp(n))}, we obtain the algebra isomorphism
		\begin{equation*}
			H^*(\Omega(Sp(n)/T^n);\mathbb{Z})
			\cong \Gamma_{\mathbb{Z}}[x_2,x_6,\dots,x_{4n-2}] \otimes \Lambda_{\mathbb{Z}}(y_1,\dots,y_n),
		\end{equation*}
		where $\Gamma_{\mathbb{Z}}[x_2,x_4,\dots,x_{4n-2}]$ is the integral divided polynomial algebra on variables $x_2,x_6,\dots,x_{4n-2}$
		with $|x_i|=i$ for each $i=2,6,\dots,4n-2$
		and $\Lambda(y_1,\dots,y_n)$ is an exterior algebra generated by $y_1,\dots,y_n$
		with $|y_j|=1$ for each $j=1,\dots,n$.
		The cohomology of $Sp(n)/T^n$ is given in Theorem \ref{thm:H*Sp/T}, as
		\begin{equation*}
			H^*(Sp(n)/T^n;\mathbb{Z})=\frac{\mathbb{Z}[\gamma_1,\dots,\gamma_n]}{[\sigma_1^{\lambda^2},\dots,\sigma_{n}^{\lambda^2}]},
			\end{equation*}
		where $|\gamma_i|=2$ and $\sigma_1^{\lambda^2},\dots,\sigma_n^{\lambda^2}$ are the elementary symmetric polynomials in $\gamma_1^{2},\dots,\gamma_n^2$.
		In this section we use the notation
		\begin{equation*}
			H^*(Map(I,Sp(n)/T^n);\mathbb{Z}) = \frac{\mathbb{Z}[\lambda_1,\dots,\lambda_n]}{[{\sigma^{\lambda^2}_1},\dots,{\sigma^{\lambda^2}_n}]}
		\end{equation*}
		and
		\begin{equation*}
			H^*(Sp(n)/T^n \times Sp(n)/T^n;\mathbb{Z}) = 
			\frac{\mathbb{Z}[\alpha_1,\dots,\alpha_n]}{[{\sigma^{\alpha^2}_1},\dots,{\sigma^{\alpha^2}_{n+1}}]} \otimes
			\frac{\mathbb{Z}[\beta_1,\dots,\beta_n]}{[{\sigma^{\beta^2}_1},\dots,{\sigma^{\beta^2}_{n+1}}]}
		\end{equation*}
		for the cohomology of the base space and fiber of fibration (\ref{eq:SpMapFib}).
		Where $|\lambda_1|=|\alpha_i|=|\beta_i|=2$ for $1\leq i \leq n$ and ${\sigma^{\lambda^2}_i}$, ${\sigma^{\alpha^2}_i}$ and ${\sigma^{\beta^2}_i}$
		are the complete homogeneous symmetric polynomials in variables
		$\lambda_1^2,\dots,\lambda_n^2$, $\alpha_1^2,\dots,\alpha_n^2$ and $\beta_1^2,\dots,\beta_n^2$ respectively.
		Denote by $\{ E^r,d^r \}$ the cohomology Leray-Serre spectral sequence associated to fibration (\ref{eq:SpMapFib}).
		We again use the altenative basis
		\begin{equation*}
			v_i=\alpha_i-\beta_i\;\; \text{and} \;\; u_i=\beta_i
		\end{equation*}
		for $1\leq i\leq n$.
		For exactly the same reasons as Lemma \ref{lem:E^2_{*,1}d^2}, we get an equivalent lemma in case of $\{ E_r,d^r \}$
		
		\begin{lem}\label{lem:SpE^2_{*,1}d^2}
			With the notation above, in the cohomology Leray-Serre spectral sequence of fibration (\ref{eq:SpMapFib}),
			there is a choice of basis $y_1,\dots,y_n$ such that
			\begin{center}
				$d^2(y_i)=v_i$
			\end{center}
			for each $i=1,\dots,n$.
		\end{lem}
		
		\begin{rmk}\label{rmk:Spunique}
			Similarly to Remark \ref{rmk:unique}, 
			the image of each of the differentials $d^{4i-2}$ for $1\leq i \leq n$ will be a unique class in $E_{4i-2}^{4i-2,1}$
			in the kernel of $d^2$ not already contained in the image of any $d^r$ for $r<4i-2$.
		\end{rmk}
		
		Let $S$ be the subalgebra of
		$\frac{\Lambda(y_1,\dots,y_n)\otimes\mathbb{Z}[\alpha_1,\dots,\alpha_n,\beta_1,\dots,\beta_n]}
		{[\sigma^{\alpha^2}_1,\dots,\sigma^{\alpha^2}_n,\sigma^{\beta^2}_1,\dots,\sigma^{\beta^2}_n]}$
		generated by elements of the form
		\begin{align*}
										&g_{u,l,t,s}=\sum_{\substack{1\leq i_1<\cdots<i_{t-1}\leq n \\ 1\leq i_{t}<\cdots<i_{s-1}\leq n 
																			 \\ 1\leq i_s<\cdots<i_{l-1}\leq n \\ 1\leq k \leq n, \; k\neq i_j \neq i_{j'} }}
										{y_ku_ku_{i_1}v_{i_1}\cdots u_{i_{t-1}}v_{i_{t-1}}u^2_{i_{t}}\cdots u^2_{i_{s-1}}v_{i_{s}}^2\cdots v^2_{l-1}}, \\
			\text{or}\;\;	&g_{v,l,t,s}=\sum_{\substack{1\leq i_1<\cdots<i_{t-1}\leq n \\ 1\leq i_{t}<\cdots<i_{s-1}\leq n 
																			\\ 1\leq i_s<\cdots<i_{l-1}\leq n \\ 1\leq k \leq n, \; k\neq i_j \neq i_{j'} }}
										{y_kv_ku_{i_1}v_{i_1}\cdots u_{i_{t-1}}v_{i_{t-1}}u^2_{i_{t}}\cdots u^2_{i_{s-1}}v_{i_{s}}^2\cdots v^2_{l-1}}
		\end{align*}
		for any $1\leq t<s< l$.
		Define an operations $\psi_{u^2}$, $\psi_{uv}$ and $\psi_{v^2}$ on $S$ by
		\begin{align*}
			&\psi_{u^2}(g_{u,l,t,s})=g_{u,l+1,t,s+1}, 	\; &\psi_{u^2}(g_{v,l,t,s})=g_{u,l+1,t,s+1},\;\;\; \\
			&\psi_{v^2}(g_{u,l,t,s})=g_{u,l+1,t,s}, 		\; &\psi_{v^2}(g_{v,l,t,s})=g_{u,l+1,t,s}, \;\;\;\;\;\:\:\\
			&\psi_{uv} (g_{u,l,t,s})=g_{u,l+1,t+1,s+1}, \; &\psi_{u^2}(g_{v,l,t,s})=g_{u,l+1,t+1,s+1}.
		\end{align*}
		We now prove an equivalent of Theorem \ref{thm:finaldiff}, for $Sp(n)/T^n$.
		
		\begin{thm}\label{thm:SpPathDiff}
			For each $n\geq 1$ and $1\leq l \leq n$ in the spectral sequence $\{ E_r,d^r \}$ up to class representative on $E^2_{4l-2,1}$, we have
			\begin{equation*}
				d^{4l-2}(x_{4l-2})=A+2\sum_{\substack{1\leq i_1<\cdots<i_{l-1}\leq n \\ 1\leq k \leq n, \; k\neq i_j }}{y_{k}u_{k}u_{i_1}^2\cdots u_{i_{l-1}}^2},
			\end{equation*}
			where $A$ is an element of $S$ for which each summand is divisible by $v_i$ for some $1\leq i \leq n$
			and
			\begin{equation*}
				d^2(A+2\sum_{\substack{1\leq i_1<\cdots<i_{l-1}\leq n \\ 1\leq k \leq n, \; k\neq i_j }}{y_{k}u_{k}u_{i_1}^2\cdots u_{i_{l-1}}^2})
				={\sigma_l^\alpha}^2-{\sigma_l^\beta}^2.
			\end{equation*}
		\end{thm}
		
		\begin{proof}
			We proceed by induction on $l$.
			When $l=1$, by Lemma \ref{lem:SpE^2_{*,1}d^2} 
			\begin{align*}
				d^2\big(\sum_{1\leq k\leq n}{y_kv_k+2y_ku_k}\big)&=\sum_{1\leq k\leq n}{v_k^2+2v_ku_k} \\
				&=\sum_{1\leq k \leq n}{(\alpha_k-\beta_k)^2+2(\alpha_k-\beta_k)\beta_k} \\
				&=\sum_{1\leq k \leq n}{(\alpha_k^2-2\alpha_k\beta_k+\beta_k^2)+2(\alpha_k\beta_k-\beta_k^2)} \\
				&=\sum_{1\leq k \leq n}{\alpha_k^2-\beta_k^2}.
			\end{align*}
			As $\sigma_1^{\alpha^2}=\sum_{1\leq k \leq n}{\alpha_k^2}$ and $\sigma_1^{\beta^2}=\sum_{1\leq k \leq n}{\beta_k^2}$,
			$\sum_{1\leq k\leq n}{y_kv_k+2y_ku_k}$ represents the image of $d^2$ by Remark \ref{rmk:Spunique}.
			
			Now assume the statement of the theorem is true for all $d^{4i-2}$ for $i< l$.
			Hence the by inductive hypothesis, there is an $A$ represented by an element of $E_2^{4l-6,1}$ for which each summand is divisible by $v_i$ for some $1\leq i \leq n$
			such that
			\begin{equation*}
				d^{4l-6}(x_{4l-6})=A+2\sum_{\substack{1\leq i_1<\cdots<i_{l-2}\leq n \\ 1\leq k \leq n, \; k\neq i_j }}{y_{k}u_{k}u_{i_1}^2\cdots u_{i_{l-2}}^2}
			\end{equation*}
			and
			\begin{equation*}
				d^2\big(A+2\sum_{\substack{1\leq i_1<\cdots<i_{l-2}\leq n \\ 1\leq k \leq n, \; k\neq i_j }}{y_{k}u_{k}u_{i_1}^2\cdots u_{i_{l-2}}^2}\big)
				=\sigma_{l-1}^{\alpha^2}-\sigma_{l-1}^{\beta^2}
				=\sum_{1\leq i_1 <\cdots<i_{l-1}\leq n}{\alpha_{i_1}^2\cdots\alpha_{i_{l-1}}^2-\beta_{i_1}^2\cdots\beta_{i_{l-1}}^2}.
			\end{equation*}
			Notice that
			\begin{align}
				d^2\psi_{v^2}&(A+2\sum_{\substack{1\leq i_1<\cdots<i_{l-1}\leq n \\ 1\leq k \leq n, \; k\neq i_j }}{y_{k}u_{k}u_{i_1}^2\cdots u_{i_{l-1}}^2}\big) \nonumber \\ 
				&=\sum_{\substack{1\leq i_1<\cdots<i_{l}\leq n \\ 1\leq k \leq l}}
				{(\alpha_{i_1}^2\cdots\hat{\alpha}_{i_k}^2\cdots\alpha^2_{i_{l}}-\beta_{i_1}^2\cdots\hat{\beta}_{i_k}^2\cdots\beta^2_{i_{l}})
				(\alpha^2_{i_k}-2\alpha_{i_k}\beta_{i_k}+\beta^2_{i_k})} \label{eq:Psiv2}, \\
				d^2\psi_{uv}&(A+2\sum_{\substack{1\leq i_1<\cdots<i_{l-1}\leq n \\ 1\leq k \leq n, \; k\neq i_j }}{y_{k}u_{k}u_{i_1}^2\cdots u_{i_{l-2}}^2}\big) \nonumber \\ 
				&=\sum_{\substack{1\leq i_1<\cdots<i_{l}\leq n \\ 1\leq k \leq l}}
				{(\alpha_{i_1}^2\cdots\hat{\alpha}_{i_k}^2\cdots\alpha^2_{i_{l}}-\beta_{i_1}^2\cdots\hat{\beta}_{i_k}^2\cdots\beta^2_{i_{l}})
				(\alpha_{i_k}\beta_{i_k}-\beta^2_{i_k})} \label{eq:Psiuv}, \\
				d^2\psi_{u^2}&(A+2\sum_{\substack{1\leq i_1<\cdots<i_{l-1}\leq n \\ 1\leq k \leq n, \; k\neq i_j }}{y_{k}u_{k}u^2_{i_1}\cdots u_{i_{l-2}}^2}\big) \nonumber \\ 
				&=\sum_{\substack{1\leq i_1<\cdots<i_{l}\leq n \\ 1\leq k \leq l}} 
				{(\alpha_{i_1}^2\cdots\hat{\alpha}_{i_k}^2\cdots\alpha^2_{i_{l}}-\beta_{i_1}^2\cdots\hat{\beta}_{i_k}^2\cdots\beta^2_{i_{l}})\beta^2_{i_k}} \label{eq:Psiu2}, \\
				d^2&\big(\sum_{\substack{1\leq i_1<\cdots<i_{l-1}\leq n \\ 1\leq k \leq n, \; k\neq i_j }}{y_kv_ku_{i_1}^2\cdots u^2_{i_{l-1}}}\big) 
				=\sum_{\substack{1\leq i_1<\cdots<i_{l}\leq n \\ 1\leq k \leq l}}{\beta_{i_1}^2\cdots\hat{\beta}_{i_k}^2\cdots\beta^2_{i_{l}}
				(\alpha^2_{i_k}-2\alpha_{i_k}\beta_{i_k}-\beta_{i_k}^2)}
				\label{eq:yv}, \\
				d^2&\big(\sum_{\substack{1\leq i_1<\cdots<i_{l-1}\leq n \\ 1\leq k \leq n, \; k\neq i_j }}{y_ku_ku_{i_1}^2\cdots u^2_{i_{l-1}}}\big) 
				=\sum_{\substack{1\leq i_1<\cdots<i_{l}\leq n \\ 1\leq k \leq l}}{\beta_{i_1}^2\cdots\hat{\beta}_{i_k}^2\cdots\beta^2_{i_{l}}
				(\alpha_{i_k}\beta_{i_k}-\beta_{i_k}^2)}
				\label{eq:yu}
				.
			\end{align}
			Therefore
			\begin{align}
				(\ref{eq:Psiv2})&+2(\ref{eq:Psiuv})+(\ref{eq:Psiu2})+(\ref{eq:yv})+(\ref{eq:yu}) \label{eq:FinalSum}\\
				&=\sum_{\substack{1\leq i_1<\cdots<i_{l}\leq n \\ 1\leq k \leq l}}
				{\alpha_{i_1}^2\cdots\hat{\alpha}_{i_k}^2\cdots\alpha^2_{i_{l-1}}\alpha_{i_k}^2-\beta_{i_1}^2\cdots\hat{\beta}_{i_k}^2\cdots\beta^2_{i_{l-1}}\beta_{i_k}^2} 	
				\nonumber\\
				&=\sum_{1\leq i_1<\cdots<i_l\leq n}
				{\alpha_{i_1}^2\cdots\alpha^2_{i_l}-\beta_{i_1}^2\cdots \beta^2_{i_l}} \nonumber\\
				&=\sigma_{l}^{\alpha^2}-\sigma_{l}^{\beta^2}. \nonumber
			\end{align}
			Since (\ref{eq:Psiu2}) and (\ref{eq:yu}) are the only terms obtained as the image under $d^2$ using
			\begin{equation*}
				\sum_{\substack{1\leq i_1<\cdots<i_{l-1}\leq n \\ 1\leq k \leq n, \; k\neq i_j }}{y_ku_ku_{i_1}^2\cdots u^2_{i_{l-1}}}
			\end{equation*}
			and the expression (\ref{eq:FinalSum}) is obtained as the image under $d^2$ of an element in $S$.
			So (\ref{eq:FinalSum}) is obtained as the image under $d^2$ of an expression having the required form.
		\end{proof}
		
		For dimensional reason for each $r \geq 2$ and $1\leq i\leq n$,
		\begin{equation*}
			d^r(u_i)=0=d^r(v_i).
		\end{equation*}
		Therefore all the $d^r$ is determined on all generators of the $E_2$-page, so the differential is determined everywhere in $\{ E_r, d^r \}$.
		
	\subsection{Differentials in the free loop spectral sequence}
		
		Just as we did in Section \ref{sec:diff},
		we can now use the results of Theorem \ref{thm:SpPathDiff} to deduce the differentials in the 
		cohomology Leray-Serre spectral sequence associated to the free loop fibration of $Sp(n)/T^n$.
		Similarly to Section \ref{sec:diff}, consider the map $\phi$ of fibrations
		
		\begin{equation*}\label{fig:fibcd}
						\xymatrix{
							{\Omega(Sp(n)/T^n)} \ar[r]^(.5){} \ar[d]^(.45){id} & {\Lambda(Sp(n)/T^n)} \ar[r]^{eval} \ar[d]^(.45){exp}  & {Sp(n)/T^n} \ar[d]^(.45){\Delta} \\
							{\Omega(Sp(n)/T^n)} \ar[r]^(.45){}   							 & {Map(I,Sp(n)/T^n)} 	 \ar[r]^(.425){eval}  					 & {Sp(n)/T^n\times Sp(n+1)/T^n} }
		\end{equation*}

		between the free loop space fibration on $Sp(n)/T^n$ and the path space fibration on $Sp(n)/T^n$,
		where $\exp$ is given on elements by $\exp(\alpha)(t)=\alpha(e^{2\pi i t})$.
		Since $Sp(n)/T^n$ like $SU(n+1)/T^n$ is simply connected, the free loop fibration induces a cohomology Leray-Serre spectral sequence $\{\bar{E}_r, \bar{d}^r \}$.
		Hence $\phi$ indices a map of spectral sequences $\phi^*\colon \{ E_r,d^r \} \to \{\bar{E}_r, \bar{d}^r \}$.
		For the rest of the section we denote the cohomology algebras of the base space and fiber of the
		free loop fibration $\Omega(Sp(n)/T^n)\to \Lambda(Sp(n)/T^n \to Sp(n)/T^n$ by
		\begin{equation}\label{eq:spLoppCoHomology}
			H^*(\Omega(Sp(n)/T^n);\mathbb{Z})=\Gamma_{\mathbb{Z}}(x'_2,x'_6,\dots,x'_{4n-2})\otimes \Lambda_{\mathbb{Z}}(y'_1,\dots,y'_n)
		\end{equation}
		and
		\begin{equation*}
			H^*(Sp(n)/T^n;\mathbb{Z})=\frac{\mathbb{Z}[\gamma_1,\dots,\gamma_n]}{[\sigma_1^2,\dots,\sigma_{n}^2]}
		\end{equation*}
		where $|{y'}_i|=1$, $|\gamma_i|=2$, $|x'_{4i-2}|=4i-2$ for each $1\leq i\leq n$ and $\sigma_1^2,\dots,\sigma_n^2$ 
		are the elementary symmetric polynomials in variables $\gamma_1^2,\dots,\gamma_n^2$.
		
		\begin{thm}\label{eq:SpDiff}
			For each $n\geq 1$ and $1\leq l \leq n$, the only non-zero differentials on generators of the $\bar{E}_2$-page of $\{ \bar{E}_r,\bar{d}^r \}$
			are up to class representative and sign,
			\begin{equation*}
				\bar{d}^2(x_{4l-2})=2\sum_{1\leq i_1<\cdots<i_l\leq n}{y_{i_1}\gamma_{i_1}\gamma_{i_2}^2\cdots\gamma_{i_l}^2}
			\end{equation*}
		
		\end{thm}
		
		\begin{proof}
			For the same reasons as in the proof of Theorem \ref{thm:allDiff}, we have
			\begin{equation*}
				\phi^*(y_i)=y'_i,\;\;\phi^*(x_i)=x'_i\;\;\text{and}\;\;\phi^*(\alpha_i)=\gamma_i=\phi^*(\beta_i)=\phi^*(u_i),\;\;\text{so}\;\;\phi^*(v_i)=0. 
			\end{equation*}
			Hence by exactly the same arguments used in the proof of Theorem \ref{thm:allDiff}, we have 
			\begin{equation*}
				\bar{d}^2(y'_i)=0
			\end{equation*}
			and the image of $\bar{d^r}$ on generators $x'_2,x'_6,\dots,x'_{4n-2}$
			is determined by those summands in the image of $d^2$ on $x_2,x_6,\dots,x_{4n-2}$ containing no $v_i$, replacing $u_i$ with $\gamma_i$ and $y_i$ with $y'_i$.
			This gives us the result stated in the theorem.
		\end{proof}
	
	\subsection{Free loop cohomology of $Sp(2)/T^2$}
		
		The group $Sp(1)=SU(2)$, hence the first new case for $H^*(\Lambda(Sp(n)/T^n);\mathbb{Z})$ is when $n=2$.
		So in this section we study, the cohomology algebra of the free loop space of $Sp(2)/T^2$.
		
		\begin{thm}\label{thm:FreeLoopSp(2)/T}
				The integral algebra structure of the $E_\infty$-page of the Leray-Serre spectral sequence associated to the free loop space fibration of $\Lambda(Sp(2)/T^2)$
				is $A/I$, where
			\begin{align*}
				A=\Lambda_{\mathbb{Z}}&((x_6)_b\gamma_i,\;y_1y_2(x_2)_a(x_6)_b,\;(x_6)_by_i,\;(x_2)_m(x_6)_b(y_1\gamma_2+y_2\gamma_1),\;(x_2)_m(x_6)_by_2\gamma_1^2\gamma_2, \\
					&\;(x_2)_m(x_6)_b(y_1\gamma_1-y_2\gamma_2),\;(x_2)_a(x_6)_b\gamma_1^3\gamma_2)
			\end{align*}
			and
			\begin{equation*}
				I=[(x_2)_1^m-m!(x_2)_m,\;(x_6)_1^m-m!(x_6)_m,\;\gamma_1^2+\gamma_2^2,\;\gamma_1^2\gamma_2^2,\;2(x_2)_a(y_1\gamma_1+y_2\gamma_2),\;4y_1(x_2)_a\gamma_1^3)]
			\end{equation*}
			for $i=1,2$, $m\geq 1$, $a,b\geq 1$ and where $|(x_2)_m|=2m$, $|(x_6)_m|=6m$, $|y_i|=1$ and $|\gamma_i|=2$.
			Furthermore all additive extension problems with the exception of differentiating between $2$ and $4$-torsion, are trivial.
			Hence the algebra is the same module structure as $H^*(\Lambda(Sp(2)/T^{2});\mathbb{Z})$ up to the value of $j$.
		\end{thm}
		
		\begin{proof}
			We consider the cohomology Leray-Serre spectral sequence $\{ E_r,d^r \}$ associated to the free loop fibration of $Sp(2)/T^2$,
			\begin{equation*}
				\Omega(Sp(2)/T^2)\to\Lambda(Sp(2)/T^2)\to Sp(2)/T^2.
			\end{equation*}
			By Theorem \ref{thm:H*Sp/T},
			the cohomology of the base space $Sp(2)/T^2$ is
			\begin{equation*}
				H^*(Sp(2)/T^2;\mathbb{Z})=\frac{\mathbb{Z}[\gamma_1,\gamma_2]}{[\gamma_1^2+\gamma_2^2,\;\gamma_1^2\gamma_2^2]}.
			\end{equation*}
			From (\ref{eq:spLoppCoHomology}), the cohomology of the fiber $\Omega(Sp(2)/T^2)$ is
			\begin{equation*}
				H^*(\Omega(Sp(2)/T^2);\mathbb{Z})=\Lambda_{\mathbb{Z}}(y_1,y_2)\otimes \Gamma_{\mathbb{Z}}[x_2,x_6],
			\end{equation*}
			where $|y_1|=1=|y_2|$, $|x_2|=2$, $|x_6|=6$, $\Lambda_{\mathbb{Z}}(y_1,y_2)$ is an exterior algebra
			and $\Gamma_{\mathbb{Z}}[x_2,x_6]$ is a divide polynomial algebra.
			That is
			\begin{equation*}
				\Gamma[x_2,x_6]_{\mathbb{Z}}=\frac{\mathbb{Z}[(x_2)_1,(x_2)_2,\dots,(x_6)_1,(x_6)_2,\dots]}{[(x_2)_1^m-m!(x_2)_m, (x_6)_1^m-m!(x_6)_m]},
			\end{equation*}
			where $(x_1)_2=x_2$ and $(x_6)_1=x_6$.
			
			The elements on the $E_2$-page of the spectral sequence are generated additively by representative elements
		  of the form
			\begin{equation*}
				(x_2)_a(x_6)_bP, \;\; (x_2)_a(x_6)_by_iP, \;\; (x_2)_a(x_6)_by_1y_2P
			\end{equation*}
			where $0\leq a,b$, $1\leq i \leq n$ and $P\in \mathbb{Z}[\gamma_1,\gamma_2]$ is a monomial.
			The generators of the ideal in $H^*(\Omega(Sp(2)/T^2);\mathbb{Z})$ 
			are the squares of the elementary symmetric polynomials.
			We may replace the generator $\gamma_1^2\gamma_2^2$ of the ideal with $\gamma_1^4$,
			by adding $\gamma_1^2(\gamma_1^2+\gamma_2^2)$ to the negative of this generator. 		
			Hence the monomials 
			\begin{equation}\label{eq:Sp(2)Basis}
				\gamma_1^i\gamma_2^j
			\end{equation}
			form an additive basis of $H^*(\Omega(Sp(2)/T^2);\mathbb{Z})$, for $0\leq i \leq 3$ and $0\leq j \leq 1$.
			Therefore $P$ has degree between $0$ and $4$.
			
			By Theorem, \ref{eq:SpDiff} the only non-zero differentials in $\{ E_r,d^r \}$ are $d^2$ and $d^6$,
			which are non-zero only on generators $x_2$ and $x_6$ respectively.
			Hence the spectral sequence converges on the seventh page. 
			The differentials up to sign are given by
			\begin{equation}\label{eq:Sp(2)diff}
				d^2([x_2])=2[y_1\gamma_1+y_2\gamma_2], \;\;\;
				d^4([x_6])=2[y_1\gamma_1\gamma_2^2+y_2\gamma_2\gamma_1^2].
			\end{equation}
			Since these are representatives over the symmetric ideal,
			\begin{align}\label{eq:x6Non-trival}
				d^4([x_6])+d^2([x_2])\gamma_1^2 \nonumber
				&=2[y_1\gamma_1\gamma_2^2+y_2\gamma_1^2\gamma_2]+2[y_1\gamma_1^3+y_2\gamma_1^2\gamma_2] \\
				&=2[y_1\gamma_1^3-y_2\gamma_1^2\gamma_2]+2[y_1\gamma_1^3+y_2\gamma_1^2\gamma_2] \nonumber \\
				&=4[y_1\gamma_1^3].
			\end{align} 
			Hence assuming that all extension problems are trivially resolved
			\begin{equation*} 
				2(x_2)_a(y_1\gamma_1+y_2\gamma_2) \;\; \text{and} \;\; 4(x_2)_ay_1\gamma_1^3
			\end{equation*}
			are included as relations on the $E_\infty$-page.
			The generators $\gamma_i$, and $y_i$ occur in $E_2^{*,0}$ and are always in the kernel of the differentials,
			so are free generators of the $E_\infty$-page. 
			The relation $x_2^m-m!(x_2)_m, x_6^m-m!(x_6)_m$ from the divide polynomial algebra
			and $\gamma_1^2+\gamma_2^2$, $\gamma_1^2,\gamma_2^2$ generators of the symmetric ideal
			remain as relations on the $E_\infty$-page. 
			
			We first consider the image and kernel of the differential $d^2$ on generators of the $E_2$-page of the form
			\begin{equation*}
				y_1y_2(x_2)_a(x_6)_bP.
			\end{equation*}
			Using (\ref{eq:Sp(2)diff}) the image of $d^2$ on such elements is generated by $[\gamma_i]$,
			\begin{align*}
							   -d^2([(x_2)_m(x_6)_by_1])&=2(x_2)_{m-1}(x_6)_b[y_1y_2\gamma_2]\\
				\text{and} \;\; d^2([(x_2)_m(x_6)_by_2])&=2(x_2)_{m-1}(x_6)_b[y_1y_2\gamma_1].
			\end{align*} 
			Hence on the $E_3$-page all element of the form $[y_1y_2(x_2)_a(x_6)_bP]$ are generated by
			$[y_1y_2(x_2)_a(x_6)_b]$ and $[\gamma_i]$ with all elements $2$-torsion except $[y_1y_2(x_2)_a(x_6)_b]$ itself, 
			which additively generates a copy of $\mathbb{Z}$.
			The kernel of $d^2$ on generators of the form $[y_1y_2(x_2)_a(x_6)_bP]$ is generated by $[\gamma_i]$ and
			\begin{equation}\label{eq:Spd^2KernRow1}
				[(x_2)_m(y_1\gamma_1+y_2\gamma_2)], \;\; [(x_2)_m(y_2\gamma_1-y_1\gamma_2)], \;\; [(x_2)_m(y_2\gamma_2\gamma_1^2)]
			\end{equation}
			The first is the kernel of the differential without considering the symmetric ideal, the second the kernel due to symmetric ideal generator
			$\gamma_1^2+\gamma_2^2$ and the third due to symmetric ideal generator $\gamma_1^2\gamma_2^2$.
			
			Next we consider the image and kernel of the differential $d^2$ on generators of the $E_2$-page of the form
			\begin{equation*}
				(x_2)_m(x_6)_by_iP \;\; \text{and} \;\; (x_6)_by_iP.
			\end{equation*}
			By (\ref{eq:Sp(2)diff})	the image of the $d^2$ differential on such generators is generated by $[\gamma_i]$ and
			\begin{equation*}
				d^2([(x_2)_m(x_6)_b])=2(x_2)_{m-1}(x_6)_b[y_1\gamma_1+y_2\gamma_2],
			\end{equation*}
			Which is exactly twice the first generator of the previous kernel in (\ref{eq:Spd^2KernRow1}).
			Hence The elements of the form $[y_i(x_2)_m(x_6)_bP]$ and $[y_i(x_6)_bP]$ on the $E_3$-page
			are either non-torsion or $2$-torsion.
			Multiplicatively such class are generated by $[\gamma_i]$,
			\begin{equation*}
				[(x_2)_m(x_6)_b(y_1\gamma_2+y_2\gamma_1)], \;\; [(x_2)_m(x_6)_b(y_1\gamma_1-y_2\gamma_2)]
			\end{equation*}
			and 
			\begin{equation*}
				[(x_6)_by_i]
			\end{equation*}			
			since $[(x_2)_m(x_6)_by_i]$ is not in the kernel of $d^2$.
			Assuming all extension problems are resolved trivially and these are not in the image of $d^6$ these generators will be generators of
			$H^*(\Lambda(Sp(n)/T^2);\mathbb{Z})$.
			Notice that the previous generator $[y_1y_2(x_2)_a(x_6)_b]$ is a product of generators $[y_i]$ and $[(x_6)_by_i]$
			when $a=0$, so is redundant in this case.
			The $d^2$ differential is twice the differential of the spectral sequence in Lemma \ref{lem:1to1spec}.
			Hence since in $\{ E_r,d^r \}$ we must also conditioner the symmetric ideal, any elements of the form $[(x_2)_a(x_6)_bP]$ in the kernel of $d^2$
			have image in ideal $J=[(x_2)_m,(x_6)_m,y_i,\gamma_1^2+\gamma_2^2,\gamma_1^2\gamma_2^2]$.
			When the degree of $P$ is $0$ the image of $d^2$ is $[2(x_2)_a(x_6)_b(y_1\gamma_1+y_2\gamma_2)]$
			which does not lie in ideal, since the monomials in gamma have only degree $1$.
			We will Express the remaining cases for the degree of $P$ in terms of the additive basis of (\ref{eq:Sp(2)Basis}).
			When the degree of $P$ is $1$ the image of $d^2$ is
			\begin{align*}
												&d^2([(x_2)_m(x_6)_b\gamma_1])=[2(x_2)_a(x_6)_b(y_1\gamma_1^2+y_2\gamma_1\gamma_2)] \\
				\text{and} \;\; &d^2([(x_2)_m(x_6)_b\gamma_2])=[2(x_2)_a(x_6)_b(y_1\gamma_1\gamma_2+y_2\gamma_2^2)]=[2(x_2)_a(x_6)_b(y_1\gamma_1\gamma_2-y_2\gamma_1^2)]
			\end{align*}
			which are linearly independent over $\mathbb{Z}$.
			When the degree of $P$ is $2$ the image of $d^2$ is
			\begin{align*}
												&d^2([(x_2)_m(x_6)_b\gamma_1^2])=[2(x_2)_a(x_6)_b(y_1\gamma_1^3+y_2\gamma_1^2\gamma_2)] \\
				\text{and} \;\; &d^2([(x_2)_m(x_6)_b\gamma_1\gamma_2])=[2(x_2)_a(x_6)_b(y_1\gamma_1^2\gamma_2+y_2\gamma_1\gamma_2^2)] \\
																															&=[2(x_2)_a(x_6)_b(y_1\gamma_1^2\gamma_2-y_2\gamma_1^3)]
			\end{align*}
			which are linearly independent over $\mathbb{Z}$.
			When the degree of $P$ is $3$ the image of $d^2$ is
			\begin{align}\label{eq:d2dim4}
												&d^2([(x_2)_m(x_6)_b\gamma_1^3])=[2(x_2)_a(x_6)_b(y_1\gamma_1^4+y_2\gamma_1^3\gamma_2)]=[(x_2)_a(x_6)_by_2\gamma_1^3\gamma_2] \nonumber\\
				\text{and} \;\; &d^2([(x_2)_m(x_6)_b\gamma_2\gamma_1^2])=[2(x_2)_a(x_6)_b(y_1\gamma_1^3\gamma_2+y_2\gamma_1^2\gamma_2^2)]
																																=[2(x_2)_a(x_6)_by_1\gamma_1^3\gamma_2] 
			\end{align}
			which are linearly independent over $\mathbb{Z}$.
			So the image of $d^2$ does not lie in $J$ till the degree of $P$ is $4$ and $d^2$ is trivial.
			Therefore on the $E_3$-page the only non-trivial element of the form $[(x_2)_a(x_6)_bP]$ is $[(x_2)_a(x_6)_b\gamma_1^3\gamma_2]$.
			Assuming these generators are not in the image of $d^6$, they will be generators of the $E_\infty$-page.
			
			When the image of $d^6$ lies in the span of $[(x_2)_a(x_6)_by_1y_2P]$ and $[\gamma_i]$, by (\ref{eq:Sp(2)diff}) the image $d^6$ is generated by
			\begin{align*}
											 - & d^6([(x_2)_a(x_6)_my_1])=2(x_2)_a(x_6)_{m-1}[y_1y_2\gamma_2\gamma_1^2] \\
				\text{and} \;\;  & d^6([(x_2)_a(x_6)_my_2])=2(x_2)_a(x_6)_{m-1}[y_1y_2\gamma_1\gamma_2^2]
			\end{align*}
			which is exactly the same as the image of $d^2$.
			Hence $d^6$ is always trivial in this case.
			The image of $d^6$ in lying in the span, of $[(x_2)_a(x_6)_by_1y_2P]$ and $[\gamma_i]$, is
			the image of generators $[(x_2)_a(x_6)_m]$ and $[(x_2)_a(x_6)_m\gamma_i]$.
			In the case when $a\geq 1$, these generators are trivial on the $E_6$-page.
			In (\ref{eq:x6Non-trival}) we have already shown that the image of $[(x_6)_m]$ is non-trivial.
			The image of $[(x_6)_m\gamma_i]$ is generated by $[\gamma_i]$,
			\begin{align*}
											   & d^6([(x_6)_m\gamma_1])=2(x_6)_{m-1}[y_1\gamma_1^2\gamma_2^2+y_2\gamma_2\gamma_1^3]=2[y_2\gamma_2\gamma_1^3] \\
				\text{and} \;\;  & d^6([(x_6)_m\gamma_2])=2(x_6)_{m-1}[y_1\gamma_1\gamma_2^3+y_2\gamma_2^2\gamma_1^2]=-2[_1\gamma_1^3\gamma_2^2]
			\end{align*}
			which by (\ref{eq:d2dim4}) is already in the image of $d^2$.
			Hence $d^6([(x_6)_m\gamma_i])$ is trivial.
			Therefore assuming all extension problems are resolved triviality,
			$(x_6)_m$ is not a generator of $A$ but may appear on the $E_\infty$-page
			as a multiple of any other generator.
			
			All torsion on the $E_\infty$-page of $\{ E_r,d^r \}$ is a power of $2$,
			hence we consider the spectral sequence $\{ E_r,d^r \}$ over the field of characteristic $2$.
			Since the only non-zero differentials $d^2$ and $d^6$ have bidegree $(2,-1)$ and $(6,-5)$ respectively, 
			for exactly the same reasons as for the modulo $3$ spectral sequence in Theorem \ref{thm:L(SU(3)/T^2)},
			all torsion on the $E_\infty$-page services the addative extension problems over $\mathbb{Z}$.
			The only remaining additive extension problem is weather the $4$-torsion generated by $[(x_2)_ay_1\gamma^3_1]$ on the $E_\infty$,
			is $2$-torsion or $4$-torsion in $H^*(\Lambda Sp(2)/T^2;\mathbb{Z})$.
		\end{proof}

\newpage
\section{Appendix}
		
		In Section \ref{sec:Third},
		we studied the third page of the Leray-Serre spectral sequence $\{ E_r,d^r \}$ associated to the free loop fibration
		\begin{equation*}
			\Omega(SU(n+1)/T^n)\to\Lambda(SU(n+1)/T^n)\to SU(n+1)/T^n.
		\end{equation*}
		A consequence of Theorem \ref{thm:SymFree} is that the elements in 
		\begin{equation*}
			E_3^{p,n-j+2m+\dim(X)}((x_2)_mX\hat{y}_{i_1,\dots,i_j})
		\end{equation*}
		for $0\leq j \leq n-1$, $m\geq 1$ and $X$ a monomial in $\Gamma_{\mathbb{Z}}(x_4,x_6,\dots,x_{2n})$,
		are trivial unless they are contained in the kernel of a $d^2$ differential with image divisible by a non-trivial element of the symmetric ideal. 
		For the remaining cases when $m=0$ or $j \neq n$, where all elements were in the kernel of the $d^2$ differential, there is a lot of structure left on the $E_3$-page.
		Theorems \ref{thm:BottemRow} and \ref{thm:LastColumn}
		solved the general problem of finding these quotients in the relatively simple cases but elsewhere the problem is more complicated. 
		
		In this appendix we demonstrate how to construct an algorithm to obtain the torsion on the $E_3$-page at 
		\begin{equation}\label{eq:Positions}
			E_3^{p,n-j+\dim(X)}(X\hat{y}_{i_1,\dots,i_j}) \;\; \text{and} \;\; E_3^{p,n+2m+\dim(X)}({(x_2)}_mXy_1\cdots y_n)
		\end{equation}
		for $1\leq j \leq n-1$, $m\geq 0$ and $X$ a monomial in $\Gamma_{\mathbb{Z}}(x_4,x_6,\dots,x_{2n})$.
		These are elements obtained from the $E_2$-page of the spectral sequence where the kernel of $d^2$ is trivial but the image is not.
		While this does not take into account any of elements in a non-trivial kernel of $d^2$, we demonstrate interesting patterns in the torsion 
		which currently cannot be supported by a theorem.
		
		To achieve this we first in Section \ref{sec:ImageMatrix}, construct an algorithm that given a coordinate corresponding to one in $(7.1)$,
		output a matrix whose rows correspond to the image of the $d^2$ differential and generators of symmetric ideal.
		The torsion at this coordinate can then be found by finding the integer Smith normal form of this matrix, which we discuses in Section \ref{sec:NormalForm}.
		Unfortunate the size of the matrices means that a straightforward algorithm for computing the Smith normal form will only produce results for cases that 
		could have been computed by hand, since the integers in intermediary forms of the matrix during the normalization procedure become
		too large or too small for the computer to cope with.
		
		In order to over come this in Subsection \ref{sec:RankedNormal}
		we describe an algorithm found in \cite{SmithForm}, which reduces the matrix in a more intelligent fashion, keeping entries closes to $0$.
		This is effective at the expense of computation time.
		This approach produces many more results, 
		however eventually the matrices become so large that the computer cannot produce the Smith normal form in a reasonable amount of time.
		
		Another approach attempted in Subsection \ref{sec:ModuloNormal} is to compute the Smith normal form of the matrix modulo a prime.
		This is the computationally most effective technique, however this method will not detect the multiplicity of the torsion as a power of the prime. 
		
		In the finale section, Section \ref{sec:results} we present the results of our algorithms and discuss the patterns observed.
		
		\subsection{Image matrix}\label{sec:ImageMatrix}
			
			In this section we present an algorithm to produce a matrix $X$ associated with the image of the $d^2$
			differential at a particular coordinate on the $E_2$-page of the spectral sequence.
			In subsection \ref{sec:structure} we describe the structure of $X$ in terms matrices $E^{i,x},E_i$ and $F$.
			In subsection \ref{sec:Subset}, \ref{sec:SubMulti} and \ref{sec:FixedSubmultiset} we present the algorithms to produce matrices $E,E_i$ and $F$.
			Finally in subsection \ref{sec:Diffmatrix} we present the algorithms that produces $X$.
			Throughout this section we use Proposition \ref{prop:Homogeneous-1} and assume that the symmetric ideal is generated by
			complete homogeneous symmetric polynomials $h_1,\dots,h_n$ in $n$ variables.
			
			\subsubsection{Matrix structure}\label{sec:structure}
			
				For $n\geq 1$, $0\leq x\leq n(n+1)/2$ and $1\leq y \leq n$ the matrix $X$ to be produced by our algorithm will have the following form
				
					\[
						\begin{blockarray}{cc}
																									& \hat{y}_{i_1,\dots,i_y}X\tilde{P} \\
							\begin{block}{c(c)}
								x_2\hat{y}_{i_1,\dots,i_{y+1}}XP	& {A}  \\
								\cline{2-2}
								h_iX\bar{P}												& {B}	 \\ 
							\end{block}
						\end{blockarray}
					\]

				where $P,\tilde{P},\bar{P}\in \mathbb{Z}[\gamma_1,\dots,\gamma_n]$ have degrees $x+1$, $x$, $x-\deg{h_i}$ respectively, $1\leq i \leq \max(x,n+1)$,
				$1\leq i_1<\cdots<i_{y+1}\leq n$ and $1\leq {i}_1<\cdots<{i}_y\leq n$.
				The matrix $A$ has rows representing the image of the $d^2$ differential
				and $B$ is the matrix whose rows representing a spanning set of of the symmetric ideal in degree $x$.
				The ordering on the basis of elements of $P,\tilde{P},\bar{P}$ and $\hat{y}_{i_1,\dots,i_y}$ is not important as long as the same order is consistently used.
				
				Let $F$ be the $\binom{n}{y-1}$ by $\binom{n}{y}$ matrix whose rows represent size $y-1$ subset of an $n$ set
				and whose columns represent size $y$ subset of an $n$ set,
				with an entry $1$ if the size $y-1$ subset is contained in the size $y$ subset and $0$ otherwise.
				
				Recall that there is a bijection between monomials in $n$ variables of a given degree and multisets of the same size.
				For $0\leq i \leq x-1$, let $E^{i,x}$ be the matrix whose rows represent size $i$ submultisets of an $n$ set
				and whose columns represent size $x$ submultisets of an $n$ set,
				with an entry $1$ if the size $i$ submultiset is contained in the size $x$ submultiset and $0$ otherwise.
				Note that if $i=0$ then $E^{i,x}$ will be a $1$ by $\multiset{n}{x}$ matrix of ones since the empty multiset is contained in all multisets.  
				Let $E_i$ for $1\leq i \leq n$, be the matrix whose rows represent size $x$ submultiset of an $n$ set containing at least one $i$
				and columns represent size $x$ submultiset of an $n$ set,
				with an entry $1$ if the submultisets are equal and $0$ otherwise.
				The matrix $E^{i,x}$ is a $\multiset{n}{i}$ by $\multiset{n}{x}$ matrix and $E_i$ is a $\multiset{n}{x-1}$ by $\multiset{n}{x}$ matrix.
				
				From equation (\ref{eq:boundary}) Lemma \ref{lem:1to1spec}, we have
				\begin{equation*}
					d^2(x_2\hat{y}_{i_1,\dots,i_{y-1}}XP)=\sum_{t=1}^{y-1}{(-1)^{t+1}\hat{y}_{i_1,\dots,\hat{i}_t,\dots,i_y}X\gamma_tP}.
				\end{equation*}
				Hence matrix $A$ can be further broken down into $\multiset{n}{x-1}$ by $\multiset{n}{x}$ sub matrices $A_{i_1,\dots,i_{y+1}}^{{i'}_1,\dots,{i'}_y}$
				corresponding to rows $x_2\hat{y}_{i_1,\dots,i_{y+1}}X$ and columns $\hat{y}_{{i'}_1,\dots,{i'}_y}X$
				for $1\leq i_1<\cdots<i_{y+1}\leq n$ and $1\leq {i'}_1<\cdots<{i'}_y\leq n$ where
				\begin{equation*}
					A_{i_1,\dots,i_{y+1}}^{{i'}_1,\dots,{i'}_y}=
					\begin{cases}
								0 		&\mbox{if  }  \{ {i'}_1,\dots,{i'}_y \} \nsubseteq \{ i_1,\dots,i_{y+1} \}
								\\
								(-1)^{j+1}E^{x-1,x}+E_i &\text{if  }  \{ {i'}_1,\dots,{i'}_y,i \} = \{ i_1,\dots,i_{y+1} \} \text{ and } i=i_j
								.
						\end{cases}
				\end{equation*}
				The position of the non-zero $A_{i_1,\dots,i_{y+1}}^{{i'}_1,\dots,{i'}_y}$ is determined with respect to ${i'}_1,\dots,{i'}_y$ and $i_1,\dots,i_{y+1}$
				by non-zero entries of the matrix $F$.
				
				The matrix $B$ can be further broken down into the diagonal sum
				
				\[
					\begin{blockarray}{ccccccc}
						 &  &  & \hat{y}_{i_1,\dots,i_y} &  & & \\
						\begin{block}{c(ccccc)c}
										& {B'}   & {0}	 & \cdots	& {0}	 & {0}	&\\
										& {0}	   & {B'}  & 				& {0}	 & {0}	 &\\
							{h_iX}& \vdots & 			 & \ddots & 		 & \vdots	&\\
										& {0}	   & 	{0}  & 				& {B'} & {0}		 &\\
										& {0}	   & 	{0}  & \cdots & {0}	 & {B'} 		&.\\
						\end{block}
					\end{blockarray}
					\]
				
				Where $B'$ is given by
				
				\[
					\begin{blockarray}{ccc}
						&  & \\
						\begin{block}{c(c)c}
							{h_2\bar{P}} 						 & {E^{x-2,x}}   				&   \\
							{h_3\bar{P}} 						 & {E^{x-3,x}}	   			&   \\
							\vdots									 & \vdots 							&   \\
							{h_{\max(x,n+1)}\bar{P}} & {E^{x-\max(x,n+1),x}}& . \\
						\end{block}
					\end{blockarray}
					\]
			
			\subsubsection{Subset matrix}\label{sec:Subset}
			
				In this section we present an algorithm that will produce an array of two matrices, $E\{ 2 \}$ which is the matrix $F$ described in Subsection \ref{sec:ImageMatrix}
				and an $\binom{n}{y}$ by $n$ non-negative integer matrix $E\{ 1 \}$,
				where rows represent $y$ element subsets of an $n$ set and columns the elements of the $n$ set.
				Matrix $E\{ 1 \}$ has a zero entry if the set element corresponding to the column is contained in the set
				otherwise $E\{ 1 \}$ has positive integer entry, the position (in the ordering of the basis) of the corresponding $y+1$ element subset obtained by 
				adding the element corresponding to the column to the $y$ subset.
				
				The steps in the algorithm are as follows.
				\begin{enumerate}
					\item
					If $y=0$, then output $E\{ 1 \}$ as a vertical $n$-vector of ones and $E\{ 2 \}$ as a row $n$-vector of $1$ to $n$, then terminate the algorithm.
					\item
					If $y\neq 0$, then generate two matrices $p$ and $q$ whose rows are all $y+1$ and $y$ subsets of an $n$ set respectively.
					\item
					Set  $E\{ 1 \}$ and  $E\{ 2 \}$ to be zero matrices of the correct size.
					\item
					For each row $i$ of the matrix $p$ compare with row $a$ of $q$ with a column element $j$ removed.
					If they are equal set coordinate $(i,j)$ of $E\{ 2 \}$ equal to one
					and coordinate $(i,a)$ of $E\{ 1 \}$ to be the $j^{\text{th}}$ element of row $a$ in $q$.  
				\end{enumerate}
				The Matlab program "Subsets$(n,y)$" to implement the procedure is given below.
				
				\begin{lstlisting}[frame=single]
				
function E = Subsets(n,y)
 
E={0};%defines E as an array
 
%check for exceptional first for empty set case otherwise proceeds with the general case
 
if y==0
 
    E{2}=ones(n,1);
    E{1}=transpose(1:n);
    
else
    w=nchoosek(n,y);%sets the width of the matrix E{1}
    h=nchoosek(n,y+1);%sets the height of E{1} and E{2}
    P=zeros(h,y+1); %records the position of subsets intersections, will eventually be E{1}
    D=zeros(h,w); %records the rows of h at which an intersection  occurs, matrix will eventually be E{2}
    p = nchoosek(1:n,y+1);%list of all y+1 subset of set {1,2,dots,n}
    q = nchoosek(1:n,y);%list of all y subset of set {1,2,dots,n}
       for i = 1:h %i corresponds to row of the matrix p
           r=y(i,:);%selects i-th row of  p (d+1 subset of n set) and stores as r
           for j=1:y+1 %j corresponds to which element is removed from the y+1 subset of the n set
               rtemp=zeros(1,y);
               %the next two loops store in rtemp the row of p missing the j th entry 
               for a=1:j-1
                   rtemp(a)=r(a);
               end;
               for a=j+1:y+1
                   rtemp(a-1)=r(a);
               end;
               for a=1:w %check rtemp (d+1 subset of n set without entry j) to see which row of q (y subset of n set) it i and records this information into P and D
                   if rtemp==q(a,:)
                       P(i,j)=r(j);%records corresponding row of q(y subset of an n set)
                       D(i,a)=1;   %1 placed in the row corresponding to y subset of n set column corresponding to y+1 subset of n set
                       break %end "a" loop since there is only one case to find
                   end;
               end;
           end;
       end;
    
    %the completed matrices are recorded as E{1} and E{2}
    E{1}=P;
    E{2}=D;
    
end;
				
				\end{lstlisting}

			\subsubsection{Submultiset matrix}\label{sec:SubMulti}
			
				In this subsection we present an algorithm to produce the matrices $E^{i,x}$ defined in Subsection \ref{sec:ImageMatrix}.
				Before this we require a algorithm to produce for $1\leq d \leq n$, a $\multiset{n}{d}$ by $n$ non-negative integer matrix
				whose rows represent multisets of size $d$ from an $n$ set and columns the elements of the $n$ set.
				This is the same problem as forming a non-negative integer matrix whose rows are all $n$ vectors with row sum $d$.
				The following Matlab program which can be found at \cite{MultisetPage} achieves this.
				
				\begin{lstlisting}[frame=single]
function M = Multiset(n,d)

%d The required sum (dimension)
%n The number of elements in the rows (number of variables)
%produces a matrix of all n-vector in non-negative integers whose sum is d
%with rows representing monmials in n varaibles of degree d

   d=d+n;

   c = nchoosek(2:d,n-1);
   m = size(c,1);
   M = zeros(m,n);
   for ix = 1:m
     M(ix,:) = diff([1,c(ix,:),d+1]);
   end;

   M=M-ones(size(M,1),size(M,2));
				\end{lstlisting}
			
				Now we present and algorithm that outputs $E^{i,x}$.
				Given $n\geq 1$ and $1\leq a \leq b$, the program outputs an array of matrices $C\{ i+1\}$ for $i$ between $a$ and $b$.
				Where each $C{i+1}$ is an $\binom{n+i-1}{i}$ by $\binom{n+b-1}{b}$ matrix
				with rows corresponding to size $i$ multiset of an $n$ set and
				columns size $b$ multisets of an n-set.
				Each $C\{ i+1 \}$ has entry $1$ if the size $i$ multiset is contained in the size b multiset and is $0$ otherwise.
				The steps in the algorithm are as follows.
					\begin{enumerate}
						\item
							Using the previous function, generate for each value $i$ between and including  $a$ and $b$, generate an $\multiset{n}{i}$ by $i$
							matrix $B\{ i+1 \}$ of $i$ multisets of an $n$ set.
						\item
							Form for each $i$ between and including $a$ and $b$ create zero matrices matrices $C\{ i+1 \}$ of size $\multiset{n}{i}$ by $\multiset{n}{b}$.
						\item
							For each $i$ between and inducing $a$ to $b$ do,
							for $j$ from $1$ to $\multiset{n}{i}$ and $k$ from $1$ to $\multiset{n}{b}$,
							in position $(j,k)$ of $C\{ i+1 \}$ put a $1$ if multiset on row $j$ of $B\{ i+1 \}$ is contained in the multiset on row $k$ of $B\{ b+1 \}$.
					\end{enumerate}
			
				The Matlab program "Submultiset$(a,b,n)$" to implement the procedure is given below.
				
				\begin{lstlisting}[frame=single]
function [C] = Submultiset(a,b,n)

B={0};%defines B to be an array

%Assigns to B{i} the positive integer matrix whose rows represent all multiset of size i of an n set
for i=a+1:b+1  
    B{i}=Multiset(n,i-1);
end;

C={0};%defines C to be an array

%vector l stores the size of matrices B{i} in column i
l=0:b;
for i=a+1:b+1
    l(i)=size(B{i},1);
end;

%creates an array of zero matrices C of correct size for output 
for i=a+1:b+1
    C{i}=zeros(l(i),l(b+1)); 
end;

%for each i place a 1 at position (j,k) of C{i} if the jth size i-1 multiset is contained in the kth size b multiset
for i=a+1:b+1
    for j=1:l(i)
        for k=1:l(b+1)
            C{i}(j,k)=all((B{b+1}(k,:)-B{i}(j,:))>=0);
        end;
    end;
end;
				\end{lstlisting}

			\subsubsection{Fixed element submultiset matrix}\label{sec:FixedSubmultiset}
			
				In this subsection we present an algorithm to produce matrices $C\{ i \}$, which are the martrices $E_i$ defied in Subsection \ref{sec:structure},
				for each $1 \leq i \leq x$.
				This is a $\multiset{n}{x-1}$ by $\multiset{n}{x}$ matrix whose rows represent size $x$ multisets of an $n$ set containing  at least one of element $i$. 
				and columns represent size $x$ multisets of an $n$ set.
				The matrix has an entry $1$ is the multiset of row is equal to the the multiset of the column.  
				The steps in the algorithm are as follows.
					\begin{enumerate}
						\item
							Using the function "Multiset" of Subsection \ref{sec:SubMulti},
							generate an $\multiset{n}{x}$ by $x$ matrix $W$ of size $x$ multisets of an $n$ set
							an generate an $\multiset{n}{x-1}$ by $x-1$ matrix $H$ of size $x-1$ multisets of an $n$ set.
						\item
							For each $i$ from $1$ to $n$ create a $\multiset{n}{x-1}$ by $\multiset{n}{x}$ matrix $M$ of zeros.
						\item
							For each size $x-1$ multiset $j$ of $H$ add in addition element $i$ and check to see which which size $x$ multiset $k$ of $W$ it is.
							Change element $(j,k)$ of $M$ to a $1$.
						\item
							Record the current $M$ at $C\{ i \}$ before moving to the next $i$.
					\end{enumerate}
			
				The Matlab program FixedSubmultiset$(x,n)$ to implement the procedure is given below.
				
				\begin{lstlisting}[frame=single]
function [C] = FixedSubmultiset(x,n)
 
C{1}=0;%defines C as an array
 
%Assigns to W the positive integer matrix whose rows represent all multiset of
size x in of n set
W = Multiset(n,x);
 
%Assigns to H the positive integer matrix whose rows represent all multiset of
size x-1 in of n set
H = monomials(n,x-1);
 
h=size(H,1);
w=size(W,1);
 
for i=1:n%i is the element of the n set that will be included into each size x-1 multiset
    M=zeros(h,w);%creates a zero matrix of the correct size
    Htemp=H;
    for j=1:h
        Htemp(j,i)=Htemp(j,i)+1;%add in the extra element i to each row j of H
        for k=1:w
            if Htemp(j,:)==W(k,:)%test to see which sixe x multiset the new multiset is
                M(j,k)=1;        %and records the result with a 1 in the correct column
            end;
        end;
    end;
    C{i}=M;%records the final matrix as C{i}
end;
				\end{lstlisting}

			\subsubsection{Differential matrix}\label{sec:Diffmatrix}
				
				In this final subsection we present an algorithm using the programs of Subsection \ref{sec:Subset}, \ref{sec:SubMulti} and \ref{sec:FixedSubmultiset}
				given $n\geq 2$, $x\geq 1$ and $y\geq 0$ to produce a matrix $A$ which is the one described in Subsection \ref{sec:structure}.
				The steps in the algorithm are as follows.
					\begin{enumerate}
						\item
							Calculate the number of generators in the symmetric ideal $sl$ by setting $sl=\min(x,,n+1)$. 
						\item
							Generate in array $C\{ i+1 \}$ the "Submultiset" matrices for $i$ between $x-sl$ and $x$ for a set of size $n$.
						\item
							Create the part of the output matrix $A$ corresponding to the symmetric ideal as a matrix $B$ by
							for each $i$ between $x-sl$ and $x$ stacking the $C\{ i+1 \}$ on top of each over and forming a diagonal sum of $\binom{n}{y}$ of these matrices.
						\item
							Generate for $i$ between $1$ and $n$ an array of matrices $E \{ i \}$ the "FixedSubmultiset" matrices for multisets of size $x$ and set of size $n$.
						\item
							Generate in a matrices $M\{1\}$ and $M\{2\}$, the "Subsets" matrix for value $x$ and set of size $n$.
						\item
							Create a $\binom{n}{y+1}\multiset{n-1}{x-1}$ by $\binom{n}{y}\multiset{n}{x}$ zero matrix $A$ to hold the image of the $d^2$ differential.
						\item
							For each row of $M\{2\}$ set or a value $k$ starting at $0$,
							moving along rows the row for each entry $(i,a)$ of $M\{2\}$ that is a $1$ increase the value of $k$ by $1$.
							Each time the value of $k$ increases place in $A$ with its top left had entry at position $((i+1)\multiset{n}{x},(a-1)\multiset{n}{x})$,
							a copy of $(-1)^{k+1}(C\{ x \}+E\{ M\{ 1 \}(i,y+1-k) \})$.
						\item
							Extend $A$ by stacking it on top of the matrix $B$, to form the final output.
					\end{enumerate}
			
				The Matlab program "DifferentialMatrix$(n,x,y)$" to implement the procedure is given below.
			
				\begin{lstlisting}[frame=single]
function [A] = DifferentialMatrix(n,x,y)
 
%forms a zero matrix A of the correct size, height h width w
 
sl=min([x,n+1]); %sl is the number of generators in the symmetric ideal of dgree less than or equal to n
 
%s will hold in each entry the number of multiples of h_i by a monomial for i=2 to the minimum of x and n+1
s=zeros(sl,1);
for i=2:sl
    s(i)=nchoosek(n,y)*nchoosek(n+x-i-1,x-i)+s(i-1);%number of monomials of degrre x-i
end;
 
 
h=nchoosek(n,y+1)*nchoosek(n+x-2,x-1)+s(sl);%the total height of the outputs matrix
 
wHat=nchoosek(n,y);%the number of \hat{y_{i_1,\dots,i_y}} in total
 
wMon=nchoosek(n+x-1,x);%the number of monomials of degree x in n variables
 
w=wMon*wHat;%the total width of the output matrix
 
ns=h-s(sl);%ns is the number of generators of the image as rows in the matrix
 
A=zeros(ns,w);%creates a zero matrix of the correct size
 
%First place the symmetric function rows at the bottom of the matrix
 
C=Submultiset(x-sl,x,n);
 
D=zeros(1,wMon);%creates zero row of the same width as the C
 
for i=2:sl%i represents the degree of the symmetric generator
    D=[D; C{x+1-i}];%stacks submultset matrices for different generators
end;
 
D=D([2:size(D,1)],[1:wMon]);%removes zero row
B=D;
 
for i=2:wHat%repeats the matrix D for each \hat{Y}_{i_1,\dots,i_y} along diagonal
   B=blkdiag(B,D);
end;
 
%place the d2 image rows
 
E=FixedSubmultiset(x,n);
 
Ctemp=C{x};
hMon=size(Ctemp,1);
 
M=Subsets(n,y);
hHatsType=M{2};
hHatsPosition=M{1};
hHats=size(hHatsType,1);
 
for i=1:hHats
    temph=(i-1)*hMon;%record the top row-1 of the current position being considered
    temp=0;
    for a=1:wHat
        tempw=(a-1)*wMon;%record the left most column-1 of the current position being considered
        if hHatsType(i,a)==1
            Etemp=E{hHatsPosition(i,y+1-temp)};
            A([temph+1:temph+hMon],[tempw+1:tempw+wMon])=(-1)^(temp)*(Ctemp+Etemp);
            temp=temp+1;
        end;
    end;
end;
 
A=[A;B];%combines the image matrix vertically with the symmetric ideal matrix
				\end{lstlisting}

		\subsection{Normal form}\label{sec:NormalForm}
		
			The integral Smith normal form of an integral matrix $M$ is the unique diagonal matrix $N$ obtained from $M$ by integral row an column operations
			such that entries on the leading diagonal are non-negative integers in decreasing order of size.
			The most straightforward process to obtain matrix $N$ from matrix $M$ is as follows.
			\begin{enumerate}
				\item
					Set the current position at the top left hand entry of the matrix.
				\item
					Compute $R$, the greatest common divisor the the row containing the current position.
				\item
					Use integral column operation to reduce the current position to $R$ and then all other entries on that row to $0$.
				\item
					Compute $C$ the greatest common divisor the the column containing the current position.
				\item
					Use integral row operation to reduce the current position to $C$ and then all other entries on that column to $0$.
				\item
					Repeat steps $2$, $3$, $4$ and $5$ with the current position at each entry on the lending diagonal in turn.
				\item
					Reorder the leading diagonal with the largest values first.
			\end{enumerate}
			
			In Subsection \ref{sec:RankedNormal} we discuss how the procedure can be improved to avoid very large or very small values occurring during it implementation.
			In Subsection \ref{sec:ModuloNormal} we show how to adapted the procure to compute the Smith normal form modulo $p$, for some prime $p$.
		
			\subsubsection{Ranked normal form}\label{sec:RankedNormal}
				
				In this subsection we describe a procedure from \cite{SmithForm} which improves the elementary procedure outlined at the begging of the section.
				The the main problem that can occur during the implementation of an algorithm computing the Smith normal form is at a intermediary stages 
				the entries of the matrix become too large or too small for the computer to handle, causing rounding errors or a crash.
				The idea of the solution is rather than just reducing the matrix along the leading diagonal, before performing the row and column
				reductions move to the current position to the entry of the matrix
				which after the reduction, will minimise the maximal magnitude of entries in the matrix.
				
				Suppose we have a matrix $M=(m_{i,j})$, on which we want to perform steps $2$, $3$, $4$ and $5$ in the process above from a position
				that minimises the magnitude of values in the resulting matrix.
				For each column $m_{*,1}$ and $m_{*,j}$ of $M$, step $3$ repeats the process of replacing column $m_{*,1}$ with $x_1m_{*,1}+x_2m_{*,j}$
				and column $m_{*,j}$ with $m_{1,j}\gcd(m_{1,1},m_{1,j})m_{*,1}-m_{1,1}\gcd(m_{1,1},m_{1,j})m_{*,j}$,
				where $x_1,x_2\in \mathbb{Z}$ are such that $\gcd(m_{1,1},m_{1,j})=x_1m_{1,1}+x_2m_{1,j}$.
				Hence if the first $k$ columns have first value $m_{1,1},\dots,m_{1,k}$ such that for each $l$ less than $k$,
				$\gcd(m_{1,1},\dots,m_{1,l-1})>\gcd(m_{1,1},\dots,m_{1,l})$.
				Then after $k$ interactions the first column is
				\begin{equation*}
					m_{*,1}\prod_{t=1}^{k-1}x_{2t-1}+\sum^k_{l=2}\big( m_{*,l}x_{2(l-1)}\prod^{k-1}_{t=l} x_{2l-1} \big)
				\end{equation*}
				where $x_{2l-1}$ and $x_{2l}$ are such that $x_{2l-1}\gcd(m_{1,1},\dots,m_{1,l-1})+x_{2l}m_{1,l}=\gcd(m_{1,1},\dots m_{1,l})$.
				These values are then used in subsequent steps, so if they become large entry in the matrix become cumulatively large over those subsequence steps.
				
				Given a vector $X$ such that $X\cdot m_{*,1}=\gcd(m_{1,1},\dots,m_{m,1})$,
				in general we would like to minimize
				\begin{equation*}
					\max_{i,j}|m_{i,j}-\frac{X\cdot m_{*,j}}{X \cdot m_{*,1}}m_{i,1}|
				\end{equation*}
				which we call the pivot value on the first column of $M$.
				Clearly we could calculate this pivot value for any column of $M$.
				We could also calculate in the same way a pivot value for the rows of $M$
				and multiply the pivot value for each column by the pivot value for row of each entry.
				This gives us the matrix of the same size as $M$ which we call the pivot value matrix.
				The entries with the smallest values in the pivot value matrix should be the best candidates to use as the current positions in our standard 
				Smith normal form procedure.
				Hence given one such value in $M$ we move this row and column to be the first row and column in the matrix and perform steps $2$, $3$, $4$ and $5$ above. 
				For a compete description of the procedure see \cite{SmithForm}.
				
				The Matlab program "PivotMinNomal$(A)$" implements the procedure to calculate the Smith normal form of a matrix $A$ using the improved method above.
				The function "PivotMinNomal$(A)$" call upon "PivotValue$(A)$" which computes the pivot value matrix of a given matrix $A$,
				which in turn calls upon function "VecGCD$(V)$" that given an integer vector $V$ computes using the Euclidean algorithm the greatest common divisor $G$
				of the values of $V$ and vector of integers $X$ whose scalar product with $V$ if $G$.
				
				\begin{lstlisting}[frame=single]
function [A] = PivotMinNomal(A)
 
%Given a matrix A finds its Smith normal form in a way that attempts to minimise the magnitude of intermediary values 
 
[y,x]=size(A);%records the size of A
 
max=min(x,y);%size of the leading diagonal
 
for i=1:max
 
    null=1;
    
    for a=i:x %check to see if all remaining entries are zero
        for b=i:y
            if A(b,a)
                null=0;
                break
            end;
        end;
        if null
        else
            break
        end;
    end;
    
    if null
        break
    end;
    
    B=zeros(y-i+1,x-i+1);
    
    %takes B the part of the matrix which we still need to reduce
    for a=i:x
        for b=i:y
            B(b-i+1,a-i+1)=A(b,a);
        end;
    end;
       
    B=PivotValue(B);
    
    MinPiv=[1,1,inf];
    
    for a=1:size(B,1) %finds non-zero value with smallest pivot value
        for b=1:size(B,2)
            if A(i+a-1,i+b-1)
                if B(a,b)<MinPiv(3)
                    MinPiv=[a,b,B(a,b)];
                end;
            end;
        end;
    end;
    
    p=MinPiv(1);
    q=MinPiv(2);
    
    A(:,[i,q+i-1])=A(:,[q+i-1,i]);
    A([i,p+i-1],:)=A([p+i-1,i],:);
    
    %now perform GCD reduction on the first row column for the top left position.
    in=1;
    
    while in
       
        
        if A(i,i)<0
            A(i,:)=-1*A(i,:);
        end;
 
        for a=i+1:y
            if A(a,i)<0
                A(a,:)=-1*A(a,:);
            end;
            A(a,:)=A(a,:)-floor(A(a,i)/A(i,i))*A(i,:);
        end;
 
        for a=i+1:x
            if A(i,a)<0
                A(:,a)=-1*A(:,a);
            end;
            A(:,a)=A(:,a)-floor(A(i,a)/A(i,i))*A(:,i);
        end;
        
        %check to see if all first row and column are zero except top left.
        out=1;
        
        for a=i+1:y
            if A(a,i)
                out=0;
            end;
        end;
        
        for a=i+1:x
            if A(i,a)
                out=0;
            end;
        end;
        
        if out
           break 
        end;
        
        %finds new pivot in fist row or column and repeat reduction
        
        B=zeros(y-i+1,x-i+1);
    
        for a=i:x
            for b=i:y
                B(b-i+1,a-i+1)=A(b,a);
            end;
        end;
        
        B=PivotValue(B);
        
        V=B(:,1).';
        
        H=B(1,:);
        
        piv=[1,1,inf];
        
        p=abs(A(i,:));
        q=abs(A(:,i)).';
        
        P=p(1);
        
        for a=2:size(p,2)
            if P<p(a)
                P=p(a);
            end;
        end;
        
        for a=1:size(q,2)
            if P<q(a)
                P=q(a);
            end;
        end
        
        temp=P;
        
        U=0;
        
        for a=i:x
            if abs(A(i,a))==temp
                U=U+1;
            end;
        end;
        
        for a=i+1:y
            if abs(A(a,i))==temp
                U=U+1;
            end;
        end;
        
        for a=1:size(V,2) %find lowest pivot value
            if A(a+i-1,i)
                if V(a)<piv(3)
                    if temp>abs(A(a+i-1,i))
                        piv=[1,a,V(a)];
                    else
                        if U>1
                            piv=[1,a,V(a)];
                        end;
                    end;
                end;
            end;
        end;
    
        for a=1:size(H,2)
            if A(i,a+i-1)
                if H(a)<piv(3)
                    if temp>abs(A(i,a+i-1))
                        piv=[0,a,H(a)];
                    else
                        if U>1
                            piv=[0,a,H(a)];
                        end;
                    end;
                end;
            end;
        end;
   
        if piv(1)
            A([i,piv(2)+i-1],:)=A([piv(2)+i-1,i],:);
        else
            A(:,[i,piv(2)+i-1])=A(:,[piv(2)+i-1,i]);
        end;
        
    end;
    
end;
 
%rearranges elements on diagonal smallest towards top left.
swap=1;
 
while swap
    swap=0;
    for i=1:max-1
        if A(i+1,i+1)==0
            break
        end;
        if A(i,i)>A(i+1,i+1)
            temp=A(i,i);
            A(i,i)=A(i+1,i+1);
            A(i+1,i+1)=temp;
            swap=1;
        end;
    end;
end;
				\end{lstlisting}
				
				\begin{lstlisting}[frame=single]
function [P] = PivotValue(A)
 
%Given matrix A outputs its matrix P of pivot values 
 
[y,x]=size(A);%records size of A
 
P=zeros(y,x);%output matrix of the correct size
 
%computes value for columns
for k=1:x
    
    [Xc,Gc] = VecGCD(A(:,k).');%computes gcd for current column
 
    if Gc%checks the column was not a zero vector
    
        %if the first value of gcd scalar vector is zero changes it to an equivalent vector where the first entry in non-zero
        if Xc(1)==0
            temp=(A(1,k))/Gc;
            Xc=Xc*(temp+1);
            Xc(1)=-1;
        end;
 
         %computes the values of the matrix if this column were pivot
        ColVal=zeros(y,x);
        for i=1:y
            for j=1:x
                ColVal(i,j)=abs(A(i,j)-((dot(Xc,A(:,j)))/(dot(Xc,A(:,k)))*A(i,k)));
            end;
        end;
 
        temp=max(max(ColVal));%maximum value in column the pivot matrix
 
        P(:,k)=P(:,k)+temp*ones(y,1);%records max value in the corresponding column of P
    
    else
    
        P(:,k)=P(:,k)+inf*ones(y,1);%records zero column as infinite pivot value
        
    end;
    
end;
 
%computes value for rows
for k=1:y
     
     [Xr,Gr] = VecGCD(A(k,:));%computes gcd for current row
 
     if Gr%checks the row was not a zero vector
        
         %if the first value of gcd scalar vector is zero changes it to an equivalent vector where the first entry in non-zero
         if Xr(1)==0
             temp=(A(k,1))/Gr;
             Xr=Xr*(temp+1);
             Xr(1)=-1;
         end;
        
         %computes the values of the matrix if this column were pivot
         RowVal=zeros(y,x);
         for i=1:y
             for j=1:x
                 RowVal(i,j)=abs(A(i,j)-((dot(Xr,A(i,:)))/(dot(Xr,A(k,:)))*A(k,j)));
             end;
         end;
 
         temp=max(max(RowVal));%the maximum value in the pivot matrix for this row
 
         %multiplies the row of P by this max value
         for a=1:x
             P(k,a)=P(k,a)*temp;
         end;
     
     else
     
         for a=1:x
             P(k,a)=inf;%records zero row as infinite pivot value
         end;
         
     end;
         
 end;

				\end{lstlisting}
				
				\begin{lstlisting}[frame=single]
function [X,G] = VecGCD(V)
 
%given a vector V outputs gcd G and vector of scalers X whose scalar product
%with V is G
 
s=size(V,2);%number of elements in V
 
P=eye(s);%for recording intermediary values for X
 
minV=[1,inf];
 
temp=0;
 
neg=zeros(1,s);%for recording sign changes
 
%ensures V is non-negative integer vector and vectors where the sign changes
for i=1:s
    if V(i)<0
        V(i)=-V(i);
        neg(i)=1;
    end;
end;
 
%first checks for exceptional case when V is the zero vector
if V==zeros(1,s)
    X=zeros(1,s);
    G=0;
else
%computes G and X using Euclidean algorithm
while minV(1)
 
    minV=[0,inf];
 
    %finds the smaes value in V
    for i=1:s
        if V(i)
            if V(i)<minV(2)
                minV=[i,V(i)];
            end;
        end;
    end;
    
    %if the minimum positive value is unchanged this is the gcd and the procdure terminates 
    if temp==minV(2)
        G=minV(2);
        X=P(minV(1),:);
        break
    end;
    
    %reduce the vector v modulo its minimum value and records what was done in P
    if minV(1)
        for i=1:s
            if i==minV(1)
            else
                f=floor(V(i)/minV(2));
                V(i)=V(i)-f*minV(2);
                P(i,:)=P(i,:)-f*P(minV(1),:);
            end;
        end;
    end;
    
    temp=minV(2);
 
end;
end;
 
%assigns the correct sign to elements of X
for i=1:s
    X(i)=X(i)*(-1)^(neg(i));
end;

				\end{lstlisting}

			\subsection{Modulo $p$ normal form}\label{sec:ModuloNormal}
				
			In this subsection we present an algorithm to compute the Smith normal form of a matrix $A$ modulo a prime $p$.
			Since the entries on the leading diagonal of a matrix in Smith Normal form are $0$, $1$ or a prime power,
			the entries on the leading diagonal of a matrix in Smith normal form with entries modulo $p$ will be either $0$ or $1$.
			Hence the important information in the matrix is the number of ones on the leading diagonal.
			Our algorithm will roundly follow the steps detailed at the beginning of the section with the following exceptions.
			\begin{itemize}
				\item
					each time the current position changes and at the end of the algorithm the whole matrix is reduced modulo $p$.
				\item
					The reduction of the current position to the greatest common divisor of its row and column is performed simultaneously.
				\item
					At the end of the procedure only the number of ones on the leading diagonal is output.
			\end{itemize}
			The Matlab program "ModuloNomalForm$(A,p)$" to implement the procedure is given below.
				
				\begin{lstlisting}[frame=single]
function [U] = ModuloNomalForm(A,p)
 
U=0;
 
h=size(A,1);%hight of A
w=size(A,2);%width of A
 
L=min(h,w);%the size of the leading diagonal
 
%The normal form procedure moves the current position along the leading diagonal
for a=1:L
    
    %reduces the matrix to it simplest integral representatives modulo p
    for i=a:h
        for j=a:w
            if A(i,j)>0
                A(i,j)=A(i,j)-floor(A(i,j)/p)*p;
            else
                A(i,j)=A(i,j)-floor(A(i,j)/p)*p;
            end;
        end;
    end;
    
    done=1;
    
    %checks to see if the current row and column are zero and if so proceeds to the next position on the leading diagonal
    if A(a,:)==zeros(1,w)
        if A(:,a)==zeros(h,1)
            done=0;
        end;
    end;
    
    %Use integral row and column operations to reduces the current position to the greatest common devisor of its the row, then all other entries to zero 
    while done
        
        %moves the smallest positive integer in the current row or column to the current position
        Low=[A(a,a),a,0];
                
        if Low(1)
        else
             Low(1)=inf;
        end;
                
        for i=a+1:h
            if A(i,a)
                if A(i,a)<Low(1)
                    Low=[A(i,a),i,0];
                end;
            end;
        end;
                
        for i=a+1:w
            if A(a,i)
                if A(a,i)<Low(1)
                    Low=[A(a,i),i,1];
                end;
            end;
        end;
                
        if Low(3)
            A(:,[a,Low(2)])=A(:,[Low(2),a]);
        else
            A([a,Low(2)],:)=A([Low(2),a],:);
        end;
                
        done=0;
        
        %reduces all non-zero entries in the current column by the integer in current position
        for i=a+1:h
            if A(i,a)
                A(i,:)=A(i,:)-floor(A(i,a)/A(a,a))*A(a,:);
            end;
            if A(i,a)
                done=1;
            end;
        end;
        
        %reduces all non-zero entries in the current row by the integer in the current position
        for i=a+1:w
            if A(a,i)
                A(:,i)=A(:,i)-floor(A(a,i)/A(a,a))*A(:,a);
            end;
            if A(a,i)
                done=1;
            end;
        end;
        
        %if no reductions took place then move the current position to the next position on the leading diagonal otherwise repeat from finding the smallest entry
        
    end;
        
end;
 
%reduces the final diagonal from of the matrix modulo p
for i=1:L
    if A(i,i)>0
        A(i,i)=A(i,i)-floor(A(i,j)/p)*p;
    else
        A(i,i)=A(i,i)+floor(A(i,i)/p)*p;
    end;
end;
 
temp=0;
 
%counts the number of non-zero entries on the leading diagonal of the normal form matrix
for i=1:L
    if A(i,i)
        temp=temp+1;
    end;
end;
 
U=w-temp;%outputs the umber of non-zero entries on the leading diagonal of the normal form matrix
				\end{lstlisting}

		\subsection{Results}\label{sec:results}
			
			In this section we present the our findings on the torsion of the $E_3$-page aided by a computer.
			We do this in the case of element of the form
			\begin{equation}\label{eq:Positions}
				E_3^{p,n-j+\dim(X)}(X\hat{y}_{i_1,\dots,i_j}) \;\; \text{and} \;\; E_3^{p,n+2m+\dim(X)}({(x_2)}_mXy_1\cdots y_n)
			\end{equation}
			for $1\leq j \leq n-1$,$m\geq 0$ and $X$ a monomial in $\Gamma_{\mathbb{Z}}(x_4,x_6,\dots,x_{2n})$.	
			The integral results from running the algorithms in Section \ref{sec:ImageMatrix} and Subsection \ref{sec:RankedNormal} for $n=2,3$ and $4$ are as follows.
			By Theorems \ref{thm:BottemRow} and \ref{thm:LastColumn} the bottom row and final column can be filled in without the aid of the computer.
			
			\begin{table}[ht]
			\centering
			\caption{Part of the $E^3$-page of the spectral sequence converging to $H^*(\Lambda(SU(3)/T^2);\mathbb{Z})$}
			\begin{tabular}{c|cccc}
			$\langle \hat{y}_{i_1}X \rangle$ 	& $\mathbb{Z}^2$	& $\mathbb{Z}^3$ & $\mathbb{Z}^2\oplus\mathbb{Z}_3$ & $\mathbb{Z}_3$ \\
			$\langle y_1y_2X \rangle$ 				&	$\mathbb{Z}$ 		& $\mathbb{Z}_3$ & $\mathbb{Z}_3$ 									&    $0$
			\end{tabular}
			\end{table}
			
			\begin{table}[ht]
			\centering
			\caption{Part of the $E^3$-page of the spectral sequence converging to $H^*(\Lambda(SU(4)/T^3);\mathbb{Z})$}
			\begin{tabular}{c|ccccccc}
			$\langle \hat{y}_{i_1}X \rangle$			& $\mathbb{Z}^3$ & $\mathbb{Z}^8$ & $\mathbb{Z}^{12}$ & $\mathbb{Z}^{13}$ & $\mathbb{Z}^9\oplus\mathbb{Z}_2$
																						& $\mathbb{Z}^4\oplus\mathbb{Z}_2$ & $\mathbb{Z}_4$ \\
			$\langle \hat{y}_{i_1,i_2}X \rangle$	& $\mathbb{Z}^3$ & $\mathbb{Z}^6$ & $\mathbb{Z}^7\oplus\mathbb{Z}_2$ & $\mathbb{Z}^6\oplus\mathbb{Z}_2\oplus\mathbb{Z}_4$ 
																						& $\mathbb{Z}^3\oplus\mathbb{Z}_2\oplus\mathbb{Z}_4$ & $\mathbb{Z}\oplus\mathbb{Z}_2$ & 0 \\
			$\langle y_1y_2y_3X \rangle$					& $\mathbb{Z}$ 	 & $\mathbb{Z}_4$ & $\mathbb{Z}_2$ & $\mathbb{Z}_2$ & $0$ & $0$ & $0$
			\end{tabular}
			\end{table}

			\begin{table}[ht]
			\centering
			\caption{Part of the $E^3$-page of the spectral sequence converging to $H^*(\Lambda(SU(5)/T^4);\mathbb{Z})$}
			\resizebox{1\hsize}{!}{$
			\begin{tabular}{c|lllllllllll}
			$\langle \hat{y}_{i_1}X \rangle$		 	 &$\mathbb{Z}^4$&$\mathbb{Z}^{15}$&$\mathbb{Z}^{32}$&$\mathbb{Z}^{51}$&$\mathbb{Z}^{65}$&$\mathbb{Z}^{68}$&$\mathbb{Z}^{58}$
																							 &$\mathbb{Z}^{40}\oplus\mathbb{Z}_5$&$\mathbb{Z}^{21}\oplus\mathbb{Z}_5$&$\mathbb{Z}^7\oplus\mathbb{Z}_5$&$\mathbb{Z}_5$\\
			$\langle \hat{y}_{i_1,i_2}X \rangle$		 &$\mathbb{Z}^6$&$\mathbb{Z}^{20}$&$\mathbb{Z}^{39}$&$\mathbb{Z}^{58}$&$\mathbb{Z}^{69}\oplus\mathbb{Z}_5$&?&?&?&?&?&0 \\
			$\langle \hat{y}_{i_1,i_2,i_3}X \rangle$ &$\mathbb{Z}^4$&$\mathbb{Z}^{10}$&$\mathbb{Z}^{16}\oplus\mathbb{Z}_5$&$\mathbb{Z}^{21}\oplus\mathbb{Z}_5^2$
																							 &$\mathbb{Z}^{23}\oplus\mathbb{Z}_5^3$&?&?&?&?&?&0 \\
			$\langle y_1y_2y_3y_4X \rangle$					 &$\mathbb{Z}$&$\mathbb{Z}_5$&$\mathbb{Z}_5$&$\mathbb{Z}_5$&$\mathbb{Z}_5$& 0 & 0 & 0 & 0 & 0 & 0 
			\end{tabular}
			$}
			\end{table}

			Each row of the table corresponds to a row of the spectral sequence divisible by the generators in the first column,
			but not divisible by $(x_2)_m$ for any $m\geq 1$ in any row except the bottom one.
			Rows ordered by the number of generators $y_i$ present with all $y_i$ present in the bottom row and one less in each row above it.
			Recall $\hat{y}_{i_1,\dots,i_j}=\frac{y_1\cdots y_n}{y_{i_1}\cdots y_{i_j}}$ for some $1\leq j \leq n-1$ and $1\leq i_1<\cdots<i_j\leq n$.
			The columns represent all the potentially non-zero entries on those rows, ordered by degree.
			That is even degree between and including $0$ and $(n+2)(n+1)/2$.
			
			Just recording the torsion in case $n=2$ and $n=3$ gives the following two tables.
			
			\begin{table}[ht]
			\centering
			\caption{Multiplicity of torsion on the $E^3$ of the spectral sequence converging to $H^*(\Lambda(SU(3)/T^2);\mathbb{Z})$}
			\label{n=2}
			\begin{tabular}{c|cccc}
			$\langle \hat{y}_{i_1}X \rangle$ 	&			-				& 			-			 	 & $\mathbb{Z}_3$ & $\mathbb{Z}_3$ \\
			$\langle y_1y_2X \rangle$ 				&			-				& $\mathbb{Z}_3$ & $\mathbb{Z}_3$ &    -
			\end{tabular}
			\end{table}
			
			\begin{table}[ht]
			\centering
			\caption{Multiplicity of torsion on the $E^3$ of the spectral sequence converging to $H^*(\Lambda(SU(4)/T^3);\mathbb{Z})$}
			\label{n=3}
			\begin{tabular}{c|ccccccc}
			$\langle \hat{y}_{i_1}X \rangle$			& - &  			-				&  				-			 &  							-									& $\mathbb{Z}_2$ 									 & $\mathbb{Z}_2$ & $\mathbb{Z}_4$ \\
			$\langle \hat{y}_{i_1,i_2}X \rangle$	& - &  			-				& $\mathbb{Z}_2$ & $\mathbb{Z}_2\oplus\mathbb{Z}_4$ & $\mathbb{Z}_2\oplus\mathbb{Z}_4$ & $\mathbb{Z}_2$ &  			-				 \\
			$\langle y_1y_2y_3X \rangle$					& - & $\mathbb{Z}_4$& $\mathbb{Z}_2$ & $\mathbb{Z}_2$ 									& 							-									 &  		-					&				-
			\end{tabular}
			\end{table}
			
			In table \ref{n=3} the result when $n=3$ are given. 
			
			Notices there is a symmetry in the table where if we remove the first column, the bottom and top rows are the reverse of each-over and the middle row
			is symmetric about its center.
			
			For large $n$ the an increasingly large matrix is used which greatly increase the time necessary to compute the smith normal form.
			We can use the modulo-p an algorithm to compute results over over a finite field of order prime $p$
			by replacing each coordinate of the matrix with is representative $0,\dots,p-1$ modulo $p$ after each step of the smith normal form algorithm.
			In this case a simpler algorithm can the used as the numbers in the matrix will never be larger than $p$ reducing the execution time.
			It can be shown that any torsion occurring on the $E^3$ page of the spectral sequence will be a divisor of $n+1$.
			\newline
			by...
			\newline
			Hence we can obtaining the rank of a matrix of a corresponding of the spectral sequence modulo a prime co-prime to $n+1$
			and subtracting this from the the result modulo 
			a prime divisor of $n+1$ will give us the multiplicity of the torsion at that position.
			Computing modulo a prime would allows us to obtain the multiplicity of the torsion,
			at the expenses of knowing the exact degree of the torsion away from a prime $n+1$.   
			Table (\ref{n=4}), contains the multiplicities of torsion on the $E^3$ page, when $n=4$.
			
			\begin{table}[ht]
			\centering
			\caption{Multiplicity of $5$-torsion on the $E^3$ of the spectral sequence converging to $H^*(\Lambda(SU(5)/T^4);\mathbb{Z})$}
			\label{n=4}
			\begin{tabular}{c|lllllllllll}
			$\langle \hat{y}_{i_1}X \rangle$				 & 0 & 0 & 0 & 0 & 0 & 0 & 0 & 1 & 1 & 1 & 1 \\
			$\langle \hat{y}_{i_1,i_2}X \rangle$		 & 0 & 0 & 0 & 0 & 1 & 2 & 3 & 3 & 2 & 1 & 0 \\
			$\langle \hat{y}_{i_1,i_2,i_3}X \rangle$ & 0 & 0 & 1 & 2 & 3 & 3 & 2 & 1 & 0 & 0 & 0 \\
			$\langle y_1y_2y_3y_4X \rangle$					 & 0 & 1 & 1 & 1 & 1 & 0 & 0 & 0 & 0 & 0 & 0 
			\end{tabular}
			\end{table}

			The symmetry in the torsion continues in table (\ref{n=4}), in addition the multiplicity of the torsion continues to increases in the center of the table 
			suggesting that these observations may continue to be true for larger $n$.

\newpage
\bibliographystyle{plain}
\bibliography{Ref}

\begin{thebibliography}{10}

\bibitem{Arkowitz}
M.~Arkowitz.
\newblock {\em Introduction to homotopy theory}.
\newblock Universitext. Springer, New York, 2011.

\bibitem{Borel}
A.~Borel.
\newblock Sur la cohomologie des espaces fibres principaux et des espaces
  homogenes de groupes de lie compacts.
\newblock {\em Annals of Mathematics}, 57(1):115--207, 1953.

\bibitem{bott1958}
R.~Bott.
\newblock The space of loops on a lie group.
\newblock {\em Michigan Math. J.}, 5(1):35--61, 1958.

\bibitem{torG/T}
R.~Bott and H.~Samelson.
\newblock The cohomology ring of {$G/T$}.
\newblock {\em Proc. Nat. Acad. Sci. U. S. A.}, 41:490--493, 1955.

\bibitem{AplicationsOfMorse}
R.~Bott and H.~Samelson.
\newblock Applications of the theory of {M}orse to symmetric spaces.
\newblock {\em American Journal of Mathematics}, 80(4):964--1029, 1958.

\bibitem{Cartan1914}
E.~Cartan.
\newblock Les groupes réels simples, finis et continus.
\newblock {\em Annales scientifiques de l'École Normale Supérieure},
  31:263--355, 1914.

\bibitem{StringTopology}
M.~{Chas} and D.~{Sullivan}.
\newblock {String Topology}.
\newblock {\em ArXiv Mathematics e-prints}, (9911159), November 1999.

\bibitem{MR1942249}
R.~Cohen and J.~Jones.
\newblock A homotopy theoretic realization of string topology.
\newblock {\em Math. Ann.}, 324(4):773--798, 2002.

\bibitem{AdvanceCombintorics}
L.~Comtet.
\newblock {\em Advanced Combinatorics}.
\newblock Springer Netherlands, 1974.

\bibitem{MatrixGroups}
M.~L. Curtis.
\newblock {\em Matrix Groups}.
\newblock Springer, 1984.

\bibitem{CohomologyLieGroup}
J.~Fung.
\newblock The cohomology of lie groups, April 2017.
\newblock http://math.uchicago.edu/~may/REU2012/REUPapers/Fung.pdf, 2012.

\bibitem{homology_Lflags}
J.~Grbi{\'c} and S.~Terzi{\'c}.
\newblock The integral {P}ontrjagin homology of the based loop space on a flag
  manifold.
\newblock {\em Osaka J. Math.}, 47(2):439--460, 2010.

\bibitem{PeriodicGeodesics}
D.~Gromoll and W.~Meyer.
\newblock Periodic geodesics on compact riemannian manifolds.
\newblock {\em J. Differential Geom.}, 3(3-4):493--510, 1969.

\bibitem{hatcher}
A.~Hatcher.
\newblock {\em Algebraic topology}.
\newblock Cambridge University Press, Cambridge, 2002.

\bibitem{HSS}
A.~Hatcher.
\newblock {\em Spectral Sequences in Algebraic Topology chaper 1}.
\newblock April 2017.
\newblock https://www.math.cornell.edu/~hatcher/SSAT/SSch1.pdf.

\bibitem{SmithForm}
G.~Havas and B.~S. Majewski.
\newblock Integer matrix diagonalization.
\newblock {\em Symbolic Computation}, 1997.

\bibitem{HepwothString}
R.~A. {Hepworth}.
\newblock {String Topology for Lie Groups}.
\newblock {\em ArXiv e-prints}, (arXiv:0905.1199), May 2009.

\bibitem{Killing1888}
W.~Killing.
\newblock Die zusammensetzung der stetigen endlichen transformationsgruppen.
\newblock {\em Mathematische Annalen}, 31:252--290, 1888.

\bibitem{ArtOfProgramming}
D.~E. Knuth.
\newblock {\em The Art of Computer Programming Volume 3 Sorting and Searching}.
\newblock AddisonWesley, 1998.

\bibitem{Kostant2009}
B.~Kostant.
\newblock {\em The Principal Three-Dimensional Subgroup and the Betti Numbers
  of a Complex Simple Lie Group}, pages 130--189.
\newblock Springer New York, New York, NY, 2009.

\bibitem{Kupers1010}
A.~Kupers.
\newblock The string topology structure of the lie groups {SU(n), U(n), Sp(n)},
  {$G_2$} and {$F_4$}.
\newblock {\em in preparation}, 2010.

\bibitem{Macdonald}
I.~G. Macdonald.
\newblock {\em Symmetric Functions and Hall Polynomials}.
\newblock Oxford University Press, 1979.

\bibitem{Macdonald1972}
I.G. Macdonald.
\newblock The poincaré series of a coxeter group.
\newblock {\em Mathematische Annalen}, 199:161--174, 1972.

\bibitem{UGSS}
J.~McCleary.
\newblock {\em A user's guide to spectral sequences}, volume~58 of {\em
  Cambridge Studies in Advanced Mathematics}.
\newblock Cambridge University Press, Cambridge, second edition, 2001.

\bibitem{StringTop2006}
L.~{Menichi} and G.~{Gaudens}.
\newblock {String topology for spheres}.
\newblock {\em ArXiv Mathematics e-prints}, (arXiv:math/0609304), September
  2006.

\bibitem{TLI&II}
M.~Mimura and H.~Toda.
\newblock {\em Topology of {L}ie groups. {I}, {II}}, volume~91 of {\em
  Translations of Mathematical Monographs}.
\newblock American Mathematical Society, Providence, RI, 1991.
\newblock Translated from the 1978 Japanese edition by the authors.

\bibitem{NakagawaE7}
M.~Nakagawa.
\newblock The integral cohomology ring of {$E_7/T$}.
\newblock {\em Journal of Mathematics of Kyoto University}, 41(2):303--321,
  2001.

\bibitem{nakagawaE8}
M.~Nakagawa.
\newblock The integral cohomology ring of {$E_8/T$}.
\newblock {\em Proc. Japan Acad. Ser. A Math. Sci.}, 86(3):64--68, 03 2010.

\bibitem{ClosedGeodesicServay}
A.~{Oancea}.
\newblock {Morse theory, closed geodesics, and the homology of free loop
  spaces}.
\newblock {\em ArXiv e-prints}, (arXiv:1406.3107), June 2014.

\bibitem{Onishchik}
A.~Onishchik.
\newblock {\em Topologiya tranzitivnykh grupp preobrazovanii [Topology of
  transitive transformation groups]}.
\newblock Fizmatlit Nauka, Moscow, 1995.

\bibitem{PITTIE}
H.~Pittie.
\newblock The integral homology and cohomology rings of so(n) and spin(n).
\newblock {\em Journal of Pure and Applied Algebra}, 73(2):105 -- 153, 1991.

\bibitem{loop_homology_spectral_sequence}
{R. Cohen, J. Jones and J. Yan}.
\newblock The loop homology algebra of spheres and projective spaces.
\newblock In {\em Categorical decomposition techniques in algebraic topology
  ({I}sle of {S}kye, 2001)}, volume 215 of {\em Progr. Math.}, pages 77--92.
  Birkh\"auser, Basel, 2004.

\bibitem{MR2251006}
{R. Cohen, K. Hess and A. Voronov}.
\newblock {\em String topology and cyclic homology}.
\newblock Advanced Courses in Mathematics. CRM Barcelona. Birkh\"auser Verlag,
  Basel, 2006.
\newblock Lectures from the Summer School held in Almer{\'{\i}}a, September
  16--20, 2003.

\bibitem{RealCoHomogeneous}
P.~K. Rashevskii.
\newblock The real cohomology of homogeneous spaces.
\newblock {\em Russian Mathematical Surveys}, 24(3):23, 1969.

\bibitem{rotman}
J.~Rotman.
\newblock {\em An introduction to homological algebra}.
\newblock Universitext. Springer, New York, second edition, 2009.

\bibitem{cohololgy_Lprojective}
N.~Seeliger.
\newblock Addendum to: {O}n the cohomology of the free loop space of a complex
  projective space.
\newblock {\em Topology Appl.}, 156(4):847, 2009.

\bibitem{CohomologyOmega(G/U)}
L.~Smith.
\newblock Cohomology of {$\Omega(G/U)$}.
\newblock {\em Proc. Amer. Math. Soc {\bf 19}}, pages 399--404, 1968.

\bibitem{MultisetPage}
R.~Stafford, April 2017.
\newblock
  https://uk.mathworks.com/matlabcentral/answers/143103-what-is-the-fastest-way-to-list-all-positive-integer-vectors-whose-sum-of-elements-equals-k.

\bibitem{ECstanly}
R.~P. Stanley.
\newblock {\em Enumerative Combinatorics Volume 2}.
\newblock Cambridge University Press, 1999.

\bibitem{FiniteReflectionGroups}
R.~Steinberg.
\newblock Finite reflection groups.
\newblock {\em Transactions of the American Mathematical Society},
  91(3):493--504, 1959.

\bibitem{toda1975}
H.~Toda.
\newblock On the cohomology ring of some homogeneous spaces.
\newblock {\em J. Math. Kyoto Univ.}, 15(1):185--199, 1975.

\bibitem{toda1974}
H.~Toda and T.~Watanabe.
\newblock The integral cohomology rings of {$F_4/T$} and {$E_6/T$}.
\newblock {\em J. Math. Kyoto Univ.}, 14(2):257--286, 1974.

\bibitem{Exceptional}
I.~{Yokota}.
\newblock {Exceptional Lie groups}.
\newblock {\em ArXiv e-prints}, (arXiv:0902.0431), February 2009.

\end{thebibliography}

\end{document}